\documentclass[brochure,english,12pt]{smfbourbaki}

\usepackage[T1]{fontenc}
\usepackage{lmodern,amssymb,bm}

\usepackage{babel}

\usepackage[utf8]{inputenc}

\usepackage[colorlinks=true, linkcolor=blue, citecolor=red,
urlcolor=blue]{hyperref}

\usepackage[stretch=10,shrink=10,step=2,kerning=true,protrusion=true,expansion=true,final]{microtype} 

\usepackage[
backend=biber,
style=authoryear, 
citestyle=authoryear-comp,
maxnames=7,
sortcites=false 
]{biblatex}
\usepackage{csquotes}

\usepackage{enumerate, graphicx, subcaption, thm-restate}

\newtheorem{problem}[defi]{Problem}
\newtheorem*{induction}{Induction hypothesis}
\newtheorem*{notation}{Notation}

\newcommand{\N}{\mathbb{N}}
\newcommand{\R}{\mathbb{R}}
\newcommand{\Z}{\mathbb{Z}}

\newcommand{\bbS}{\mathbf{S}}
\newcommand{\T}{\mathbf{T}}
\newcommand{\bW}{\mathbf{W}}

\newcommand{\cA}{\mathcal{A}}
\newcommand{\cB}{\mathcal{B}}
\newcommand{\cI}{\mathcal{I}}
\newcommand{\cL}{\mathcal{L}}
\newcommand{\cP}{\mathcal{P}}
\newcommand{\cQ}{\mathcal{Q}}
\newcommand{\cS}{\mathcal{S}}
\newcommand{\cT}{\mathcal{T}}

\newcommand{\ud}{\mathrm{d}}

\newcommand{\bC}{\mathbf{C}_{\varepsilon}}

\def\inn#1#2{\langle#1,#2\rangle}

\newcommand{\rad}{\mathrm{rad}}
\newcommand{\supp}{\mathrm{supp}\,}

\newcommand{\dist}{\mathrm{dist}}

\DefineBibliographyExtras{french}{\restorecommand\mkbibnamefamily}

\DeclareDelimFormat{nameyeardelim}{\addcomma\space}

\DeclareNameAlias{sortname}{given-family}

\renewcommand{\bibnamedash}{\leavevmode\raise3pt\hbox to3em{\hrulefill}\space}

\AtEveryBibitem{%
  \clearfield{issn} 
  \clearfield{isbn} 
  \clearfield{doi} 
  \clearlist{language} 
  \ifentrytype{online}{}{
    \clearfield{url}
  }
}

\renewbibmacro{in:}{%
    \ifentrytype{article}{}{\printtext{\bibstring{in}\intitlepunct}}}

\DeclareFieldFormat[article,periodical,inreference]{number}{\mkbibparens{#1}}
\DeclareFieldFormat[article,periodical,inreference]{volume}{\mkbibbold{#1}}
\renewbibmacro*{volume+number+eid}{%
    \printfield{volume}%
    \setunit*{\addthinspace}
    \printfield{number}%
    \setunit{\addcomma\space}%
    \printfield{eid}}

\DeclareFieldFormat[article,inbook,incollection]{title}{\enquote{#1}\addcomma} 

\date{Avril 2023}
\bbkannee{75\textsuperscript{e} année, 2022--2023}  
\bbknumero{1205}                                      

\title{Pointwise convergence for the Schr\"odinger equation}
\subtitle{After Xiumin Du and Ruixiang Zhang}

\author{Jonathan Hickman}
\address{School of Mathematics, \\ James Clerk Maxwell Building, \\ The King's Buildings,\\ Peter Guthrie Tait Road,\\ Edinburgh,\\ EH9 3FD, UK.}

\email{jonathan.hickman@ed.ac.uk}

\begin{document}

\maketitle




\section{Introduction: the Carleson problem}




\subsection{Solutions to the Schr\"odinger equation} Suitably normalised, the free Schr\"odinger equation on $\R^n$ is the second order partial differential equation
\begin{equation}\label{eq: Schrodinger}
i u_t - \Delta_x u = 0.
\end{equation}
Here $u$ is a complex-valued function of the space-time variables $(x,t) \in \R^n \times \R$, whilst $u_t$ and $\Delta_x u$ denote the first order time derivative and spatial Laplacian, respectively. We are interested in the Cauchy problem for this equation, whereby we specify an initial datum $f$ and wish to solve
\begin{equation}\label{eq: Cauchy}
\left\{
\begin{array}{l}
i u_t - \Delta_x u = 0, \\
u(x,0) = f(x)
\end{array}
\right. \qquad (x,t) \in \R^n \times \R. 
\end{equation}

Depending on our hypotheses on $f$, what it means for $u$ to be a `solution' to the equation \eqref{eq: Cauchy} varies. Here we consider two examples: \medskip

\noindent \textit{Classical solution}. If $f$ is sufficiently regular, then elementary Fourier transform methods show that \eqref{eq: Cauchy} has a unique solution in the classical sense.\footnote{In particular, the derivatives $u_t$ and $\Delta_x u$ are all well-defined in the usual sense from calculus, and the identities in \eqref{eq: Cauchy} hold pointwise.} For instance, if we assume $f \in \cS(\R^n)$, the Schwartz space, then the unique solution is given by 
\begin{equation*}
   u(x,t) := e^{it \Delta}f(x) 
\end{equation*}
where $e^{it\Delta}$ is the \textit{Schr\"odinger propagator}
\begin{equation}\label{eq: propagator} 
    e^{it \Delta}f(x) := \frac{1}{(2\pi)^n} \int_{\widehat{\R}^n} e^{i(x \cdot \xi + t|\xi|^2)} \hat{f}(\xi)\,\ud \xi. 
\end{equation}
Note that the regularity -- or smoothness -- of the initial datum $f$ is crucial to these observations. Indeed, the smoothness of $f$ directly translates into the decay of the Fourier transform $\hat{f}(\xi)$ as $|\xi| \to \infty$. This decay ensures the integral in \eqref{eq: propagator} is well-defined and also allows one to pass the derivatives inside the integral in order to verify~\eqref{eq: Schrodinger}.\medskip

\noindent \textit{$L^2$ solution}. Now suppose $f \in L^2(\R^n)$, without any additional regularity assumptions. In this case, Plancherel's theorem allows us to define the Fourier transform $\hat{f}$ as a function in $L^2(\widehat{\R}^n)$, but in general we cannot conclude that $\hat{f}$ is integrable. Consequently, the integral formula \eqref{eq: propagator} is not well-defined in the classical sense. 

To circumvent these issues, we further appeal to the $L^2$ theory of Fourier transform. Note that the propagator $e^{it\Delta}$ introduced above can be interpreted as a linear operator on $\cS(\R^n)$ which, given an initial datum $f$, outputs the solution at time $t$. Using Plancherel's theorem, we can extend $e^{it\Delta}$ to a Fourier multiplier operator acting on the whole of $L^2(\R^n)$. In particular, we define
\begin{equation*}
    e^{i t \Delta}f := \mathcal{F}^{-1} \big( e^{it|\,\cdot\,|^2} \cdot \mathcal{F} f \big) \qquad \textrm{for $f \in L^2(\R^n)$,}
\end{equation*}
where here $\mathcal{F}$ denotes the Fourier transform acting on $L^2(\R^n)$. Furthermore, this operator is an isometry of the $L^2$ space, in the sense that 
\begin{equation}\label{eq: conserved energy}
    \|e^{it \Delta}f\|_{L^2(\R^n)} = \|f\|_{L^2(\R^n)} \qquad \textrm{for all $f \in L^2(\R^n)$ and all $t \in \R$;}
\end{equation}
this identity is typically referred to as \textit{conservation of energy}.

As before, we may define
\begin{equation*}
   u(x,t) := e^{it \Delta}f(x),
\end{equation*}
 but in general this is no longer a classical solution to the Schr\"odinger
 equation: for instance, for a fixed time~$t$, the best we can say about
 $u(\,\cdot\,,t)$ is that it belongs to $L^2(\R^n)$ and so the Laplacian
 $\Delta_x u$ is not defined in the classical sense. However, we can
 interpret~$u$ as a solution to \eqref{eq: Schrodinger} in the sense of
 distributions. Indeed, using \eqref{eq: conserved energy} it is not difficult
 to show $u$~defines a distribution in $\cS'(\R^{n+1})$ and so $\partial_t u$ and $\Delta_x u$ can be understood in the distributional sense. Furthermore, a simple Fourier analytic argument shows $\inn{i\partial_t u - \Delta_x u}{\phi} = 0$ for all test functions $\phi \in \cS(\R^{n+1})$.
 
 \medskip




\subsection{The Carleson problem} 

Once a solution $u$ to \eqref{eq: Cauchy} has been constructed, it is natural to investigate the behaviour of $u$ and how it relates to the initial datum $f$. There is a huge variety of different questions one can ask in this direction. Here we are interested in the classical \textit{Carleson problem}, which aims to understand whether the initial datum can be recovered as a pointwise limit of the solution.\smallskip

First consider the case where $f \in \cS(\R^n)$, so that the solution $u(x,t) := e^{it\Delta}f(x)$ is classically defined. By definition, we know the solution $u$ satisfies $u(x,0) = f(x)$ and is differentiable, and therefore continuous, with respect to $t$. In particular,
\begin{equation}\label{eq: pointwise convergence}
    \lim_{t \to 0_+} e^{it\Delta}f(x) = f(x) \qquad \textrm{for all $x \in \R^n$.}
\end{equation}
The Carleson problem asks to what extent this elementary limit identity continues to hold when we consider more general $L^2$ solutions to the Schr\"odinger equation. \smallskip

Since an $L^2$ function is only defined almost everywhere, in order to make sense of the problem for  general initial data in $L^2(\R^n)$ it is necessary to weaken the requirement that convergence holds for \textit{all} $x \in \R^n$ in \eqref{eq: pointwise convergence} to \textit{almost all} $x \in \R^n$. That is, given $f \in L^2(\R^n)$ we wish to determine whether
\begin{equation}\label{eq: ae convergence}
    \lim_{t \to 0_+} e^{it\Delta}f(x) = f(x) \qquad \textrm{for almost every $x \in \R^n$.}
\end{equation}
Nevertheless, it is still unclear how to precisely interpret the above limit, since for every time slice $t$ (belonging to the \textit{continuum} $[0,1]$, say) we have a choice of representation for $e^{it\Delta}f$. We shall gloss over these technicalities for now and return to them in \S\ref{sec: maximal estimates} below.

It is not difficult to show that the limit holds in the $L^2$-sense: that is, given $f \in L^2(\R^n)$ we have
\begin{equation}\label{eq: norm convergence}
\lim_{t \to 0_+} \|e^{it\Delta}f - f\|_{L^2(\R^n)} = 0.
\end{equation}
Indeed, this can be easily verified for $f \in \cS(\R^n)$ using the integral formula \eqref{eq: propagator} for the propagator and the dominated convergence theorem. One can then pass to general $f \in L^2(\R^n)$ via density, using the conservation of energy identity \eqref{eq: conserved energy}. \smallskip

On the other hand, there are examples of $f \in L^2(\R^n)$ for which \eqref{eq: ae convergence} in fact \textbf{fails} (see \S\ref{sec: key results} below). Thus, we are interested in determining an additional hypothesis on $f$ under which the above norm convergence \eqref{eq: norm convergence} can be `upgraded' to almost everywhere convergence. Contrasting the situation for $f \in \cS(\R^n)$ with that for general $f \in L^2(\R^n)$, it is natural that the additional hypothesis should enforce some degree of regularity on the initial datum.\smallskip

The above considerations lead us to consider the Sobolev spaces $H^s(\R^n)$. Roughly speaking, $H^s(\R^n)$ consists of all $f \in L^2(\R^n)$ with derivatives up to order $s$ lying also in $L^2(\R^n)$. More precisely,
\begin{equation*}
  H^s(\R^n) :=  \big\{ f \in L^2(\R^n): (1 - \Delta_x)^{s/2}f \in L^2(\R^n) \big\}, \qquad s \geq 0, 
\end{equation*}
where $(1 - \Delta)^{s/2}$ denotes the fractional differential operator, defined in terms of the Fourier transform $\mathcal{F}$ now acting on the space of distributions $\mathcal{S}'(\R^n)$ by
\begin{equation*}
   (1 - \Delta_x)^{s/2} f := \mathcal{F}^{-1} \big((1 + |\,\cdot\,|^2)^{s/2} \cdot \mathcal{F}f \big). 
\end{equation*}
In particular, given $f \in L^2(\R^n)$, we can always make sense of the fractional derivative $(1 - \Delta_x)^{s/2} f$ as a distribution, and $f \in H^s(\R^n)$ if this distribution coincides with an $L^2$ function. It is clear from the definitions that
\begin{equation*}
    H^0(\R^n) = L^2(\R^n) \quad \textrm{and} \quad H^{s_1}(\R^n) \supseteq H^{s_2}(\R^n) \quad \textrm{for $0 \leq s_1 \leq s_2$.} 
\end{equation*}

Sobolev spaces provide a natural framework in which to formalise the Carleson problem.  

\begin{problem}[\cite{C1980}]\label{prob: Carleson} Determine the values of $s \geq 0$ such that 
\begin{equation}\label{eq: Carleson}
   \textrm{if $f \in H^s(\R^n)$,} \quad \textrm{then} \quad  \lim_{t \to 0} e^{it\Delta}f(x) = f(x) \quad \textrm{for almost every $x \in \R^n$.}
\end{equation}
\end{problem}

That is, we wish to determine the minimal degree of regularity (measured in terms of the Sobolev space index $s$) for which almost everywhere convergence is guaranteed to hold.

Aside from its intrinsic appeal, Problem~\ref{prob: Carleson} is intimately related to important questions regarding the distribution of the solution $e^{it\Delta}f(x)$ in space-time. Pointwise convergence is typically proved via analysis of the \textit{Schr\"odinger maximal operator}, an object of interest in its own right. The maximal operator can in turn be studied using \textit{fractal energy estimates} for the Schr\"odinger solutions. We introduce these concepts in \S\ref{sec: maximal estimates} and \S\ref{sec: fractal} below. Through these connections, progress on Problem~\ref{prob: Carleson} has led to new developments on a surprising array of different problems, such as the Falconer distance problem (see, for instance, \cite{GIOW2020, DZ2019}) and the Fourier restriction conjecture (see \cite{WW}). 




\subsection{A resolution of the Carleson problem: introducing the key results}\label{sec: key results}

Problem~\ref{prob: Carleson} has a rich history, paralleling many important developments in harmonic analysis over the last 40 years. We do not intend to give a complete survey of the relevant literature, but focus on definitive results and recent highlights. 

Whilst the $n=1$ case of Problem~\ref{prob: Carleson} was fully understood by the early 1980s through the works of \textcite{C1980,DK1982}, in higher dimensions the situation is much more nuanced. Nevertheless, a recent series of dramatic developments brought about an almost complete resolution. 

\subsubsection*{Necessary conditions} Problem~\ref{prob: Carleson} splits into two parts: finding necessary conditions for the index $s$ for \eqref{eq: Carleson} to hold and finding sufficient conditions. Both parts are difficult. The recent spate of activity on the Carleson problem was initiated by the surprising discovery of a new necessary condition on $s$. 

\begin{theo}[\cite{B2016}]\label{thm: Bourgain counterexample} For all $s < \frac{n}{2(n+1)}$, there exists some $f \in H^s(\R^n)$ such that \eqref{eq: ae convergence} fails.  
\end{theo}

Theorem~\ref{thm: Bourgain counterexample} relies on the construction of an explicit\footnote{Strictly speaking, the proof of Theorem~\ref{thm: Bourgain counterexample} proceeds by constructing a counterexample to the $H^s(\R^n) \to L^1(\R^n)$ boundedness of the \textit{Schr\"odinger maximal operator}. This in turn implies the existence of a counterexample to \eqref{eq: Carleson} through a variant of Stein's maximal principle. See \S\ref{sec: maximal estimates}.} initial datum $f$; the proof is intricate, involving number theoretic considerations. Prior to \textcite{B2016}, weaker necessary conditions were established in \textcite{DK1982, B2013,LR2017}.

We shall not discuss the proof of Theorem~\ref{thm: Bourgain counterexample} here, but instead refer the reader to the detailed exposition in \textcite{P2020}. An alternative argument, based on ergodic arguments rather than number theory, can also be found in \textcite{LR2019}. 

\subsubsection*{Sufficient conditions} We now turn to positive results, which form the focus of this article. In the wake of Bourgain's counterexample, there was a flurry of activity on the Carleson problem. In a major advance, the $n=2$ case was completely settled through work of \textcite{DGL2017}. The higher dimensional case later followed in a landmark paper of \textcite{DZ2019}. 

\begin{theo}[\cite{DZ2019}\footnote{The $n=1$ and $n=2$ cases of Theorem~\ref{thm: DZ pointwise} were established earlier in \textcite{C1980,DGL2017}, respectively.}]\label{thm: DZ pointwise} If $f \in H^s(\R^n)$ for some $s > \frac{n}{2(n+1)}$, then
\begin{equation*}
    \lim_{t \to 0_+} e^{it\Delta}f(x) = f(x) \qquad \textrm{holds for almost every $x \in \R^n$.}
\end{equation*}
\end{theo}

\begin{rema}\label{rmk: technical annoyance}
    As previously noted, there are measure-theoretic technicalities regarding the meaning of the above statement, since the limit is taken over a continuum. We address this in \S\ref{sec: maximal estimates} below.
\end{rema}

Together, Theorem~\ref{thm: Bourgain counterexample} and Theorem~\ref{thm: DZ pointwise} give an almost complete\footnote{That is, except for the question of behaviour at the endpoint exponent $s = n/(2(n+1))$.} answer to the Carleson problem and constitute a major milestone in harmonic analysis and PDE. Furthermore, Theorem~\ref{thm: DZ pointwise} is in fact a special case of a significantly more general result proved in \textcite{DZ2019}, which has a variety of additional applications: see \S\ref{sec: fractal} below.

The proof of Theorem~\ref{thm: DZ pointwise} builds on many important developments in harmonic analysis and previous works on the Carleson problem in particular. For the purpose of this article, we shall roughly divide the recent history of the problem into two epochs.

\begin{enumerate}
   \item \textbf{Multilinear theory / broad-narrow analysis}. A key development in modern harmonic analysis was the introduction of multilinear Strichartz estimates for the Schr\"odinger equation in \textcite{BCT2006} (see Theorem~\ref{thm: BCT} below). These estimates were applied to the study of (linear) oscillatory integral operators by \textcite{BG2011}, where an important mechanism was introduced for estimating linear operators via their multilinear counterparts. The technique of \textcite{BG2011} has since become known as \textit{broad-narrow analysis}; the relevant ideas are discussed in detail in \S\ref{sec: Broad-narrow} below. The multilinear technology was applied to the study of Problem~\ref{prob: Carleson} in \textcite{B2013}, where Theorem~\ref{thm: DZ pointwise} was shown to hold for the more restrictive range $s > \tfrac{2n-1}{4n}$ (this result was previously established for $n=2$ in \textcite{L2006} using bilinear methods). 

    \item \textbf{Refined Strichartz estimates.} A landmark paper of \textcite{DGL2017} established the $n=2$ case of Theorem~\ref{thm: DZ pointwise}. Their work relied on important advances in harmonic analysis such as the $\ell^2$ decoupling theorem of \textcite{BD2015} (see Theorem~\ref{thm: decoupling} below) and polynomial partitioning techniques introduced in \textcite{GK2015,G2016}. Moreover, \textcite{DGL2017} introduced \textit{refined Strichartz estimates}, which were later developed to attack the Carleson problem in higher dimensions in \textcite{DGLZ2018} and have been found to have an array of additional applications (see, for instance, \cite{GIOW2020, WW}).
\end{enumerate}

The argument of \textcite{DZ2019} incorporates many of the tools and ideas mentioned above: in particular, the broad-narrow analysis of \textcite{BG2011}; the multilinear Strichartz inequalities of \textcite{BCT2006} and the decoupling estimates of \textcite{BD2015}. We shall discuss these ingredients in detail in \S\ref{sec: Broad-narrow} below. On the other hand, the methods of \textcite{DZ2019} are in many respects quite different from those used in \textcite{DGL2017} to settle the $n=2$ case. Here no polynomial partitioning is used and the refined Strichartz estimates are not necessary to the argument.\footnote{In \textcite{DZ2019} multilinear refined Strichartz estimates from \textcite{DGLZ2018} are applied but, as noted in the paper, the more elementary multilinear Strichartz estimates of \textcite{BCT2006} suffice for the proof of Theorem~\ref{thm: DZ pointwise}.} Nevertheless, the main novel ingredient in \textcite{DZ2019} is an ingenious induction-on-scale method which has its roots in the proof of the refined Strichartz estimates from \textcite{DGL2017, DGLZ2018}. We shall discuss these techniques in \S\ref{sec: main proof} below.  
    



\subsection{About this article}

What follows is an exposition of the proof of Theorem~\ref{thm: DZ pointwise}, following the methods of \textcite{DZ2019}. As described above, the proof combines sophisticated modern machinery from harmonic analysis and, in particular, the multilinear Strichartz estimates of \textcite{BCT2006} and the $\ell^2$ decoupling theory of \textcite{BD2015}. We shall introduce these two key ingredients in Theorem~\ref{thm: BCT} and Theorem~\ref{thm: decoupling} below, but we shall not provide proofs. The rest of the article is self-contained. Equally important to the argument are a variety of elementary guiding principles, rooted in Fourier analysis, which govern the behaviour of solutions to the Schr\"odinger equation. We shall spend some time in \S\ref{sec: Basic tools} discussing these principles, and as such this article could serve as an accessible introduction to this highly active area of harmonic analysis and PDE.




\subsection{Acknowledgement} The author wishes to thank Marco Vitturi, Bernat Ramis Vich and an anonymous referee for many tremendously helpful comments which improved the exposition. He also wishes to thank Anthony Carbery and Andreas Seeger for some interesting conversations concerning aspects of the theory.




\section{Notation}

Throughout this article, we work either in the \textit{space-time domain} $\R^{n+1}$ or \textit{spatial domain} $\R^n$. The latter is endowed with the product metric
\begin{equation}\label{eq: space-time metric}
    |z - \bar{z}| = \max\{|x - \bar{x}|, |t-\bar{t}|\} \qquad \textrm{for $z = (x,t), \bar{z} = (\bar{x}, \bar{t}) \in \R^{n+1}$},  
\end{equation}
where the norms $|\,\cdot\,|$ appearing on the right-hand side are the usual Euclidean norms on $\R^d$ for $d = n$ and $d=1$, respectively. The \textit{spatial domain}, on the other hand, is endowed with the usual Euclidean metric. We write $B^{n+1}(z, R)$ for the space-time ball centred at $z \in \R^{n+1}$ of radius $R$, defined with respect to \eqref{eq: space-time metric}, and $B^n(x,R)$ for the usual Euclidean ball centred at $x \in \R$ of radius $R$. This gives rise to a slight ambiguity in the notation, but the meaning should always be clear from the context. In some cases we will write $B^{n+1}_R$ or $B^n_R$ for $B^{n+1}(0,R)$ and $B^n(0,R)$, respectively. 

Given functions $f \in L^1(\R^n)$ and $g \in L^1(\widehat{\R}^n)$, we define the Fourier transform of $f$ and the inverse Fourier transform of $g$ by the integral formul\ae 
\begin{equation*}
    \hat{f}(\xi) := \int_{\R^n} e^{-i x \cdot \xi} f(x)\,\ud x \quad \textrm{and} \quad \check{g}(x) := \frac{1}{(2 \pi)^n} \int_{\widehat{\R}^n} e^{i x \cdot \xi} g(\xi)\,\ud \xi.
\end{equation*}
Note that we distinguish between the spatial domain $\R^n$ and the \textit{frequency domain} $\widehat{\R}^n$; this is simply for notational purposes and $\widehat{\R}^n$ can be thought of as a copy of $\R^n$.

Given a set $E \subseteq \R^d$, we let $\chi_E \colon \R^d \to \R$ denote its characteristic function, so that $\chi_E(x) = 1$ if $x \in E$ and $\chi_E(x) = 0$ otherwise. If $E$ is Lebesgue measurable, we let $|E|$ denote its Lebesgue measure.

An \textit{$r$-cube} (or simply a \textit{cube}) $Q \subset \R^n$ is a set of the form
\begin{equation*}
    Q := x + [-r/2, r/2]^n \qquad \textrm{for some $x \in \R^n$ and $r > 0$};
\end{equation*}
in this case $x$ is referred to as the \textit{centre} of the cube and $r$ the \textit{side-length}. Note that, for the purposes of this article, all cubes have faces parallel to the coordinate axes. We say $Q \subseteq \R^{n+1}$ is a \textit{lattice $r$-cube} for some $r > 0$ if it is an $r$-cube centred at a point on the integer lattice $r\Z^{n+1}$. In the case $r = 1$, we will also refer to $Q$ as a lattice unit-cube. Given a cube $Q$ and $M > 0$, we let $M \cdot Q$ denote the cube concentric to $Q$ but side-length scaled by a factor of~$M$.

Given a list of objects $L$ and real numbers $A$, $B \geq 0$, we write $A \lesssim_L B$ or $B \gtrsim_L A$ to indicate $A \leq C_L B$ for some constant $C_L$ which depends only items in the list $L$ and perhaps other admissible parameters such as the dimension $n$ or Lebesgue exponents $p$. We write $A \sim_L B$ to indicate $A \lesssim_L B$ and $B \lesssim_L A$.




\section{Standard reductions and reformulations}\label{sec: Standard reductions}

In this section we perform a series of arguments to reduce Theorem~\ref{thm: DZ pointwise} to the \textit{fractal energy estimate} stated in Theorem~\ref{thm: fractal L2} below. Whilst these arguments are standard (use of maximal estimates, linearisation, discretisation), the final form of the fractal energy estimate in Theorem~\ref{thm: fractal L2} is an interesting aspect of the approach.




\subsection{Elementary symmetries}
For fixed time $t \in \R$, the operator $e^{it \Delta}$ is a Fourier multiplier, and therefore automatically enjoys a number of special symmetries.\footnote{We will discuss additional special symmetries particular to $e^{it\Delta}$ in \S\ref{sec: Parabolic geometry} below.} In particular, $e^{it \Delta}$ commutes with any other multiplier and, therefore, with (spatial) translations. We call this property \textit{spatial translation invariance}. With respect to dilations, the operators satisfy
\begin{equation*}
   e^{it \Delta} (\delta_{R}f) =  \delta_{R}(e^{iR^2t \Delta} f) \qquad \textrm{where} \qquad \delta_R f (x) := f(Rx),
\end{equation*}
and so are dilation invariant up to a scaling of the temporal parameter. Finally, the operators $e^{it\Delta}$ form a semigroup, which leads to temporal translation invariance properties, at least at the level of $L^2$ norms. 




\subsection{Schr\"odinger maximal estimates}\label{sec: maximal estimates} 

The first step is to apply a standard argument to reduce the pointwise convergence problem to proving an \textit{a priori} bound for a maximal operator.  

\begin{theo}[\cite{DZ2019}]\label{thm: DZ max} For all $s > \frac{n}{2(n+1)}$, we have a maximal estimate
\begin{equation}\label{eq: max est}
\|\sup_{0 \leq t \leq 1}|e^{it\Delta}f| \|_{L^2(B^n(0,1))} \lesssim \|f\|_{H^s(\R^n)} \qquad \textrm{for all $f \in \cS(\R^n)$.}
\end{equation}
\end{theo}

Here, given $f \in H^s(\R^n)$, the \textit{$H^s$-norm} is defined by
\begin{equation*}
    \|f\|_{H^s(\R^n)} := \|(1 - \Delta_x)^{s/2} f\|_{L^2(\R^n)}.
\end{equation*}
Note that we only ask for \eqref{eq: max est} to hold for $f \in \cS(\R^n)$; in this case the solution $u(x,t) := e^{it\Delta}f$ is continuous in $(x,t)$ and so the expression appearing on the left-hand side of \eqref{eq: max est} is certainly well-defined. We refer to the operation $f \mapsto \sup_{0 \leq t \leq 1}|e^{it\Delta}f|$ as the \textit{Schr\"odinger maximal operator}.

By a standard argument, Theorem~\ref{thm: DZ max} implies Theorem~\ref{thm: DZ pointwise}. For completeness, the details of this implication are given below. We remark that the choice of $L^2$ space here is likely sub-optimal and \eqref{eq: max est} may hold with the left-hand $L^2$-norm replaced with an $L^p$-norm for larger $p$. Indeed, for $n=1$ it was shown in \textcite{KPV1991} that one may take $p = 4$ and for $n=2$ it was shown in \textcite{DGLZ2018} that one may take $p=3$. These exponents are sharp. In general, an obvious necessary condition is $p \leq 2 \cdot \frac{n+1}{n}$ which arises through Sobolev embedding. Indeed, since we know $\lim_{t \to 0} e^{it\Delta}f$ for $f \in \cS(\R^n)$, we can only hope to have an $H^s(\R^n) \to L^p(\R^n)$ maximal estimate if $H^s(\R^n)$ embeds into $L^p(\R^n)$. Surprisingly, a more restrictive necessary condition on $p$ was found in \textcite{DKWZ2020} by adapting Bourgain's counterexample from \textcite{B2016}. 

Theorem~\ref{thm: DZ max} has the following consequence, which formalises the meaning of the limit identity in Theorem~\ref{thm: DZ pointwise} (\textit{c.f.} Remark~\ref{rmk: technical annoyance} and the discussion following \eqref{eq: ae convergence}).  

\begin{coro}\label{cor: technical measure theory} Let $s > \frac{n}{2(n+1)}$ and $f \in H^s(\R^n)$. Then for each $0 \leq t \leq 1$ there is a choice of representative of the $L^2$-function $e^{it\Delta}f$ such that for almost every $x \in \R^n$ the map $t \mapsto e^{it\Delta}f(x)$ is continuous. 
\end{coro}

\begin{proof} By translation invariance, it suffices to consider only $x \in B^n(0,1)$. Fixing $f$ as in the statement of the corollary, for each $k \in \N$ we can find some $f_k \in \cS(\R^n)$ such that $\|f - f_k\|_{H^s(\R^n)} \leq 2^{-2k}$. We claim that the set 
\begin{equation*}
    E := \big\{x \in B^n(0,1) : \limsup_{k \to \infty} 2^k \sup_{0 \leq t \leq 1} |e^{it\Delta}(f_k - f_{k-1})(x)| > 1 \big\}
\end{equation*}
is of Lebesgue measure zero. Temporarily assuming this is true, we see that the sequence $e^{it\Delta}f_k(x)$ of functions is Cauchy over $(x,t) \in B^n(0,1)\setminus E \times [0,1]$, uniformly in $t$. The sequence therefore converges pointwise in $x$ and uniformly in $t$ on this set to some limit function $L(x,t)$, which is continuous in $t$. On the other hand, for fixed $0 \leq t \leq 1$, by the conservation of energy identity we have
\begin{equation*}
\|e^{it\Delta}f_k - e^{it\Delta}f\|_{L^2(\R^n)} = \|f_k - f\|_{L^2(\R^n)} \leq\|f_k - f\|_{H^s(\R^n)} \leq 2^{-2k},
\end{equation*}
and so $e^{it\Delta}f_k$ converges to $e^{it\Delta}f$ in the $L^2(\R^n)$ sense. Since convergence in $L^2$ implies almost everywhere convergence along a subsequence, it follows that for each $0 \leq t \leq 1$, the function $L(x,t)$ must agree with $e^{it\Delta}f(x)$ for almost every $x \in B^n(0,1)$. This is precisely the desired conclusion. 

It remains to show $E$ is of Lebesgue measure zero, which is achieved using a Borel--Cantelli argument. Note that
\begin{equation*}
   E \subseteq \bigcap_{K \in \N} \bigcup_{k \geq K} E_k 
\end{equation*}
where 
\begin{equation*}
E_k := \big\{x \in B^n(0,1) :  \sup_{0 \leq t \leq 1} |e^{it\Delta}(f_k - f_{k-1})(x)| > 2^{-k}\big\}.
\end{equation*}
By Tchebyshev's inequality and the maximal estimate,
\begin{equation*}
  |E| \leq \sum_{k \geq K}  |E_k| \leq  \sum_{k \geq K}  2^{2k} \big\|\sup_{0 \leq t \leq 1} |e^{it\Delta}(f_k-f_{k-1})|\big\|_{L^2(\R^n)}^2 \lesssim_s   \sum_{k \geq K} 2^{-2k} \lesssim 2^{-2K}
\end{equation*}
for all $K \in \N$. Thus, $|E| = 0$, as required. 

\end{proof}

For $s > \frac{n}{2(n+1)}$ and $f \in H^s(\R^n)$, we can use Corollary~\ref{cor: technical measure theory} to clarify the meaning of our limits. We choose representatives of the functions $e^{it\Delta}f$ for $0 \leq t \leq  1$ to guarantee the conclusion of Corollary~\ref{cor: technical measure theory}. Since $e^{it\Delta}f(x)|_{t = 0} = f(x)$ for almost every $x \in \R^n$, with this choice of representatives we do indeed have
\begin{equation*}
    \lim_{t \to 0_+} e^{it\Delta}f(x) = f(x) \qquad \textrm{for almost every $x \in \R^n$.}
\end{equation*}
Thus, Corollary~\ref{cor: technical measure theory} can be interpreted as the precise formulation of Theorem~\ref{thm: DZ pointwise}.




\subsection{Littlewood--Paley decomposition}\label{sec:  Littlewood--Paley decomposition}
The next step is to recast the $H^s(\R^n) \to L^2(\R^n)$ maximal bound in Theorem~\ref{thm: DZ max} as an $L^2(\R^n) \to L^2(\R^n)$ estimate for frequency localised data. We are easily able to do this since Theorem~\ref{thm: DZ max} involves an open range of $s$. 

For $R \geq 1$, let $\chi_{A(R)}$ denotes the characteristic function of the frequency annulus
\begin{equation*}
   A(R) := \{ \xi \in \widehat{\R}^n : R/2 \leq |\xi| < R\}
\end{equation*}
and let $\chi_{A(R)}(D)$ denote the Fourier multiplier operator defined by
\begin{equation*}
\chi_{A(R)}(D)f := \mathcal{F}^{-1} \big( \chi_{A(R)} \cdot \mathcal{F}f\big) \qquad \textrm{for all $f \in L^2(\R^n)$.}
\end{equation*}
Thus, $\chi_{A(R)}(D)$ corresponds to a frequency projection to $A(R)$, with a rough cutoff. As a simple consequence of Plancherel's theorem,
\begin{equation}\label{eq: LP characterisation}
    \|f\|_{H^s(\R^n)} \sim \Big( \sum_{k = 1}^{\infty} 2^{2ks} \|\chi_{A(2^k)}(D)f\|_{L^2(\R^n)}^2 \Big)^{1/2} + \|f\|_{L^2(\R^n)};
\end{equation}
this is known as the \textit{Littlewood--Paley characterisation} of the $H^s$-norm. 

In view of the characterisation \eqref{eq: LP characterisation}, to prove Theorem~\ref{thm: DZ max} it suffices to show that for all $\varepsilon >  0$ and all $R \geq 1$, the inequality
\begin{equation}\label{eq: band limited max}
\|\sup_{0 \leq t \leq 1}|e^{it\Delta}f| \|_{L^2(B^n(0,1))} \lesssim_{\varepsilon} R^{n/(2(n+1)) + \varepsilon} \|f\|_{L^2(\R^n)} 
\end{equation}
holds for all $f \in L^2(\R^n)$ with $\supp \hat{f} \subseteq A(R)$. By scaling and exploiting certain pseudo-local properties of the propagator, the problem is further reduced to the following proposition. 

\begin{prop}\label{prop: reduced max} For all $\varepsilon >  0$ and all $R \geq 1$ the inequality
\begin{equation}\label{eq: reduced max}
\|\sup_{0 \leq t \leq R}|e^{it\Delta}f| \|_{L^2(B^n(0,R))} \lesssim_{\varepsilon} R^{n/(2(n+1)) + \varepsilon} \|f\|_{L^2(\R^n)} 
\end{equation}
holds whenever $f \in L^2(\R^n)$ satisfies $\supp \hat{f} \subseteq A(1)$. 
\end{prop}

Scaling alone shows that \eqref{eq: band limited max} is equivalent to a variant of \eqref{eq: reduced max} in which the supremum is taken over the longer time interval $0 < t < R^2$. To pass to the shorter time interval, we use an argument of \textcite{L2006}, based on pseudo-local properties of the operator; we describe the details in \S\ref{sec: wave packet dec}.




\subsection{Linearising the maximal operator}\label{sec: Linearisation}

To prove Proposition~\ref{prop: reduced max}, we linearise the maximal operator using the Kolmogorov--Seliverstov--Plessner method. In particular, it suffices to show, for all $\varepsilon > 0$, $R \geq 1$ and all measurable functions $\mathbf{t} \colon B^n(0,R) \to (0,R)$, the inequality
\begin{equation}\label{eq: lin max}
    \Big(\int_{B^n(0,R)} |e^{i\mathbf{t}(x)\Delta}f(x)|^2 \,\ud x \Big)^{1/2} \lesssim_{\varepsilon} R^{n/(2(n+1)) + \varepsilon} \|f\|_{L^2(\R^n)} 
\end{equation}
holds for all $f \in L^2(\R^n)$ with $\supp \hat{f} \subseteq B^n(0,1)$. 

Let $U$ denote the operator defined initially on the Schwartz class by
\begin{equation*}
    Uf(x,t) := \frac{1}{(2\pi)^n}\int_{B^n(0,1)} e^{i (\inn{x}{\xi} + t|\xi|^2)} \hat{f}(\xi) \,\ud \xi.
\end{equation*}
Thus, $Uf(\,\cdot\,,t)$ corresponds to the composition of $e^{it\Delta}f$ with a rough frequency projection to the unit ball. We think of the estimate \eqref{eq: lin max} as bounding $Uf$ over the space-time graph $\Gamma_{\mathbf{t}} := \{(x, \mathbf{t}(x)) : x \in B^n(0,R)\}$ of the measurable function $\mathbf{t}$. This perspective turns out to be useful and we shall formulate all our key estimates over the space-time domain. 

Since $Uf$ is localised in frequency, we do not expect the values $|Uf(z)|$ to vary greatly at small scales. This is a manifestation of the \textit{uncertainty principle}, which is discussed in detail in \S\ref{sec: uncertainty} below. These observations allow us to discretise our setup. 

\begin{defi} Let $\cQ$ be a family of lattice unit cubes in $\R^{n+1}$.
\begin{enumerate}[i)]
    \item We let $Z_{\cQ}$ denote the union $\bigcup_{Q \in \cQ} Q$.
    \item We say $\cQ$ satisfies the \emph{vertical line test} if for almost every $x \in \R^n$, at most one cube from $Q$ intersects the line $\{(x,t) : t \in \R\}$.
\end{enumerate}
    
\end{defi}

A lattice unit cube $Q$ should be thought of as a `discretised point'  and $Z_{\cQ}$ for a family $\cQ$ satisfying the vertical line test as a `discretised graph'. With these definitions, Proposition~\ref{prop: reduced max} is a consequence of the following bound.

\begin{prop}[\cite{DZ2019}]\label{prop: lin max op}  For all $\varepsilon >  0$ and all $R \geq 1$, the inequality
\begin{equation*}
    \|Uf\|_{L^2(Z_{\cQ})} \lesssim_{\varepsilon} R^{n/(2(n+1)) + \varepsilon} \|f\|_{L^2(\R^n)} 
\end{equation*}
holds whenever $f \in L^2(\R^n)$ and $\cQ$ is a collection of unit lattice cubes lying in $B^{n+1}(0,R)$ which satisfy the vertical line test. 
\end{prop}

We will show in detail why the discretised bound in Proposition~\ref{prop: lin max op} implies the linearised maximal estimate \eqref{eq: lin max} in \S\ref{sec: uncertainty} below. 




\subsection{Fractal energy estimates}\label{sec: fractal}

Proposition~\ref{prop: lin max op} is in fact a special case of a significantly more general theorem proved in \textcite{DZ2019}. Here we describe the general framework, which we adopt for the remainder of the exposition. 

\begin{defi} Let $M  > 0$ and $\cQ$ be a family of lattice $M$-cubes in $\R^{n+1}$. 
\begin{enumerate}[i)]
    \item For $1 \leq \alpha \leq n+1$ and a space-time ball $B \subseteq \R^{n+1}$, we define 
\begin{equation*}
    \Delta_{\alpha}(\cQ, B) := \frac{\#\{ Q \in \cQ : Q \subset B\}}{\rad(B)^{\alpha}},
\end{equation*}
where $\rad(B)$ denotes the radius of $B$. 
    \item Furthermore, let
\begin{equation*}
    \Delta_{\alpha}(\cQ) := \sup_B \Delta_{\alpha}(\cQ, B)
\end{equation*}
where the supremum is taken over all space-time balls $B = B^{n+1}(z,r) \subseteq \R^{n+1}$. 
\end{enumerate}
\end{defi} 

If $1 \leq M \leq R$ and $\cQ$ is a non-empty collection of lattice $M$-cubes contained in $B^{n+1}(0,R)$, then it follows from the definition that
\begin{equation*}
    M^{-\alpha} \lesssim \Delta_{\alpha}(\cQ) \qquad \textrm{and} \qquad \#\cQ \lesssim \Delta_{\alpha}(\cQ) R^{\alpha}.
\end{equation*}

We may now (finally) state the main result of \textcite{DZ2019}.\footnote{In fact, \textcite{DZ2019} prove a further strengthening of Theorem~\ref{thm: fractal L2} involving an additional parameter $\lambda$: see \textcite[Theorem 1.6]{DZ2019}. However, this strengthened result is unnecessary for typical applications and so here we stick to the simpler statement in Theorem~\ref{thm: fractal L2}.}

\begin{theo}[\cite{DZ2019}]\label{thm: fractal L2} 
For all $\varepsilon >  0$ and all $R \geq 1$, $1 \leq \alpha \leq n+1$, the inequality
\begin{equation}\label{eq: fractal L2}
\|Uf\|_{L^2(Z_{\cQ})} \lesssim_{\varepsilon} \Delta_{\alpha}(\cQ)^{1/(n+1)}R^{\alpha/(2(n+1)) + \varepsilon} \|f\|_{L^2(\R^n)} 
\end{equation}
holds whenever $f \in L^2(\R^n)$ and $\cQ$ is a family of lattice unit cubes in $B^{n+1}(0,R)$. 
\end{theo}

If $\cQ$ satisfies the vertical line test, then it is not difficult to see $\Delta_n(\cQ) \lesssim 1$. Consequently, Theorem~\ref{thm: fractal L2} implies Proposition~\ref{prop: lin max op} and therefore, by the preceding reductions, also pointwise convergence result in Theorem~\ref{thm: DZ pointwise}. 

We refer to the estimate \eqref{eq: fractal L2} as a \textit{fractal energy estimate}. This terminology is partly motivated by the conservation of energy identity \eqref{eq: conserved energy}. Indeed, from \eqref{eq: conserved energy} we have 
\begin{equation}\label{eq: space-time energy est}
    \|Uf\|_{L^2(Z_{\cQ})} \leq \Big(\int_{-R}^R \|e^{it\Delta}f\|_{L^2(\R^n)}^2\,\ud t\Big)^{1/2} \lesssim R^{1/2} \|f\|_{L^2(\R^n)}, 
\end{equation}
which directly implies the $\alpha = n+1$ case of Theorem~\ref{thm: fractal L2}. For general exponents $1 \leq \alpha \leq n+1$, it is useful to think of a family of cubes $\cQ$ satisfying $\Delta_{\alpha}(\cQ) \lesssim 1$ as a discretised version of an $\alpha$-dimensional `fractal' set. 

There are two main advantages in working with the general framework of fractal energy estimates:
\begin{enumerate}[1)]
\item Theorem~\ref{thm: fractal L2} has an array of additional applications beyond the pointwise convergence problem for the Schr\"odinger maximal function. These include estimates for the dimension of the divergence set in the Carleson problem and partial results towards the Falconer distance conjecture. We refer to \textcite[\S2]{DZ2019} for further details. 
\item The form of the estimate in \eqref{eq: fractal L2} is useful when it comes to the proof. In particular, the arguments involve an induction scheme and the inclusion of the $\Delta_{\alpha}(\cQ)$ factor allows greater leverage from the induction hypothesis.\footnote{ The induction argument is not applied to the statement in Theorem~\ref{thm: fractal L2} itself, but a variant described in Proposition~\ref{prop: final reduction}.}
\end{enumerate} 

For the remainder of this article we shall discuss the proof of Theorem~\ref{thm: fractal L2}. We start with some basic background in harmonic analysis and dispersive PDEs in \S\ref{sec: Basic tools}, before moving to more advanced topics in \S\S\ref{sec: mutlilinear harmonic analysis}--\ref{sec: main proof}.




\section{Basic tools from harmonic analysis}\label{sec: Basic tools}




\subsection{The uncertainty principle}\label{sec: uncertainty}

We begin with a discussion of the uncertainty principle for the Fourier transform. This is a set of heuristics which roughly state:
\begin{quote}
 If $\hat{F}$ is localised at scale $R^{-1}$, then $|F|$ should be locally constant at scale $R$.
\end{quote}
The following lemma provides a rigorous interpretation of this principle at the unit scale. 

\begin{lemm}[Locally constant property]\label{lem: loc const} There exists a continuous function $\eta \colon \R^d \to [0,\infty)$ satisfying the following:
\begin{enumerate}[i)]
    \item If $F \in \cS(\R^d)$ satisfies $\supp \hat{F} \subseteq Q_0 := [-1/2, 1/2]^d$, then
\begin{equation}\label{eq: loc const}
    |F(z)| \leq \sum_{Q \in \cQ_{\mathrm{all}}} a_Q\chi_Q(z) \lesssim |F| \ast \eta(z) \qquad \textrm{for all $z \in \R^d$}
\end{equation}
where here $\cQ_{\mathrm{all}}$ is the collection of all lattice unit cubes in $\R^d$ and 
\begin{equation*}
    a_Q := \sup_{z \in Q} |F(z)| \qquad \textrm{for all $Q \in \cQ_{\mathrm{all}}$.}
\end{equation*}
\item The function $\eta$ is $L^1$-normalised and rapidly decaying away from $Q_0$ in the sense that
\begin{equation*}
    \eta(z) \lesssim_N (1 + 2|z|_{\infty})^{-N} \qquad \textrm{for all $N \in \N$.}
\end{equation*}
\end{enumerate}
\end{lemm}

\begin{proof} Fix $\eta_0 \in \cS(\R^d)$ satisfying $\hat{\eta}_0(\xi) = 1$ for all $\xi \in [-1,1]^d$. Thus, if $F \in \cS(\R^d)$ satisfies the hypothesis of part i), we have the reproducing formula $F = F \ast \eta_0$. In particular, given $Q \in \cQ_{\mathrm{all}}$ and $z_Q \in Q$ chosen to satisfy $|F(z_Q)| = a_Q$, it follows that
\begin{equation*}
 a_Q = |F(z_Q)| \leq \int_{\R^d} |F(y)| |\eta_0(z_Q - y)| \,\ud y. 
\end{equation*}

Now define $\eta \colon \R^d \to [0,\infty)$ by 
\begin{equation*}
    \eta(z) := \sup_{|w - z|_{\infty} \leq 1} |\eta_0(w)|. 
\end{equation*}
The rapid decay of the Schwartz function $\eta_0$ then ensures $\eta$ is rapidly decaying away from $[-1,1]^d$. It therefore only remains to show that \eqref{eq: loc const} holds. 

If $z \in Q$ is an arbitrary element, then 
\begin{equation*}
    |(z_Q - y) - (z - y)|_{\infty} \leq 1 \qquad \textrm{for all $y \in \R^d$.}
\end{equation*}
Consequently, $|\eta_0(z_Q - y)| \leq \eta(z - y)$ for all $y \in \R^d$, and so
\begin{equation*}
    a_Q \leq |F| \ast \eta(z) \qquad \textrm{for all $z \in Q$,}
\end{equation*}
which immediately implies the second inequality in \eqref{eq: loc const}. On the other hand, the first inequality in \eqref{eq: loc const} is a trivial consequence of the definitions. 
\end{proof}

We may use the rigorous formulations of the uncertainty principle introduced above to justify the discretisation procedure described in \S\ref{sec: Linearisation}. 

\begin{proof}[Proof (Proposition~\ref{prop: lin max op} $\Rightarrow$ Proposition~\ref{prop: reduced max})] Assume Proposition~\ref{prop: lin max op} holds. Recall that it suffices to show the linearised maximal estimate \eqref{eq: lin max}. 

Let $\tilde{\chi} \in \cS(\R^{n+1})$ satisfy $|\tilde{\chi}(z)| \gtrsim 1$ for all $|z| \leq 2$ and $\supp \mathcal{F}\tilde{\chi} \subseteq B^{n+1}(0,1)$. Defining
\begin{equation*}
    F(z) := Uf(z) \cdot \tilde{\chi}(R^{-1} z),
\end{equation*}
then it follows that $|Uf(z)| \lesssim |F(z)|$ for all $z \in B^{n+1}(0,2R)$ and $\supp \hat{F} \subseteq B^{n+1}(0,2)$. The proof will follow from the resulting locally constant properties of the function $F$. 

Fix $\varepsilon > 0$ and a measurable function $\mathbf{t} \colon B^n(0,R) \to (0,1]$ and let $\mathfrak{q}_R$ denote the collection of lattice unit cubes in $\R^n$ which intersect $B^n(0,R)$. For each $q \in \mathfrak{q}_R$ there exists some choice of $x_q \in q$ such that
\begin{equation*}
    \sup_{x \in q} |e^{i\mathbf{t}(x)\Delta}f(x)| \leq 2 |e^{i\mathbf{t}(x_q)\Delta}f(x_q)|.
\end{equation*}
If we define $z_q := (x_q, \mathbf{t}(x_q))$ for each $q \in \mathfrak{q}_R$, then it follows that
\begin{equation*}
    \Big(\int_{B^n(0,R)} |e^{i\mathbf{t}(x)\Delta}f(x)|^2 \, \ud x \Big)^{1/2} \lesssim \Big(\sum_{q \in q_R} |F(z_q)|^2 \Big)^{1/2}.
\end{equation*}

For each $q \in \mathfrak{q}_R$ let $I_q \subseteq \R$ denote a choice of lattice unit interval containing $\mathbf{t}(x_q)$ and define $Q_q := q \times I_q$. By Lemma~\ref{lem: loc const}, we have 
\begin{equation*}
    |F(z_q)| \lesssim \||F| \ast \eta \|_{L^2(Q_q)}
\end{equation*}
where $\eta$ is rapidly decaying away from the unit cube in $\R^{n+1}$ centred at the origin. In particular, if we let $\delta := \varepsilon /(2n)$ and define the enlarged cube $Q_q^{(\delta)} := R^{\delta} \cdot Q_q$, then
\begin{equation*}
   |F(z_q)| \lesssim_{N, \varepsilon} \|F\|_{L^2(Q_q^{(\delta)})} + R^{-N}\|F\|_{L^2(\R^{n+1})} \qquad \textrm{for all $N \in \N_0$.}
\end{equation*}
Thus, if we define the family of space-time unit cubes $\cQ^{(\delta)} := \{Q_q^{(\delta)} : q \in \mathfrak{q}_R\}$, then
\begin{equation}\label{eq: lin red 1}
    \Big(\int_{B^n(0,R)} |e^{i\mathbf{t}(x)\Delta}f(x)|^2 \, \ud x \Big)^{1/2} \lesssim \|Uf\|_{L^2(Z_{\cQ^{(\delta)}})} + R^{-100n} \|Uf\|_{L^2(w_{B^{n+1}_R})},
\end{equation}
for $w_{B^{n+1}_R}$ a weight adapted to $B_R^{n+1}$. The rapidly decaying error term is easily bounded using the conservation of energy identity. In particular, it follows from the rapid decay of the weight and translation invariance properties of the operator that
\begin{equation}\label{eq: lin red 2}
    \|Uf\|_{L^2(w_{B^{n+1}_R})} \lesssim R^{1/2}\|f\|_{L^2(\R^n)}.
\end{equation}

On the other hand, the set $Z_{\cQ^{(\delta)}}$ can be covered by $O(R^{(n+1)\delta})$ sets of the form $Z_{\cQ}$ where $\cQ$ is a family of unit lattice cubes satisfying the hypotheses of Proposition~\ref{prop: lin max op} (and, in particular, the vertical line test). For any such $\cQ$, we may apply Proposition~\ref{prop: lin max op} to deduce that
\begin{equation*}
    \|Uf\|_{L^2(Z_{\cQ})} \lesssim_{\varepsilon} R^{n/(2(n+1)) + \delta} \|f\|_{L^2(\R^n)}.
\end{equation*}
Summing together these contributions, we obtain an estimate for $\|Uf\|_{L^2(Z_{\cQ^{(\delta)}})}$. This can be combined with \eqref{eq: lin red 1} and \eqref{eq: lin red 2} to deduce the desired bound \eqref{eq: lin max}. 
\end{proof}

We now discuss some further manifestations of the uncertainty principle which are of use in later arguments. 

By the basic scaling properties of the Fourier transform, Lemma~\ref{lem: loc const}  implies a generalisation of itself. We define a \textit{parallelepiped} to be set $\pi \subseteq \R^d$ given by the image of $Q_0 := [-1/2,1/2]^d$ under an affine transformation. In particular, $\pi = A(Q_0) + a$ for some $A \in \mathrm{GL}(d, \R)$ and $a \in \R^d$. Given such a parallelepiped, we define the \textit{dual parallelepiped} $\pi^* := A^{-\top}(Q_0)$, where $A^{-\top}$ is the inverse transpose of $A$, and 
\begin{equation*}
    \cP_{\mathrm{all}}(\pi) := \{ A^{-\top}Q : Q \in \cQ_{\mathrm{all}} \}.
\end{equation*}
Thus, $\cP_{\mathrm{all}}(\pi)$ is a family of translates of $\pi^*$ which tile the whole of $\R^d$. 

As a direct consequence of Lemma~\ref{lem: loc const} and an obvious change of variables, we obtain the following.

\begin{coro}\label{cor: loc const general} Let $\pi \subseteq \widehat{\R}^d$ be a parallelepiped. There exists a function $\eta_{\pi} \colon \R^n \to [0,\infty)$ satisfying the following:
\begin{enumerate}[i)]
    \item If $F \in \cS(\R^d)$ satisfies $\supp \hat{F} \subseteq \pi$, then
\begin{equation}\label{eq: loc const general}
    |F(z)| \leq \sum_{P \in \cP_{\mathrm{all}}(\pi)} a_P\chi_P \lesssim |F| \ast \eta_{\pi}(z) \qquad \textrm{for all $z \in \R^d$}
\end{equation}
where $\cP_{\mathrm{all}}(\pi)$ is as defined above and 
\begin{equation*}
    a_P := \sup_{z \in P} |F(z)| \qquad \textrm{for all $P \in \cP_{\mathrm{all}}(\pi)$.}
\end{equation*}
\item The function $\eta_{\pi}$ is $L^1$-normalised and rapidly decaying away from $\pi^*$ in the sense that
\begin{equation*}
    \eta(z) \lesssim_N  |\pi^*|^{-1} (1 + |z|_{\pi^*})^{-N} \qquad \textrm{for all $N \in \N$.}
\end{equation*}
Here $|\,\cdot\,|_{\pi^*}$ is the norm given by the Minkowski function of $\pi^*$. 
\end{enumerate}
\end{coro}

Corollary~\ref{cor: loc const general} realises the uncertainty principle by allowing us to pass from $F$ to its discretisation over $\cP_{\mathrm{all}}(\pi)$ and back again. However, the appearance of the mollifier $\eta_{\pi}$ on the right-hand side of \eqref{eq: loc const general} is a source of minor annoyance. It can be removed at the level of $L^q(\R^d)$ norm estimates: for instance, combining \eqref{eq: loc const} with Young's inequality we deduce that
\begin{equation*}
    \|F\|_{L^q(\R^d)} \leq \Big(\sum_{P \in \cP_{\mathrm{all}}(\pi)} |a_P|^q|P|\Big)^{1/q} \lesssim  \|F\|_{L^q(\R^d)}
\end{equation*}
for all $1 \leq q < \infty$. Generalising this argument, we arrive at the following fundamental estimate.

\begin{lemm}[Bernstein inequality]\label{lem: Bernstein} Let $1 \leq p \leq q \leq \infty$ and suppose $F \in \cS(\R^d)$ satisfies $\supp \hat{F} \subseteq \pi$ for some parallelepiped $\pi \subseteq \widehat{\R}^d$. Then 
\begin{equation*}
    \|F\|_{L^q(\R^d)} \lesssim |\pi|^{1/p - 1/q} \|F\|_{L^p(\R^d)}.
\end{equation*}
\end{lemm}

The moral here is that, by the locally constant property, $F$ behaves like a discrete function and, consequently, the $L^p$ norms of $F$ satisfy a nesting property like the $\ell^p$ norms of a sequence. 

\begin{proof}[Proof (of Lemma~\ref{lem: Bernstein})] We assume $q < \infty$; the same proof goes through for $q = \infty$ \textit{mutatis mutandis}. Fix $F \in \cS(\R^d)$ satisfying the Fourier support hypothesis and apply Corollary~\ref{cor: loc const general} to bound 
\begin{equation*}
    |F| \leq \sum_{P \in \cP_{\mathrm{all}}(\pi)} a_P \chi_P.
\end{equation*}
Since $|P| = |\pi|^{-1}$ for any $P \in \cP_{\mathrm{all}}(\pi)$, by the nesting of $\ell^p$ norms,
\begin{equation*}
    \|F\|_{L^q(\R^d)}  \leq \Big(\sum_{P \in \cP_{\mathrm{all}}(\pi)} |a_P|^q |P|\Big)^{1/q} \leq |\pi|^{1/p - 1/q} \Big(\sum_{P \in \cP_{\mathrm{all}}(\pi)}  |a_P|^p |P|\Big)^{1/p}.
\end{equation*}
Applying Corollary~\ref{cor: loc const general}, we have
\begin{equation*}
   \Big(\sum_{P \in \cP_{\mathrm{all}}(\pi)}  |a_P|^p |P|\Big)^{1/p} =  \Big\|\sum_{P \in \cP_{\mathrm{all}}(\pi)}  |a_P|\chi_P \Big\|_{L^p(\R^d)}  \lesssim \||F| \ast \eta_{\pi} \|_{L^p(\R^d)}. 
\end{equation*}
The desired result now follows by Young's convolution inequality. 
\end{proof}

 The Bernstein inequality can be localised in space, provided this localisation occurs at a scale which is coarse enough to respect the uncertainty principle. Before we state this local version, we introduce the following definition which plays a somewhat technical r\^ole in our arguments. 

\begin{defi}\label{def: adapted weight} Let $S_0 \subseteq \R^d$ be a symmetric convex body\footnote{In practice, we will only consider simple examples such as euclidean balls, cubes and cartesian products of these sets.} and $S := S_0 + z_0$ some translate of $S$. We say a function $w_{S} \colon \R^d \to [0,\infty)$ is a \emph{(rapidly decaying) weight adapted to $S$} if $w_S$ is continuous and satisfies
\begin{equation*}
    w_S(z) \lesssim_N (1 + |z-z_0|_{S_0})^{-N} \qquad \textrm{for all $N \in \N$,}
\end{equation*}
where here $|\,\cdot\,|_{S_0}$ is the norm given by the Minkowski function of $S_0$. 
\end{defi}

Such weight functions are used to account for `Schwartz tails errors' which arise when attempting to localise a function simultaneously in the physical and the frequency domain. 
 
\begin{coro}[Local Bernstein inequality]\label{cor: local Bernstein} Let $1 \leq p \leq q \leq \infty$, and $\pi \subseteq \widehat{\R}^d$ be a parallelepiped. For every $P \in \cP_{\mathrm{all}}(\pi)$ there exists a rapidly decaying weight $w_P$ adapted to $P$ such that the following holds. If $F \in \cS(\R^d)$ satisfies $\supp \hat{F} \subseteq \pi$, then 
\begin{equation*}
    \|F\|_{L^q(P)} \lesssim |\pi|^{1/p - 1/q} \|F\|_{L^p(w_P)}.
\end{equation*}
\end{coro}

\begin{proof} We may assume without loss of generality that $p < \infty$. It is a simple exercise to show that there exists some $\beta_P \in \cS(\R^d)$ with $\supp \hat{\beta}_P \subseteq \pi$ satisfying 
\begin{equation*}
    1 \lesssim |\beta_P(z)| \quad \textrm{for all $z \in P$} 
\end{equation*}
and such that $|\beta_P(z)|$ is a rapidly decaying weight adapted to $P$. The function $G := F\cdot \beta_P$ has Fourier support in the Minkowski sum $\pi + \pi$ and satisfies $|F(z)| \lesssim |G(z)|$ for all $z \in P$. Applying Bernstein's inequality to $G$, the desired result follows with $w_P := |\beta_P|^p$.
\end{proof}




\subsection{Parabolic geometry}\label{sec: Parabolic geometry} Given $f \in \cS(\R^n)$ with $\supp \hat{f} \subseteq B^n(0,1)$, we recall the integral formula for the solution
\begin{equation*}
    Uf(x,t) = \frac{1}{(2\pi)^n} \int_{B^n(0,1)} e^{i \phi(x,t;\xi)} \hat{f}(\xi)\,\ud \xi,
\end{equation*}
where the \textit{phase function} $\phi$ is given by
\begin{equation*}
    \phi(x,t;\xi) := \inn{x}{\xi} + t |\xi|^2.
\end{equation*}
The phase $\phi$ can be interpreted as the inner product of the space-time vector $(x,t) \in \R^{n+1}$ with a point lying on the bounded piece of the paraboloid
\begin{equation}\label{eq: bdd paraboloid}
    \Sigma_0 := \big\{ \Sigma(\xi)  : \xi \in B(0,1) \big\} \subset \Sigma := \big\{ \Sigma(\xi)  : \xi \in \widehat{\R}^n \big\}
\end{equation}
where $\Sigma(\xi) :=  (\xi, |\xi|^2)$. It follows that the (distributional) spatio-temporal Fourier transform of $Uf$ is supported on $\Sigma_0$. 

The geometry of $\Sigma$ underpins our entire analysis of the propagator $e^{it\Delta}$. As a first example of this, we observe certain symmetries of the paraboloid, which translate into symmetries of the propagator. 

Any ball $\theta \subset \widehat{\R}^n$ of radius $r$ corresponds to a \textit{cap}, or \textit{$r$-cap}, on the paraboloid, given by
\begin{equation*}
    \Sigma_{\theta} := \big\{ \Sigma(\xi)  : \xi \in \theta \big\}.
\end{equation*}
In view of this, we shall often refer to balls $\theta \subset \widehat{\R}^n$ themselves as `caps'. Note that the bounded piece of the paraboloid $\Sigma_0$ featured in \eqref{eq: bdd paraboloid} corresponds to the cap formed over the unit ball centred at the origin. We observe an important self-similarity property of the paraboloid, relating every cap to $\Sigma_0$.

\begin{lemm}[Parabolic rescaling: geometric version]\label{lem: parabolic resc} Given a cap $\Sigma_{\theta}$, corresponding to a ball $\theta \subseteq \widehat{\R}^n$, there exists an affine transformation $A_{\theta}$ of the ambient space $\widehat{\R}^{n+1}$ which restricts to a bijection from $\Sigma_0$ to $\Sigma_{\theta}$.
\end{lemm}

\begin{rema} By inverting and composing affine transformations, Lemma~\ref{lem: parabolic resc} further implies that given any two caps $\Sigma_{\theta_1}$, $\Sigma_{\theta_2}$, there exists an affine transformation $A_{\theta_2 \to \theta_1}$ of the ambient space $\widehat{\R}^{n+1}$ which restricts to a bijection from $\Sigma_{\theta_2}$ to $\Sigma_{\theta_1}$.  
\end{rema}

Before giving the (simple) proof of Lemma~\ref{lem: parabolic resc}, we introduce some notation and, in fact, give an explicit formula for $A_{\theta}$. Given a ball $\theta \subset \widehat{\R}^n$, let $\xi_{\theta}$ denote its centre and $\rad(\theta)$ its radius. We let $M_{\theta}$ be the shear transformation and $D_{\theta}$ the anisotropic scaling on $\widehat{\R}^{n+1}$ defined by 
\begin{equation*}
    M_{\theta} := \begin{pmatrix}
    I_n & 0 \\
    2 \xi_{\theta}^{\top} & 1
    \end{pmatrix}
    \qquad \textrm{and} \qquad
    D_{\theta} := \begin{pmatrix}
    \rad(\theta) I_n  & 0 \\
    0 & \rad(\theta)^2
    \end{pmatrix};
\end{equation*}
here $I_n$ denotes the $n \times n$ identity matrix. It will be shown in the proof below that the affine transformation $A_{\theta}$ in the statement of Lemma~\ref{lem: parabolic resc} can be taken to be 
\begin{equation}\label{eq: rescaling map}
  A_{\theta} \colon \zeta \mapsto  (\cL_{\theta})^{\top} \,\zeta + \Sigma(\xi_{\theta}) \qquad \textrm{where} \qquad \cL_{\theta} :=  D_{\theta}  \circ M_{\theta}^{\top}, 
\end{equation}
where here $\top$ is used to denote the matrix transpose. 
\begin{proof}[Proof (of Lemma~\ref{lem: parabolic resc})] Let $\theta \subseteq \widehat{\R}^n$ be a ball, so that the map
\begin{equation*}
    \eta \mapsto \xi_{\theta} + r_{\theta} \eta \quad \textrm{for} \quad \eta \in \widehat{\R}^n
\end{equation*} 
restricts to a bijection from $B(0,1)$ to $\theta$. Here $r_{\theta} := \rad(\theta)$. Moreover, if we fix $\xi \in \theta$ and write $\xi = \xi_{\theta} + r_{\theta} \eta$, then a simple computation shows that
\begin{equation}\label{eq: geom parabolic rescale}
    \Sigma(\xi) = \Sigma(\xi_{\theta}) + M_{\theta}  \circ D_{\theta} \Sigma(\eta)
\end{equation}
where $M_{\theta}$ and $D_{\theta}$ are as defined above. Indeed, \eqref{eq: geom parabolic rescale} follows directly from the expansion of the inner product
\begin{equation*}
    |\xi|^2 = |\xi_{\theta} + r_{\theta} \eta|^2 = |\xi_{\theta}|^2 + 2 r_{\theta}\, \inn{\xi_{\theta}}{\eta} + r_{\theta}^2 |\eta|^2, 
\end{equation*}
which may also be interpreted as a Taylor series expansion. Thus,  the map $A_{\theta}$ defined in \eqref{eq: rescaling map} satisfies the desired property. 
\end{proof}

We now relate the scaling property of the paraboloid to the solution operator $U$ via the formula $\phi(z; \xi) := \inn{z}{\Sigma(\xi)}$ for $z = (x,t) \in \R^{n+1}$. 

\begin{coro}[Parabolic rescaling]\label{cor: parabolic rescaling}  Let $\theta \subset \widehat{\R}^n$ be ball and suppose $f \in L^2(\R^n)$ satisfies $\supp \hat{f} \subseteq \theta \cap B^n(0,1)$. Then
\begin{equation}\label{eq: cor resc 1}
    |Uf(z)| = \rad(\theta)^{n/2}|U\tilde{f} \circ \cL_{\theta}(z)|
\end{equation}
for some function $\tilde{f} \in L^2(\R^n)$ satisfying
\begin{equation}\label{eq: cor resc 2}
\|\tilde{f}\|_{L^2(\R^n)} = \|f\|_{L^2(\R^n)} \quad \textrm{and} \quad \supp \mathcal{F}(\tilde{f}) \subseteq B^n(0,1).
\end{equation}
\end{coro}
Here the linear rescaling $\cL_{\theta}$ is as defined in \eqref{eq: rescaling map}.

\begin{proof}[Proof (of Corollary~\ref{cor: parabolic rescaling})] Let $\xi_{\theta}$ denote the centre of $\theta$ and $r_{\theta} := \rad(\theta)$. We simply define $\tilde{f}$ via the Fourier transform by
\begin{equation*}
    \mathcal{F}(\tilde{f})(\eta) := r_{\theta}^{n/2}\hat{f}(\xi_{\theta} + r_{\theta} \eta),
\end{equation*}
so that \eqref{eq: cor resc 2} immediately holds. On the other hand, we apply a change of variables to write
\begin{equation*}
    Uf(z) = r_{\theta}^{n/2} \int_{\widehat{\R}^n} e^{i \phi(z; \xi_{\theta} + r_{\theta}\eta)} \mathcal{F}(\tilde{f})(\eta)\,\ud \eta. 
\end{equation*}
The remaining property \eqref{eq: cor resc 1} now follows from the identity
\begin{equation*}
  \phi(z; \xi_{\theta} + r_{\theta}\eta) =  \inn{z}{\Sigma(\xi_{\theta} + r_{\theta}\eta)} = \inn{z}{\Sigma(\xi_{\theta})} + \inn{\cL_{\theta}z}{\Sigma(\eta)},
\end{equation*}
which is itself a simple consequence of \eqref{eq: geom parabolic rescale} and the definition of $\cL_{\theta}$.
\end{proof}




 \subsection{Wave packets}\label{sec: wave packet examples} We now analyse solutions to the Schr\"odinger equation for a class of very simple, well-behaved initial data. We shall see that the behaviour of our solutions is closely related to the geometry of the paraboloid $\Sigma$.

\begin{exem}[Unit-scale localised datum]\label{ex: unit scale datum} Fix $\psi \in \cS(\R^n)$ with non-negative Fourier transform satisfying 
\begin{equation*}
\supp \hat{\psi} \subseteq B^n(0,4) \quad \textrm{and} \quad \hat{\psi}(\xi) = 1 \quad \textrm{for all $\xi \in B^n(0,2)$.}    
\end{equation*}
By a simple integration-by-parts argument, the propagator
\begin{equation*}
    e^{it\Delta}\psi(x) = \frac{1}{(2 \pi)^n}\int_{\widehat{\R}^n} e^{i \inn{x}{\xi}} \Big(e^{it|\xi|^2} \hat{\psi}(\xi) \Big) \,\ud \xi 
\end{equation*}
satisfies
\begin{equation}\label{eq: basic datum}
    |e^{it\Delta}\psi(x)| \lesssim_N (1 + |x|)^{-N} \qquad \textrm{for all $N \in \N_0$ whenever $|t| \leq 1$.}
\end{equation}
Indeed, the point here is that for $|t| \leq 1$, the functions $e^{it|\,\cdot\,|^2} \hat{\psi}$ form a uniformly bounded family in the Schwartz class.
\end{exem}

\begin{exem}[Wave packets]\label{ex: wave packets}  We now rescale and translate the unit-scale localised datum $\psi$ to create a rich family of examples. Let $\rho \geq 1$ and $\theta_T \subset B^n(0,1)$ be a ball of radius $\rho^{-1/2}$ centred at $\xi_T \in \widehat{\R}^n$. Consider the $L^2$-normalised function $\psi_{\theta_T}$ given by\footnote{The reason for the apparently superfluous subscript $T$ will be made clear below.}
\begin{equation}\label{eq: psi theta T}
    \hat{\psi}_{\theta_T}(\xi) := \rho^{n/4}\hat{\psi}\big(\rho^{1/2}(\xi - \xi_T)\big).
\end{equation}
 In addition, for $x_T \in \R^n$ we consider the translated datum 
 \begin{equation}\label{eq: wave packet}
     \psi_T(x) := \psi_{\theta_T}(x - x_T).
 \end{equation}
 Thus, $\psi_T$ corresponds to a modulated bump function spatially localised in a ball of radius $\rho^{1/2}$ centred at $x_T$ and has frequency support lying in $4 \cdot \theta_T$.\smallskip

It follows from parabolic rescaling in the form of Corollary~\ref{cor: parabolic rescaling}, translation invariance and \eqref{eq: basic datum} that 
\begin{equation*}
    |e^{it\Delta}\psi_T(x)| \lesssim_N (1 + \rho^{-1/2}|x - x_T + 2t \xi_T|)^{-N} \qquad \textrm{for all $N \in \N_0$ whenever $|t| \leq \rho$.}
\end{equation*}
Thus, during the time interval $|t| \leq \rho$, the solution at time $t$ is concentrated in the spatial ball $B(x_T-2t \xi_T, \rho^{1/2})$, in the sense that $e^{it\Delta}\psi_T$ rapidly decays away from this set. We illustrate this phenomenon in 1 spatial dimension in Figure~\ref{fig: fixed time Wavepacket}. This solution is an example of what is known as a \textit{wave packet}.
\end{exem}

\begin{figure}
     \centering
     \begin{subfigure}[b]{0.45\textwidth}
         \centering
\includegraphics{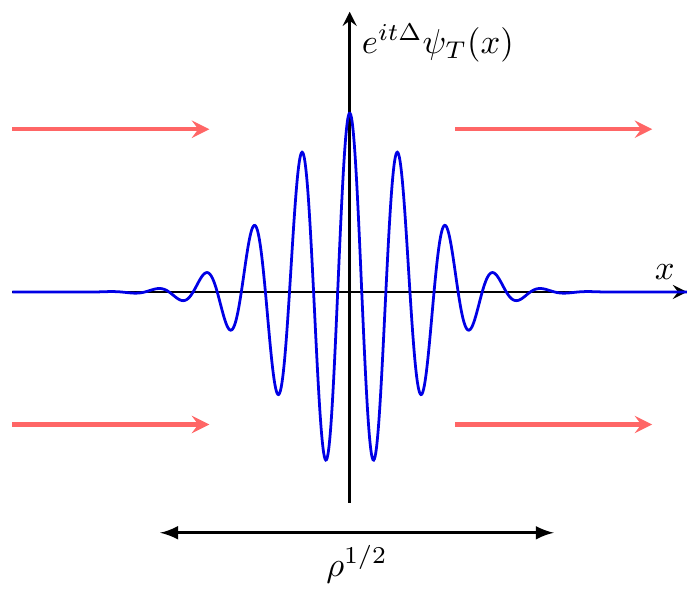}
         \caption{{\small A wave packet at time $t$. The wave packet is concentrated in an interval in the physical space and, as time evolves, travels with fixed velocity $v(T) = -2 \xi_T$ proportional to the frequency $\xi_T$.}}
         \label{fig: fixed time Wavepacket}
     \end{subfigure}
     \hfill
     \begin{subfigure}[b]{0.45\textwidth}
         \centering

\includegraphics{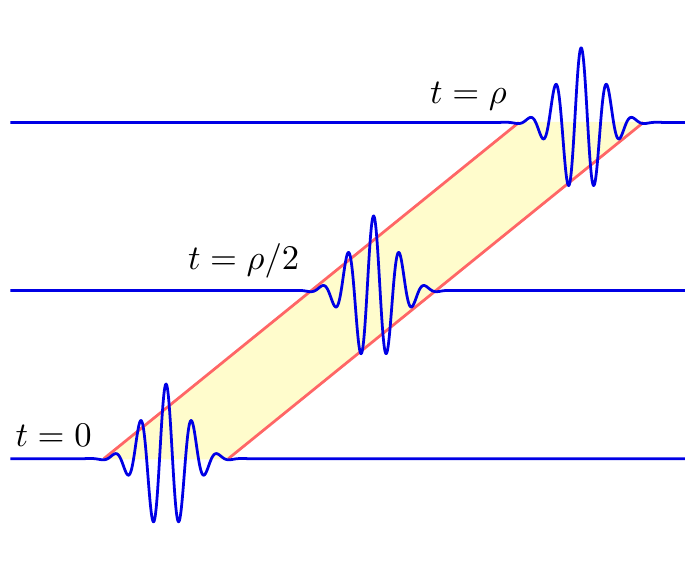}

        \caption{{\small A space-time tube $T$. The tube describes the spatial localistion of the solution $U\psi_T$ at each time slice. Three such time slices are illustrated in the figure.}}
         \label{fig: space time Wavepacket}
         \hfill
     \end{subfigure}

        \caption{Two perspectives on wave packets.}
    
\end{figure}

We highlight the basic properties of the wave packets from  Example~\ref{ex: wave packets}:
\begin{itemize}
\item The initial datum $\psi_T$ (and the solution $e^{it\Delta}\psi_T$ at each fixed time) is frequency supported in a ball centred at $\xi_T$ of radius $\sim \rho^{-1/2}$;
\item For each time $t$ satisfying $|t| \leq \rho$ the solution is concentrated in a spatial ball of radius $\rho^{1/2}$;
\item The spatial ball travels from the initial position $x_T$ at $t = 0$ along a linear trajectory with velocity $-2 \xi_T$. In particular, the velocity is determined by the frequency.
\end{itemize}

This illustrates the fundamental \textit{dispersion relation} between the velocity $v$ of a Schr\"odinger wave and its frequency $\xi$, summed up by the formula
\begin{equation*}
    v = -2\xi.
\end{equation*}
The fact that waves of different frequency travel with different velocities (and therefore disperse at large time scales) is the defining characteristic of \textit{dispersive PDE}.

\begin{rema} For our wave packets $\psi_T$, we have localised the frequency at scale $\rho^{-1/2}$ so that the distinct frequency modes making up the wave $e^{it\Delta}\psi_T$ travel at the same velocity up to an error of $O(\rho^{-1/2})$. Consequently, the wave packet remains `stable' for the time scale $|t| \leq \rho$. Beyond this time scale, the difference in velocities between the distinct modes causes the wave packet to disperse. 
\end{rema}

In view of what follows, it is useful to adopt a new perspective and visualise the wave packet $e^{it\Delta}\psi_T(x)$ as a function on the spatio-temporal domain $\R^{n+1}$. 

 \begin{defi}\label{def: space-time tube} For $\rho \geq 1$, a \emph{space-time $\rho$-tube}, or simply a \emph{$\rho$-tube}, is a set $T \subset \R^{n+1}$ of the form
\begin{equation*}
    T := \big\{ (x,t) \in \R^{n+1} : |x - x(T) - t v(T)| \leq \rho^{1/2} \quad \textrm{and} \quad |t| \leq \rho \big\}
\end{equation*}
for some $x(T)$, $v(T) \in \R^n$. In this case, we say $T$ has initial position $x(T)$, velocity $v(T)$, duration $\rho$ and spatial radius $\rho^{1/2}$. 
\end{defi}

As a function of $(x,t)$, we see that the wave packet\footnote{By the hypothesis $\theta \subseteq B^n(0,1)$, we have $e^{it\Delta} \psi_T(x) = U\psi_T(x,t)$.} $U\psi_T(x,t)$ is concentrated on the space-time tube $T$ centred at $x_T$ with velocity $v(T) = -2\xi_T$, duration $\rho$ and spatial radius $\rho^{1/2}$; see Figure~\ref{fig: space time Wavepacket}. In particular, if for $0 < \delta < 1$ we define the slightly enlarged space-time tube
\begin{equation}\label{eq: enlarged T}
    T^{(\delta)} := \big\{ (x,t) \in \R^{n+1} : |x - x(T) - t v(T)| \leq \rho^{1/2 + \delta} \quad \textrm{and} \quad |t| \leq \rho \big\},
\end{equation}
then
\begin{equation}\label{eq: wave packet local}
    |U\psi_T(x,t)| \lesssim_{\delta,N} \rho^{-N} \qquad \textrm{if $(x,t) \in B^{n+1}(0,\rho) \setminus T^{(\delta)}$.}
\end{equation}

We update our notation to accommodate our change in perspective. If $T \subset \R^{n+1}$ is a space-time $\rho$-tube as in Definition~\ref{def: space-time tube}, then we define $\theta_T \subset \widehat{\R}^n$ to be the ball centred at $\xi_T := -v(T)/2$ of side-length $\rho^{-1/2}$. Moreover, we let $\psi_T$ denote the function \eqref{eq: wave packet} from Example~\ref{ex: wave packets}.

Note that a $\rho$-tube $T$ is centred around its \textit{core line}
\begin{equation*}
    \ell_T := \{(x,t) \in \R^{n+1} : x = x_T - 2t \xi_T \}.
\end{equation*}
This has direction $G(\xi_T)$ where
\begin{equation}\label{eq: def Gauss map}
    G(\xi) := \frac{1}{\sqrt{1 + 4|\xi|^2}} 
    \begin{pmatrix}
        -2\xi \\
        1
    \end{pmatrix};
\end{equation}
we will also refer to $G(\xi_T)$ as the \textit{direction} of $T$. Note that $G$ is precisely the Gauss map associated to the paraboloid $\Sigma$. Thus, the direction of the tube $T$ corresponds to the normal vector to the paraboloid at the centre of the cap $\Sigma_T$ associated to $\theta_T$.\medskip

\begin{figure}
     \centering
     \begin{subfigure}[b]{0.45\textwidth}
         \centering

        \includegraphics[scale=0.75]{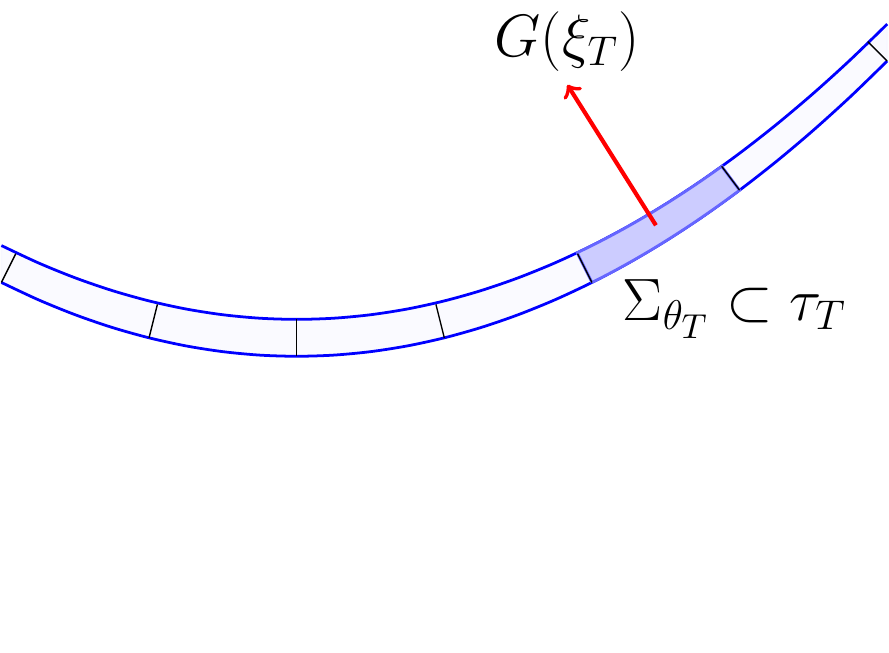}       
     \caption{\small{Each $\rho^{-1/2}$-cap $\theta_T$ is contained in a rectangle $\tau_T$ of dimension approximately $\rho^{-1/2} \times \cdots \times \rho^{-1/2} \times \rho^{-1}$, with short side in the direction of $G(\xi_T)$.}}
     \hfill
     \end{subfigure}
     \hfill
     \begin{subfigure}[b]{0.45\textwidth}
         \centering
       \includegraphics[scale=1.2]{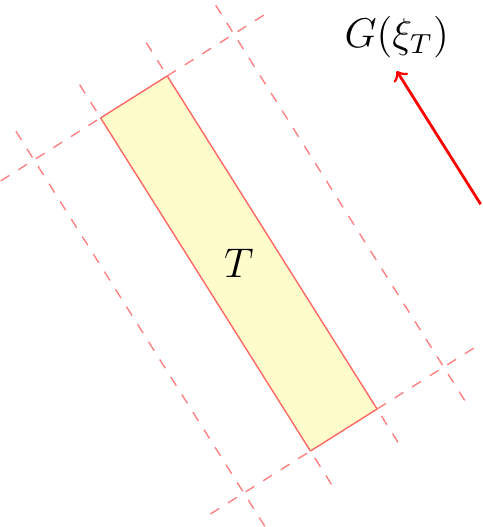}
       \vspace*{3mm} 
       \caption{\small{The space-time $\rho$-tube $T$ associated to the cap $\theta_T$. This tube has dimensions $\rho^{1/2} \times \cdots \times \rho^{1/2} \times \rho$, with long side in the direction of $G(\xi_T)$.}}
          \hfill
     \end{subfigure}
        \caption{The wave packets decomposition respects the uncertainty principle.}
        \label{fig: caps vs wavepackets}
\end{figure}

\begin{rema} These observations are consistent with the uncertainty principle and should be compared with the statement of Corollary~\ref{cor: loc const general}. Indeed, $e^{it\Delta}\psi_T (x)$ has spatio-temporal Fourier support in $\Sigma_{\theta_T}$. This cap is itself contained in a rectangle $\tau_T$ centred at $\xi_T$ with $n$ sides of length $O(\rho^{-1/2})$ lying in the tangent directions of $\Sigma$ at $\Sigma(\xi_T)$ and a remaining side of length $O(\rho^{-1})$ lying in the normal direction. The (essential) spatial-temporal support $T$ therefore corresponds to a translate of the dual $\tau_T^*$ of the frequency support $\tau_T$. See Figure~\ref{fig: caps vs wavepackets}.
\end{rema}




\subsection{Wave packet decomposition}\label{sec: wave packet dec} The examples considered of the previous section appear rather specialised. Nevertheless, using a Fourier series decomposition, \textit{any} initial datum $f \in L^2(\R^n)$ can be expressed as a superposition 
\begin{equation*}
    f = \sum_{T \in \T} a_T \psi_T
\end{equation*}
where $\T$ is a (possibly infinite) collection of space-time tubes; the $\psi_T$ are basic initial data introduced above and $(a_T)_{T \in \T}$ is a sequence of complex coefficients. Consequently, 
\begin{equation*}
    e^{it\Delta} f = \sum_{T \in \T} a_T e^{it\Delta} \psi_T
\end{equation*}
and so any solution can be realised as a superposition of wave packets. This observation is referred to as the \textit{wave packet decomposition}. It is essentially a consequence of the uncertainty principle and is closely related to the local constancy property described in Corollary~\ref{cor: loc const general}.




\subsubsection*{Definition and basic properties} Turning to the precise details, we first introduce some notation. 

\begin{defi} For $\rho \geq 1$, let $\T[\rho]$ denote the collection of all space-time $\rho$-tubes as in Definition~\ref{def: space-time tube} where
\begin{itemize}
    \item The initial position $x(T)$ is free to vary over the lattice $\rho^{1/2}\Z^n$;
    \item The velocity $v(T)$ is free to vary over the lattice $c_n \rho^{-1/2}\Z^n$.
\end{itemize}
Here  $c_n := 1/ 2n^{1/2}$ is a fixed constant which plays a minor technical r\^ole.
\end{defi}

The precise form of the wave packet decomposition is given by the following lemma. 

\begin{lemm}[Wave packet decomposition]\label{lem: wp dec} Given $f \in L^2(\R^n)$ with $\supp \hat{f} \subseteq B^n(0,1/2)$ and $\rho \geq 10$ we may write
\begin{equation}\label{eq: wp dec}
    f = \sum_{T \in \T[\rho]} f_T
\end{equation}
    where the functions $f_T \in \cS(\R^n)$ satisfy the following:
    \begin{enumerate}[i)]
    \item \textbf{Dispersion relation.} Let $T \in \T[\rho]$ and suppose $f_T$ is not identically $0$. Then 
    \begin{equation*}
        |v(T) + 2 \xi| \leq 2\rho^{-1/2} \qquad \textrm{for some $\xi \in \supp \hat{f}$.}
    \end{equation*}
        \item \textbf{Orthogonality.} For any collection $\bW \subseteq \T[\rho]$, we have
\begin{equation}\label{eq: wp orthogonality}
    \big\| \sum_{T \in \bW} f_T \big\|_{L^2(\R^n)}^2 \lesssim \sum_{T \in \bW} \|f_T\|_{L^2(\R^n)}^2 \lesssim \|f\|_{L^2(\R^n)}^2.
\end{equation} 
        \item \textbf{Spatio-temporal localisation.} If for $T \in \T[\rho]$ and $0 < \delta < 1$ we define the slightly enlarged space-time tube $T^{(\delta)}$ as in \eqref{eq: enlarged T}, then
\begin{equation}\label{eq: wp localisation}
    |Uf_T(x,t)| \lesssim_{\delta,N} \rho^{-N} \|f\|_{L^2(\R^n)} \qquad \textrm{if $(x,t) \in B^{n+1}(0,\rho) \setminus T^{(\delta)}$.}
\end{equation}
    \end{enumerate}
    We refer to \eqref{eq: wp dec} as the wave packet decomposition of $f$ at scale $\rho$.
\end{lemm}

The idea is to first decompose $f$ as a sum of pieces $f_{\theta}$, each  frequency supported in some cap $\theta$. The individual $f_{\theta}$ are then further decomposed using Fourier series. Before giving the details, we fix some notation. For $\rho \geq 1$ and let 
\begin{equation*}
    \Theta[\rho] := \big\{ B(\xi_T, \rho^{-1/2}) : \xi_T \in c_n \rho^{-1/2}\Z^n \cap B^n(0,2) \big\}
\end{equation*}
denote a covering of $B^n(0,2)$ by $\rho^{-1/2}$-caps. As at the end of \S\ref{sec: wave packet examples}, we associate to each $T \in \T[\rho]$ a cap $\theta_T \in \Theta[\rho]$ with centre $\xi_T$ satisfying the dispersion relation $v(T) = -2\xi_T$.

Fix $\phi \in \cS(\R^n)$ satisfying\footnote{To construct such a function, choose $\beta \in C^{\infty}_c(\R^n)$ such that $\supp \beta \subseteq Q_0 := [-1/2,1/2]^n$ and $\int_{\R} \beta = 1$. Since $\sum_{k\in \Z^n} \chi_{Q_0} \ast \beta(\,\cdot\, - k) \equiv 1$, we may take $\hat{\phi}$ to be a suitable scaling of $\chi_{Q_0} \ast \beta$.} 
\begin{equation*}
\supp \hat{\phi} \subseteq [-c_n, c_n]^n \quad \textrm{and} \quad \sum_{k \in \Z^n} \hat{\phi}(\xi - c_n k) = 1 \quad \textrm{for all $\xi \in \widehat{\R}^n$.}    
\end{equation*}
 Furthermore, let $\psi \in \cS(\R^n)$ be as in Example~\ref{ex: unit scale datum}, so that 
\begin{equation}\label{eq: self-reproducing}
    \hat{\phi}(\xi) = \hat{\phi}(\xi) \hat{\psi}(\xi) \qquad \textrm{for all $\xi \in \widehat{\R}^n$.}
\end{equation}

Given a cap $\theta_T \in \Theta[\rho]$ with centre $\xi_T$, define $\phi_{\theta_T} \in \cS(\R^n)$ by
\begin{equation*}
    \hat{\phi}_{\theta_T}(\xi) := \hat{\phi}(\rho^{1/2} (\xi - \xi_T)).
\end{equation*}
The family of functions $\hat{\phi}_{\theta_T}$ then forms a partition of unity subordinate to the covering $\Theta[\rho]$ of $B^n(0,2)$. Consequently, 
\begin{equation}\label{eq: f freq dec}
    f = \sum_{\theta_T \in \Theta[\rho]} f_{\theta_T} \qquad \textrm{where} \qquad f_{\theta_T} := \phi_{\theta_T} \ast f.
\end{equation}
With these definitions, we turn to the proof of the wave packet decomposition.

\begin{proof}[Proof (of Lemma~\ref{lem: wp dec})] For $f_{\theta_T}$ as defined in \eqref{eq: f freq dec}, it follows that $\hat{f}_{\theta_T}$ is compactly supported in a cube of side-length $\rho^{-1/2}$. Consequently, we can expand $\hat{f}_{\theta_T}$ as a (suitably scaled) Fourier series, giving
\begin{equation}\label{eq: wave packet dec 1} 
    \hat{f}_{\theta_T}(\xi) =  \sum_{k \in \rho^{1/2} \Z^n} \rho^{n/2} e^{-i \inn{\xi}{k}} f_{\theta_T}(k) = \sum_{k \in \rho^{1/2} \Z^n} \rho^{n/4} e^{-i \inn{\xi}{k}} f_{\theta_T}(k) \hat{\psi}_{\theta_T}(\xi).
\end{equation}
Here the convergence can be taken in the $L^2$ sense and $\psi_{\theta}$ is as defined in \eqref{eq: psi theta T}. For the second step in \eqref{eq: wave packet dec 1}  we have used \eqref{eq: self-reproducing}. Note, by Plancherel's theorem,
\begin{equation*}
\sum_{k \in \rho^{1/2} \Z^n} |f_{\theta_T}(k)|^2  = \rho^{-n/2}
 \|f_{\theta}\|_{L^2(\R^n)}^2. 
\end{equation*}

In light of the above, we may write
\begin{equation*}
    f = \sum_{T \in \T[\rho]} f_T \quad \textrm{where} \quad f_T(x) := \rho^{n/4} f_{\theta_T}(x(T)) \hat{\psi}_T(x)
\end{equation*}
for $\psi_T$ as defined in \eqref{eq: wave packet}. It remains to prove the functions $f_T$ have the desired properties.\smallskip

\noindent i) \textbf{Dispersion relation.} From the definitions, $\supp \hat{\phi}_{\theta_T}$ lies in the ball $\theta_T$ of radius $\rho^{-1/2}$. If $f_T$ is not identically zero, then $f_{\theta_T}$ is not identically zero. In this case, $\theta_T$ must intersect the frequency support of $f$, which immediately implies the desired property. \smallskip

\noindent ii) \textbf{Orthogonality.} For each $\theta \in \Theta[\rho]$, let $\T_{\theta}[\rho]$ denote the collection of all tubes $T \in \T[\rho]$ associated to $\theta$ and $\bW_{\theta} := \bW \cap \T_{\theta}[\rho]$. Since the functions $\sum_{T \in \mathbf{W_{\theta}}} f_T$ have finitely overlapping Fourier support,
\begin{equation}\label{eq: wave packet dec 3}
    \big\| \sum_{T \in \bW} f_T \big\|_{L^2(\R^n)}^2  = \big\| \sum_{\theta \in \Theta[\rho]} \sum_{T \in \bW_{\theta}} f_T \big\|_{L^2(\R^n)}^2 \lesssim  \sum_{\theta \in \Theta[\rho]}  \big\|\sum_{T \in \bW_{\theta}} f_T \big\|_{L^2(\R^n)}^2,
\end{equation}
 where in the final step we use Plancherel's theorem and the Cauchy--Schwarz inequality.

Fix $\theta \in \Theta[\rho]$ and note that $\supp \hat{\psi}_{\theta}$ lies in the cube $\tilde{\theta}$ concentric to $\theta$ with side-length $8\rho^{-1/2}$. In particular, by Plancherel's theorem (and periodicity), 
\begin{equation}\label{eq: wave packet dec 4}
    \big\|\sum_{T \in \bW_{\theta}} f_T \big\|_{L^2(\R^n)}^2 \lesssim \rho^{n} \big\|\sum_{T \in \bW_{\theta}} f_{\theta}(x(T)) e^{-i \inn{\xi}{x(T)}} \big\|_{L^2(\tilde{\theta})}^2 \sim \rho^{n/2}\sum_{T \in\bW_{\theta}} |f_{\theta}(x(T))|^2.
\end{equation}
However, from the definition of $f_T$ and the choice of normalisation of $\psi_T$, it follows that 
\begin{equation}\label{eq: wave packet dec 5}
   \rho^{n/2} |f_{\theta}(x(T))|^2 \sim \|f_T\|_{L^2(\R^n)}^2  \qquad \textrm{for any $T \in \T_{\theta}[\rho]$.}
\end{equation}
Combining \eqref{eq: wave packet dec 3}, \eqref{eq: wave packet dec 4} and \eqref{eq: wave packet dec 5}, we deduce
\begin{equation*}
     \big\| \sum_{T \in \bW} f_T \big\|_{L^2(\R^n)}^2 \lesssim  \sum_{T \in \bW} \|f_T\|_{L^2(\R^n)}^2,
\end{equation*}
which is the first of the desired inequalities.

On the other hand, in view of \eqref{eq: wave packet dec 1} and \eqref{eq: wave packet dec 5}, we have
\begin{equation*}
    \sum_{T \in \bW} \|f_T\|_{L^2(\R^n)}^2 \lesssim \rho^{n/2} \sum_{\theta \in \Theta[\rho]} \sum_{T \in \bW_{\theta}} |f_{\theta}(x(T))|^2 \lesssim \sum_{\theta \in \Theta[\rho]} \|f_{\theta}\|_{L^2(\R^n)}^2 \lesssim \|f\|_{L^2(\R^n)},
\end{equation*}
where the last inequality follows since the $f_{\theta}$ has finitely-overlapping frequency support.

\smallskip

\noindent iii) \textbf{Spatio-temporal locality.} It follows from \eqref{eq: wave packet local} that 
\begin{equation}\label{eq: spatial loc 1}
    |Uf_T(x,t)| \lesssim_N \rho^{-N} |f_{\theta_T}(x(T))|  \qquad \textrm{if $(x,t) \in B^{n+1}(0,\rho) \setminus T^{(\delta)}$.}
\end{equation}
   However, since $f_{\theta_T}$ has Fourier support in a small ball, it follows from Bernstein's inequality and the orthogonality properties of the wave packets that
   \begin{equation}\label{eq: spatial loc 2}
       |f_{\theta_T}(x(T))| \leq \|f_{\theta_T}\|_{L^{\infty}(\R^n)} \lesssim \|f_{\theta_T}\|_{L^2(\R^n)} \lesssim \|f\|_{L^2(\R^n)} .
   \end{equation}
Combining \eqref{eq: spatial loc 2} and \eqref{eq: spatial loc 1} gives the desired bound. 
\end{proof}

 We can use Lemma~\ref{lem: wp dec} to address a point left open from \S\ref{sec: Standard reductions}. The following argument is essentially taken from \textcite{L2006}. 

\begin{proof}[Proof (Proposition~\ref{prop: reduced max} $\Rightarrow$ Theorem~\ref{thm: DZ max})] Assume Proposition~\ref{prop: reduced max} holds. For all $\varepsilon > 0$ and $R \geq 1$, it suffices to show 
\begin{equation}\label{eq: pre reduced max}
    \|\sup_{0 < t < R^2}|e^{it\Delta}f| \|_{L^2(B^n(0,R))} \lesssim_{\varepsilon} R^{n/(2(n+1)) + \varepsilon} \|f\|_{L^2(\R^n)}
\end{equation}
for all $f \in L^2(\R^n)$ with $\supp \hat{f} \subseteq A(1)$. Indeed, as discussed in \S\ref{sec:  Littlewood--Paley decomposition}, the desired result then follows from a simple scaling argument and the Littlewood--Paley characterisation of $H^s(\R^n)$.

Fix $\varepsilon > 0$, $R \geq 1$ and $f \in L^2(\R^n)$ with $\supp \hat{f} \subseteq A(1)$. Let $\cI_R$ be the collection of all lattice $R$-intervals which intersect the long time interval $[0,R^2]$. Trivially, we have
\begin{align}
    \nonumber
     \|\sup_{0 < t \leq R^2}|e^{it\Delta}f| \|_{L^2(B^n(0,R))} &= \|\sup_{I \in \cI_R}\sup_{t \in I}|e^{it\Delta}f| \|_{L^2(B^n(0,R))} \\
     \label{eq: R2 red 1}
     &\leq \Big(\sum_{I \in \cI_R} \|\sup_{t \in I}|e^{it\Delta}f|\|_{L^2(B^n(0,R))}^2 \Big)^{1/2}.
\end{align}
On the other hand, by Proposition~\ref{prop: reduced max} and the invariance of the estimates under temporal translation,
\begin{equation}\label{eq: R2 red 2}
    \|\sup_{t \in I}|e^{it\Delta}g| \|_{L^2(B^n(0,R))} \lesssim_{\varepsilon} R^{n/(2(n+1)) + \varepsilon/2} \|g\|_{L^2(\R^n)} \qquad \textrm{for all $I \in \cI_R$}
\end{equation}
whenever $g \in L^2(\R^n)$ with $\supp \hat{g} \subseteq A(1)$.\smallskip

In order to sum the estimates from \eqref{eq: R2 red 2} we observe a certain orthogonality between the localised maximal operators associated to distinct time intervals $I$. We first discuss this orthogonality at a heuristic level. As above, let $\T[R]$ denote the collection of all space-time $R$-tubes. Since $\supp \hat{f} \subseteq A(1)$, given $T \in \T[R]$ the wave packet $e^{it\Delta}f_T$ has speed $|v(T)| \sim 1$. This means that the wave $e^{it\Delta}f_T$ spends roughly $R$ units of time in the spatial ball $B^n(0,R)$. Hence, for each $T \in \T[R]$ there is essentially a unique time interval $I \in \cI_R$ for which $\sup_{t \in I}|e^{it\Delta}f_T|$ is non-negligible. Thus, each wave packet contributes only to a single term in the sum on the right-hand side of \eqref{eq: R2 red 1}; since the functions $(f_T)_{T \in \T[R]}$ are themselves orthogonal, this leads to the desired orthogonality between the maximal operators.\smallskip

We now turn to the formal details of the proof. Fix $\delta := \varepsilon / 2$ and for each $I \in \cI_R$, let $Q_I := B^n(0,R) \times I$ and $Q_I^{(\delta)} := R^{\delta} \cdot Q_I$. Define
\begin{equation}\label{eq: R2 red 3}
    \T_I[R] := \big\{ T \in \T[R] : T \cap Q_I^{(\delta)} \neq \emptyset \textrm{ and } 1/4 \leq |v(T)| \leq 2 \big\}
\end{equation} 
so that, by frequency support hypotheses on $f$ and the spatio-temporal locality of the wave packets,
\begin{equation}\label{eq: R2 red 4}
    \|\sup_{t \in I}|e^{it\Delta}f| \|_{L^2(B^n(0,R))} \lesssim_{\varepsilon} \|\sup_{t \in I}|e^{it\Delta}f_I| \|_{L^2(B^n(0,R))} + R^{-100n} \|f\|_{L^2(\R^n)}
\end{equation}
where
\begin{equation*}
    f_I := \sum_{T \in \T_I[R]} f_T. 
\end{equation*}

The key observation is that the sets $\T_I[R]$ are essentially disjoint: more precisely, 
\begin{equation}\label{eq: R2 red 5}
    \max_{T \in \T[R]} \#\{ I \in \cI_R : T \in \T_I[R] \} \lesssim R^{\delta}. 
\end{equation}
Indeed, once we have \eqref{eq: R2 red 5}, we may combine \eqref{eq: R2 red 1}, \eqref{eq: R2 red 2} and \eqref{eq: R2 red 4} to deduce that
\begin{align*}
    \|\sup_{0 < t \leq R^2}|e^{it\Delta}f| \|_{L^2(B^n(0,R))} &\lesssim_{\varepsilon} R^{n/(2(n+1)) + \varepsilon/2} \Big(\sum_{I \in \cI_R} \|f_I\|_{L^2(\R^n)}^2 \Big)^{1/2}  + R^{-100n} \|f\|_{L^2(\R^n)} \\
    &\lesssim_{\varepsilon} R^{n/(2(n+1)) + \varepsilon} \Big(\sum_{T \in \T[R]} \|f_T\|_{L^2(\R^n)}^2 \Big)^{1/2}  + R^{-100n} \|f\|_{L^2(\R^n)}
\end{align*}
where the second step follows from orthogonality, interchanging the order of summation and \eqref{eq: R2 red 5}. The desired estimate \eqref{eq: pre reduced max} now follows by the orthogonality property of the wave packets.\smallskip

It remains to show \eqref{eq: R2 red 5}. Fix $T \in \T[R]$ with initial position $x(T)$ and velocity $v(T)$. Suppose $T \in \T_{I_1}[R] \cap \T_{I_2}[R]$ for a pair of time intervals $I_1$, $I_2 \in \cI_R$. It suffices to show $\dist(I_1, I_2) \leq 30 R^{1+\delta}$. Suppose this is not the case and let $\bar{t}_j$ denote the centre of $I_j$ for $j=1$, $2$. By the definition of the sets $\T_{I_j}[R]$, there exist $(x_j,t_j) \in \R^{n+1}$ such that
\begin{equation*}
    |x_j - x(T) - t_j v(T)| \leq R^{1/2}, \quad |x_j| \leq R^{1+\delta} \quad \textrm{and} \quad |t_j - \bar{t}_j| \leq R^{1+\delta}
\end{equation*}
for $j = 1$, $2$. By the triangle inequality,
\begin{equation*}
    |\bar{t}_1 - \bar{t}_2||v(T)| \leq |x_1 - x_2|  + \sum_{j=1}^2 |x_j - x(T) - t_jv(T)|  + |t_j - \bar{t}_j| \leq 6R^{1+\delta}.
\end{equation*}
Since, by hypothesis, $|\bar{t}_1 - \bar{t}_2| \geq 30 R^{1+\delta}$, it follows that $|v(T)|  \leq 1/10$. In view of the definition of $\T_{I_j}[R]$ from \eqref{eq: R2 red 3}, this is a contradiction. 
\end{proof}




\subsection{Pseudo-local properties}\label{sec: pseudo-local}

Throughout the following, we fix a large spatio-temporal scale $R \geq 1$. This scale plays a similar r\^ole to $R$ in the statement of Theorem~\ref{thm: fractal L2} and, in particular, we shall consider estimates for $Uf$ which are localised in space-time to $B^{n+1}(0,R)$. For this reason, it is useful to perform a wave packet decomposition at scale $R$, so that the corresponding space-time tubes have duration $R$.

We shall often work with an additional, much smaller, intermediate scale $1 \leq K \leq R^{1/2}$; this should be thought of as reasonably large, but still vastly smaller than $R$. For instance, in applications we typically either take $K$ to be a large dimensional constant or $K = R^{\delta}$ for some very small $\delta > 0$. 

Let $\cT_{K^{-1}}$ be a finitely-overlapping covering of $B^n(0,1)$ by $(2K)^{-1}$-caps with centres lying in $B^n(0,2)$. Fix $\tau \in \cT_{K^{-1}}$ with $\tau \subseteq B^n(0,1)$ and suppose $f_{\tau} \in L^2(\R^n)$ satisfies $\supp \hat{f}_{\tau} \subseteq \tau$.\footnote{The caps are chosen to have radius $(2K)^{-1}$ for technical reasons; one should think of $f_{\tau}$ as supported `well within' a cap of radius $K^{-1}$.} By applying a wave packet decomposition to $f_{\tau}$ at scale $R$, we have 
\begin{equation*}
 f_{\tau} = \sum_{T \in \T_{\tau}[R]} f_T \qquad \textrm{where} \qquad    \T_{\tau}[R] := \big\{ T \in \T[R] : |v(T) + 2 \xi_{\tau} | \leq 4K^{-1} \big\};
\end{equation*}
here the dispersion relation property from Lemma~\ref{lem: wp dec} is used to restrict the velocities. 

Each $T \in \T_{\tau}[R]$ is aligned in the direction $G(\xi_{\tau})$ up to an angular difference of $O(K^{-1})$. Consequently, we can group the wave packets together according to a partition of space-time into disjoint strips $S$ of dimension $R/K \times \cdots \times R/K \times R$: see Figure~\ref{fig: Pseudolocal_partition}. In view of this, the solution operator $U$ essentially acts independently between distinct strips. 

To describe these properties more precisely, let $\bbS_{\tau}[R]$ denote the collection of all strips
\begin{equation*}
    S = \big\{ (x,t) \in \R^{n+1} : |x - x(S) - t v(S)| \leq R/K \quad \textrm{and} \quad |t| \leq R \big\}
\end{equation*}
where $x(S) \in \R^n$ is a choice of lattice point in $R/K \cdot \Z^n$ and $v(S) := - 2 \xi_{\tau}$. Form a partition $\T_{S}[R]$ of $\T_{\tau}[R]$, parameterised by the the strips $S \in \bbS_{\tau}[R]$ such that 
\begin{equation*}
    \T_{S}[R] \subseteq \big\{ T \in \T_{\tau}[R] : T \cap S \neq \emptyset \big\}.
\end{equation*}
Then we may write
\begin{equation}\label{eq: pseudo loc 1}
     f_{\tau} = \sum_{S \in \bbS_{\tau}[R]} f_S \qquad \textrm{where} \qquad  f_S = \sum_{T \in \T_S[R]} f_T  \quad \textrm{for all $S \in \bbS_{\tau}[R]$.} 
\end{equation}
\begin{figure}
    \centering    \includegraphics{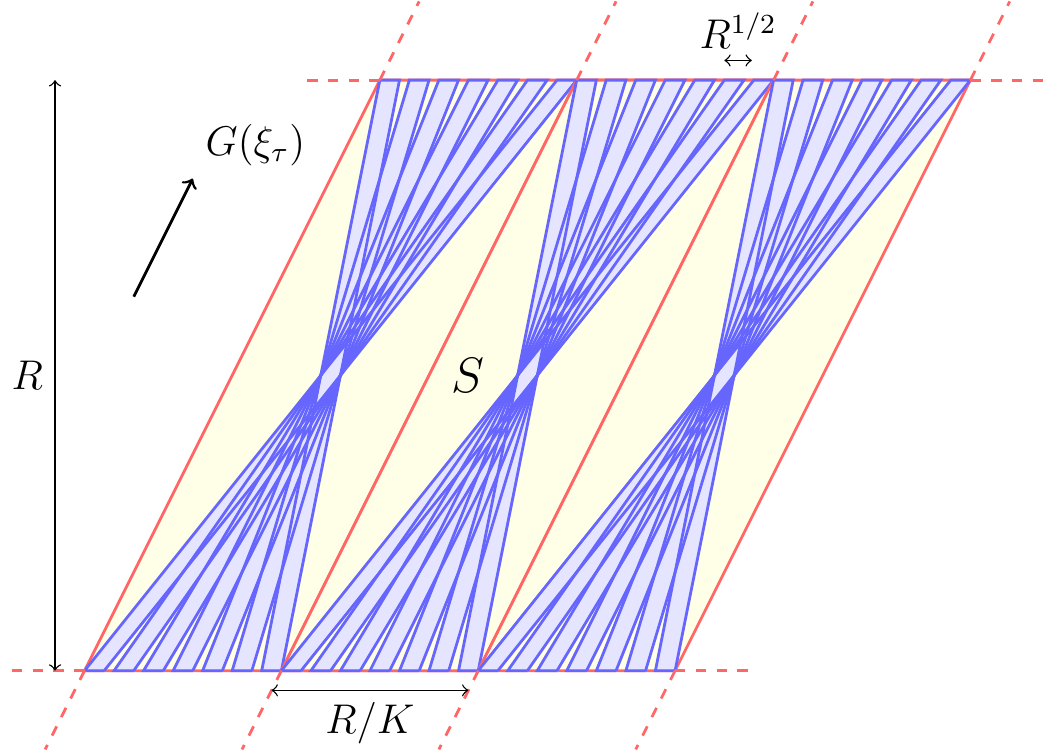}
    \caption{A schematic demonstrating the pseudolocal property. The tubes $T$ arising from the wave packet decomposition of $f_{\tau}$ have directions lying in a $K^{-1}$-neighbourhood of $G(\xi_{\tau})$. Consequently, they can (essentially) be partitioned into families which lie in disjoint strips $S$.}
    \label{fig: Pseudolocal_partition}
\end{figure}

Morally, if $S_1$, $S_2 \in \bbS_{\tau}[R]$ are distinct strips, then the waves $Uf_{S_1}$ and $Uf_{S_2}$ concentrate on $S_1$ and $S_2$, respectively, and therefore do not interact. To make this statement precise, we define the slightly enlarged strip
\begin{equation*}
    \bar{S} = \big\{ (x,t) \in \R^{n+1} : |x - x(S) - tv(S)| \leq 20 R/K \quad \textrm{and} \quad |t| \leq R \big\}.
\end{equation*}
Then we have the following lemma.  

\begin{lemm}\label{lem: pseudo loc} Suppose $0 < \delta < 1$ and $1 \leq K \leq R^{1/2 - \delta}$. With the above definitions,  for all $S \in \bbS_{\tau}[R]$ we have
\begin{equation*}
    |Uf_S(z)| \lesssim_{\delta, N} |U f_S(z)| \chi_{\bar{S}}(z) + R^{-N}\|f_{\tau}\|_{L^2(\R^n)} \qquad \textrm{for all $z \in B^{n+1}(0,R)$ and $N \in \N$.}
\end{equation*}
\end{lemm}

\begin{proof} Let $S \in \bbS_{\tau}[R]$ and $T \in \T_S[R]$. By the localisation of the wave packets from \eqref{eq: wp localisation}, it suffices to show $T^{(\delta)} \subseteq \bar{S}$, where $T^{(\delta)}$ is the enlargement of $T$ defined in \eqref{eq: enlarged T}.

Let $z = (x,t) \in B^{n+1}(0,R) \setminus \bar{S}$, so that 
\begin{equation*}
    |x - x(S) - t v(S)| > 20 R/K.
\end{equation*}
Since $T \in \T_S[R]$, it follows that $|v(T) - v(S)| \leq 4K^{-1}$ and there exists some $(x',t') \in \R^{n+1}$ satisfying
\begin{equation*}
    |x' - x(T) - t' v(T)| \leq R^{1/2}, \quad |x' - x(S) - t' v(S)| \leq R/K \quad \textrm{and} \quad |t'| \leq R.
\end{equation*}
Therefore, by the triangle inequality, $|x(T) - x(S)| \leq 6 R/K$. Consequently,
\begin{equation*}
    |x - x(T) - t v(T)| \geq  R/K > R^{1/2 + \delta}
\end{equation*}
and so $(x,t) \notin T^{(\delta)}$, as required.  
\end{proof}

\begin{rema} Note that the time zero slice $B_S := S  \cap \{(x,0) : x \in \R^n\}$ is an $R/K$-ball. Another interpretation of the preceding observations is that $S \in \bbS_{\tau}[R]$ essentially corresponds to the domain of influence for the initial datum $f_{\tau}$ localised to $B_S$, over the time interval $[-R,R]$. 
\end{rema}

For $S \in \bbS_{\tau}[R]$ define the mapping 
\begin{equation*}
    \cA_S \colon (x,t) \mapsto (K^{-1}(x - x(S) - tv(S)), K^{-2}t). 
\end{equation*}
Thus, $\cA_S$ maps the strip $S$ to the ball $B^{n+1}(0, R/K^2)$; see Figure~\ref{fig: Pseudolocal_rescaling}. Furthermore, we may write $\cA_S(z) = \cL_{\tau}(z) - z(S)$ where $\cL_{\tau}$ is the scaling map associated to the cap $\tau$ as defined in \eqref{eq: rescaling map} and $z(S) := (x(S), 0)$. In light of these observations, we obtain the following $L^p$ variant of the parabolic rescaling from Corollary~\ref{cor: parabolic rescaling}. 

\begin{figure}
       \centering
\includegraphics{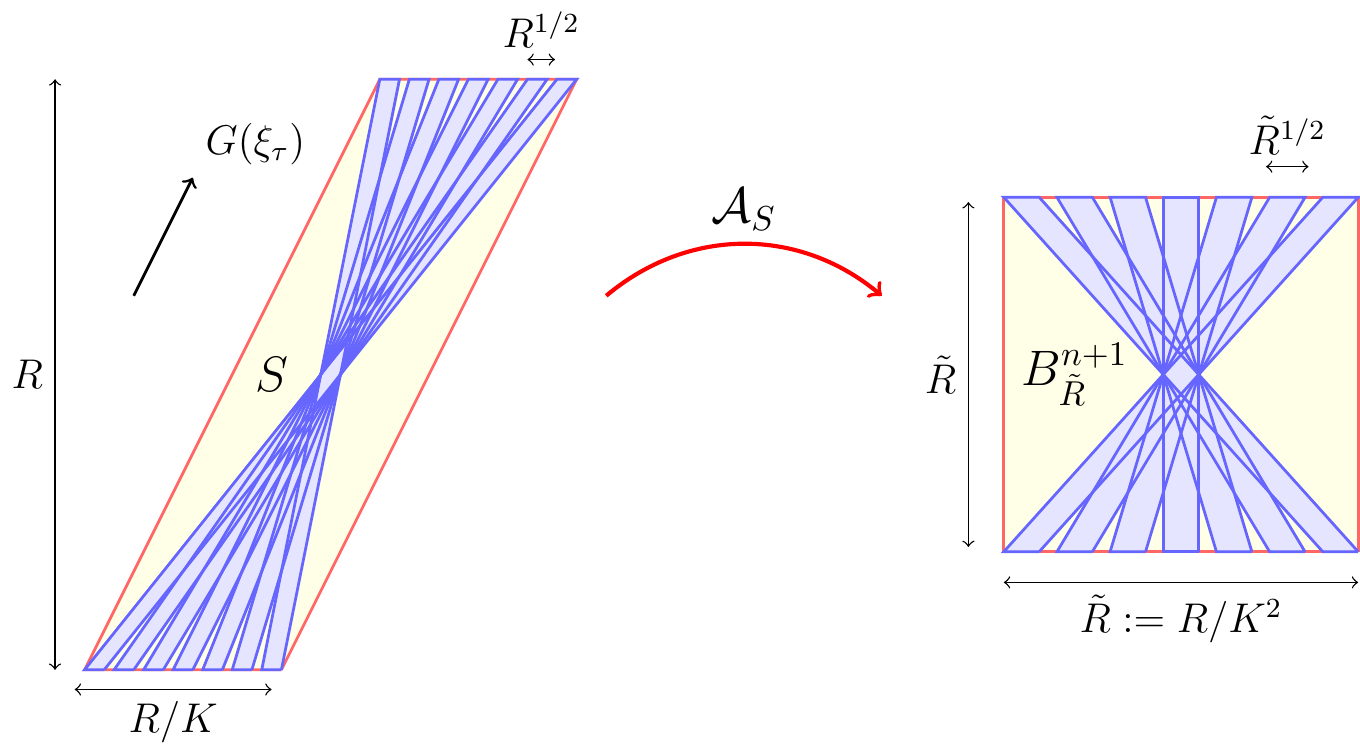}
    \caption{Rescaling in Lemma~\ref{lem: Lp rescaling} in the (physical) space-time domain. The strip $S$ is mapped to $B^{n+1}(0,\tilde{R})$ under the affine map $\cA_S$.}
    \label{fig: Pseudolocal_rescaling}
\end{figure}

 \begin{figure}
       \centering
       \includegraphics{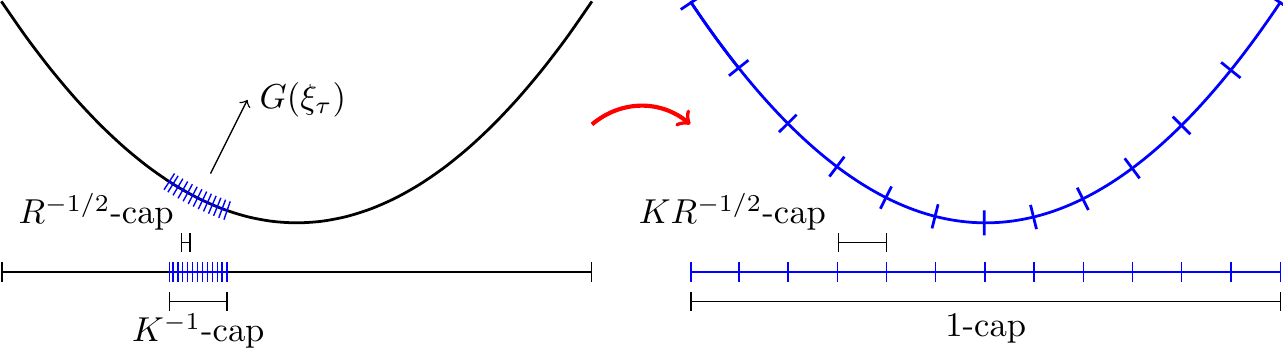}
    
    \caption{Rescaling in Lemma~\ref{lem: Lp rescaling} in the frequency domain. The $K^{-1}$-cap $\tau$ is mapped to the whole paraboloid. Note the relationship between the directions of the space-time tubes in Figure~\ref{fig: Pseudolocal_rescaling} and the underlying frequencies.}
    \label{fig: Pseudolocal_paraboloid}
\end{figure}

\begin{lemm}\label{lem: Lp rescaling} Let $1 \leq p \leq \infty$. With the above setup, 
\begin{equation*}
    \|Uf_S\|_{L^p(\bar{S})} \leq K^{(n+2)/p-n/2} \|U \tilde{f}_S\|_{L^p(B^{n+1}(0,\tilde{R}))}
\end{equation*}
for $\tilde{R} = 20 R/K^2$ and some function $\tilde{f}_S \in L^2(\R^n)$ satisfying
\begin{equation*}
\|\tilde{f}_S \|_{L^2(\R^n)} = \|f_S \|_{L^2(\R^n)} \quad \textrm{and} \quad \supp \mathcal{F}(\tilde{f}_S) \subseteq B^n(0,1).
\end{equation*}
\end{lemm}

\begin{proof} This follows by applying parabolic rescaling in the form of Corollary~\ref{cor: parabolic rescaling}, exploiting the relationship between $\cL_{\tau}$ and $\cA_S$ and changing the spatio-temporal variables. The scaling is represented in the physical and frequency domains in Figure~\ref{fig: Pseudolocal_rescaling} and Figure~\ref{fig: Pseudolocal_paraboloid}, respectively.\footnote{There are some minor technical complications in the proof of Lemma~\ref{lem: Lp rescaling} owing to slightly larger frequency support of $f_S$ (compared with $f_{\tau}$) and the use of the enlarged strip $\bar{S}$ rather than $S$. This accounts for the factor of $20$ in the definition of $\tilde{R}$. }
\end{proof}




\subsection{Dyadic pigeonholing and reverse H\"older}\label{sec: pigeonholing}

We shall make extensive use of pigeonholing arguments in the proof of Theorem~\ref{thm: fractal L2} in~\S\ref{sec: main proof}. Although such arguments are entirely elementary, they are nevertheless surprisingly useful and it is worth discussing the \textit{dyadic pigeonholing} method in particular.

\begin{defi} We say $\cB \subseteq (0, \infty)$ is \emph{dyadically constant} if there exists some $j \in \Z$ such that $\cB \subseteq [2^j, 2^{j+1}]$. 
\end{defi}

\begin{rema}
    Typically, $\cB = \{H(a) : a \in \cA\}$ is a sequence of positive numbers indexed over some finite set $\cA$; in such cases we shall often write
\begin{equation*}
    H(a) \qquad \textrm{is dyadically constant over $a \in \cA$}
\end{equation*}
to mean $\cB$ is dyadically constant.
\end{rema}

Let $M > 0$, $R \geq 1$. By taking logarithms, we see that there are only $O(\log R)$ values of $j \in \Z$ such that $2^j \in [MR^{-1}, MR]$. It follows that any set $\cB \subseteq [MR^{-1}, MR]$ can be written as a union
\begin{equation}\label{eq: pigeon decomp}
    \cB = \bigcup_{j=1}^J \cB_j
\end{equation}
where each $\cB_j$ is dyadically constant and $J \lesssim  \log R$. Applying the pigeonhole principle to the sets $\cB_j$ arising from this decomposition, we deduce the following elementary (but surprisingly powerful) lemma.

\begin{lemm}[Dyadic pigeonholing]\label{lem: pigeonholing} Let $M > 0$, $R \geq 1$ and $\cB \subseteq [MR^{-1}, MR]$ be a finite set. 
\begin{enumerate}[i)]
\item There exists some $\cB' \subseteq \cB$ which is dyadically constant and satisfies
\begin{equation*}
    \#\cB' \gtrsim (\log R)^{-1} \# \cB. 
\end{equation*}
\item More generally, given $F \colon \cB \to (0,\infty)$, there exists some $\cB_F \subseteq \cB$ which is dyadically constant and satisfies
\begin{equation*}
    \sum_{a \in \cB} F(a) \lesssim  \log R \sum_{a \in \cB_F} F(a).
\end{equation*}
\end{enumerate}
\end{lemm}

A close relative of Lemma~\ref{lem: pigeonholing} is the following reverse form of H\"older's inequality. 

\begin{lemm}[Reverse H\"older]\label{lem: rev Holder} Let $M > 0$, $R \geq 1$ and $H \colon \cA \to [MR^{-1}, MR]$ be a function defined on a finite set $\cA$. Then we may write $\cA = \bigcup_{j=1}^J \cA_j$ where 
\begin{equation*}
    \Big(\sum_{a \in \cA_j} H(a)^p \Big)^{1/p} \lesssim \log R\, [\#\cA_j]^{-(1/p - 1/q)} \Big(\sum_{a \in \cA_j} H(a)^q \Big)^{1/q}
\end{equation*}
holds for all $1 \leq p \leq q < \infty$ and all $1 \leq j \leq J$ and $J \lesssim  \log R$.
\end{lemm}

\begin{proof} Let $\cB := \{H(a) : a \in \cA\}$ and decompose this set into $O(\log R)$ disjoint dyadically constant pieces $\cB_j$, $1 \leq j \leq J$, as in \eqref{eq: pigeon decomp}. The result then follows by taking $\cA_j := H^{-1}\big(\cB_j\big)$ for $1 \leq j \leq J$. 
\end{proof}

In view of Lemma~\ref{lem: rev Holder}, the pigeonhole principle is useful for `real interpolation' arguments; that is, when one wishes to reconcile distinct estimates involving different $L^p$ norms. We shall see multiple instances of this later in \S\ref{sec: main proof}.




\section{Tools from mutlilinear harmonic analysis}\label{sec: mutlilinear harmonic analysis}




\subsection{Linear Strichartz estimates}\label{sec: strichartz}

Recall that our goal is to prove Theorem~\ref{thm: fractal L2}, which bounds the solution $Uf$ over the union $Z_{\cQ}$ of a family of space-time cubes. Here we consider a simpler class of estimates, which bound the solution over the whole space-time domain.

\begin{theo}[Strichartz estimate]\label{thm: Strichartz} Let $2 \cdot \frac{n+2}{n} \leq q \leq \infty$. The inequality
\begin{equation}\label{eq: linear Strichartz}
   \|Uf\|_{L^q(\R^{n+1})} \lesssim  \|f\|_{L^2(\R^n)}
\end{equation}
holds for all $f \in L^2(\R^n)$.
\end{theo}

This is a reinterpretation of the Stein--Tomas Fourier restriction theorem,
dating back to \textcite{Ste1984, T1975, Str1977}.

\begin{rema} Theorem~\ref{thm: Strichartz} is in fact a special case of the more general \textit{Strichartz estimates for Schr\"odinger equation}, which involve mixed norms in the spatio-temporal variables. 
\end{rema}

Theorem~\ref{thm: Strichartz} is clearly related to the theory of fractal energy estimates. Indeed, if $\cQ$ is a family of lattice unit cubes in $B^{n+1}(0,R)$ and $Z_{\cQ}$ denotes their union, then 
\begin{equation*}
    \|Uf\|_{L^2(Z_{\cQ})} \leq |Z_{\cQ}|^{1/(n+2)} \|Uf\|_{L^{2 \cdot \frac{n+2}{n}}(\R^{n+1})}
\end{equation*}
Applying the simple bound $|Z_{\cQ}| = \#\cQ \leq \Delta_{\alpha}(\cQ) R^{\alpha}$ and Theorem~\ref{thm: Strichartz}, we deduce that
\begin{equation*}
       \|Uf\|_{L^2(Z_{\cQ})} \lesssim \Delta_{\alpha}(\cQ)^{1/(n+2)} R^{\alpha/(n+2)} \|f\|_{L^2(\R^n)}.
\end{equation*}
Unfortunately, this is not strong enough to imply Theorem~\ref{thm: fractal L2}, since the $R$-exponent $\alpha/(n+2)$ is always larger than the required value $\alpha/(2(n+1))$.

\begin{rema}\label{rmk: sharp Stricharz} The range $q \geq 2 \cdot \frac{n+2}{n}$ in Theorem~\ref{thm: Strichartz} is sharp, so there is no hope of directly improving the Strichartz estimate to give the desired fractal energy estimate via the above argument. To see this, we fix a spatial scale $R \geq 1$ and a wave packet $\psi_T$ as in \S\ref{sec: wave packet examples} with $\rho = R$. Complementing the rapid decay property \eqref{eq: wave packet local}, it is not difficult to show
\begin{equation*}
    |U\psi_T(x,t)| \gtrsim R^{-n/4} \qquad \textrm{for all $(x,t) \in \tfrac{1}{100} \cdot T$.}
\end{equation*}
In particular, for any $1 \leq q \leq \infty$ we have
\begin{equation*}
\|U\psi_T\|_{L^q(B^{n+1}(0,R))} \gtrsim R^{-n/4 + (n+2)/(2q)} \qquad \textrm{and} \qquad \|\psi_T\|_{L^2(\R^n)} \lesssim 1.
\end{equation*}
Since these inequalities hold for all $R \geq 1$, it follows that \eqref{eq: linear Strichartz} fails whenever $q < 2 \cdot \frac{n+2}{n}$.
\end{rema}
 
Theorem~\ref{thm: Strichartz} is not used in the proof of Theorem~\ref{thm: fractal L2}; the linear Strichartz estimate is included here to motivate the multilinear Strichartz estimates introduced in the following section. For this reason, the full proof of Theorem~\ref{thm: Strichartz} is omitted.\footnote{See \S\ref{sec: Broad-narrow}, however, where we do prove slightly weakened version of the $n=1$ case.}




\subsection{Multilinear Strichartz estimates}

Suppose we have initial data $\psi_{T_1}$ and $\psi_{T_2}$ as in Example~\ref{ex: wave packets} which oscillate at well-separated frequencies $\xi_{T_1}$ and $\xi_{T_2}$, respectively. Since the waves $U\psi_{T_1}$ and $U\psi_{T_2}$ have distinct frequencies, we expect destructive interference between them; indeed, a rigorous manifestation of this is the orthogonality property \eqref{eq: wp orthogonality}. Furthermore, by the dispersion relation, the velocities $v(T_1)$ and $v(T_2)$ are also well-separated. Thus, $U\psi_{T_1}$ and $U\psi_{T_2}$ only interact for a short time interval. By these considerations, we expect that the product $U\psi_{T_1}U\psi_{T_2}$ is small, since it measures the interaction between the waves. 

These heuristics naturally lead to the study of multilinear Stichartz estimates for the Schr\"odinger equation. This multilinear theory is a central ingredient in the proof of Theorem~\ref{thm: fractal L2}. To introduce the main results, we first discuss the basic example of the bilinear estimate for $n=1$. 

As in \S\ref{sec: pseudo-local}, let $\cT_{K^{-1}}$ be a finitely-overlapping covering of $B^n(0,1)$ by $K^{-1}$-caps with centres lying in $B^n(0,2)$. Given $\tau \in \cT_{K^{-1}}$, we let $G(\tau) \subseteq S^n$ denote the image of $\tau$ under the Gauss map $G$ introduced in \eqref{eq: def Gauss map}. If $f_{\tau}$ is Fourier supported on $\tau$, then $G(\tau)$ is essentially the set of directions of the wavepackets in the wavepacket decomposition of $f_{\tau}$.

We consider well-separated pairs of caps $(\tau_1, \tau_2)$, corresponding to wavepackets with distinct frequencies. By the dispersion relation, we may equivalently consider pairs of caps $(\tau_1, \tau_2)$ corresponding to transverse tubes $(T_1,T_2)$. In particular, given $\xi_1$, $\xi_2 \in \widehat{\R}^n$, let $|G(\xi_1) \wedge G(\xi_2)|$ denote the absolute value of the determinant of the $2 \times 2$ matrix with $j$th column $G(\xi_j)$. We then define
\begin{equation}\label{eq: trans pairs def}
    \cT_{K^{-1}}^{\mathrm{trans}} := \big\{ (\tau_1, \tau_2) \in (\cT_{K^{-1}})^2 : |G(\tau_1) \wedge G(\tau_2)| \geq K^{-1} \big\},
\end{equation}
where 
\begin{equation*}
    |G(\tau_1) \wedge G(\tau_2)| := \inf\big\{|G(\xi_1) \wedge G(\xi_2)| : \xi_j \in \tau_j \textrm{ for $j=1,2$} \big\}.
\end{equation*}

Let $(\tau_1, \tau_2) \in \cT_{K^{-1}}^{\mathrm{trans}}$ and suppose $f_1$, $f_2 \in L^2(\R)$ satisfy $\supp \hat{f}_j \subseteq \tau_j$ for $j = 1, 2$. As in the above discussion, the functions $Uf_j$ oscillate at distinct frequencies and their constituent wave packets travel at distinct velocities. In view of this, we again expect the two waves to interact weakly. 

A rigorous manifestation of the above principle is the bilinear identity
\begin{equation}\label{eq: bilin id}
    \int_{\R^{1+1}} |Uf_1(z)Uf_2(z)|^2 \,\ud z = 2 \pi^2 \int_{\R^2} \frac{|\hat{f}_1(\xi_1)|^2|\hat{f}_2(\xi_2)|^2}{|\xi_1 - \xi_2|}\,\ud \xi_1\ud \xi_2,
\end{equation}
which holds whenever the $f_j$ satisfy the above hypothesis. This identity dates back to foundational work of \textcite{F1970}. The proof is simple: we write 
\begin{equation*}
 Uf_1(z)Uf_2(z) =   \int_{\R^2} e^{i x(\xi_1 + \xi_2) + it(\xi_1^2 + \xi_2^2)} \hat{f}_1(\xi_1) \hat{f}_2(\xi_2)\,\ud \xi_1 \ud \xi_2
\end{equation*}
and perform the change of variables $\eta_1 = \xi_1 + \xi_2$, $\eta_2 = \xi_1^2 + \xi_2^2$ to deduce 
\begin{equation*}
    Uf_1(z)Uf_2(z) = \check{F}(z) \qquad \textrm{where} \qquad F(\eta) := (2 \pi)^2 \cdot \frac{ \hat{f}_1\circ \xi_1(\eta) \hat{f}_2\circ \xi_2(\eta)}{2|\xi_1(\eta) - \xi_2(\eta)|}. 
\end{equation*}
The desired identity \eqref{eq: bilin id} now follows by an application of Plancherel's theorem and changing back the variables.

As an immediate consequence of \eqref{eq: bilin id} and interpolation with trivial $L^{\infty}$ bounds, we deduce the following \textit{bilinear} Strichartz inequality. 

\begin{prop}[1d Bilinear Strichartz]\label{prop: bilinear Strichartz} Let $4 \leq p \leq \infty$ and $(\tau_1, \tau_2) \in \cT_{K^{-1}}^{\mathrm{trans}}$. The inequality
\begin{equation}\label{eq: n=1 bilinear}
   \big\| \prod_{j=1}^2 |Uf_j|^{1/2}\big\|_{L^p(\R^2)} \lesssim K^{1/p} \prod_{j=1}^2 \|f_j\|_{L^2(\R)}^{1/2}
\end{equation}
holds whenever $f_j \in L^2(\R)$ satisfies $\supp \hat{f}_j \subseteq \tau_j$ for $j = 1$, $2$.
\end{prop}

In applications, we typically take $K = O(1)$ to be bounded (or to depend sub-polynomially on a scale parameter $R$) so that the additional $K^{1/p}$ factor is admissible. A key advantage of the bilinear setup of Proposition~\ref{prop: bilinear Strichartz}, as opposed to the $n=1$ case of Theorem~\ref{thm: Strichartz}, is that the estimate \eqref{eq: n=1 bilinear} is valid all the way down to the exponent $p = 4$. By contrast, the linear inequality is only true down to $p = 6$. The wider range of estimates in the bilinear setup reflects the principle that transverse wave packets interact weakly. In particular, the `critical behaviour' in Theorem~\ref{thm: Strichartz} arises from non-transverse interactions. 

Establishing satisfactory higher dimensional analogues of Proposition~\ref{prop: bilinear Strichartz} is a deep problem. For our purposes, we are interested in the following sharp $(n+1)$-linear Strichartz estimate in $\R^n$, which is a celebrated theorem of \textcite{BCT2006}. Generalising the definition of the set of transverse pairs form \eqref{eq: trans pairs def}, we consider the set of transverse $(n+1)$-tuples
\begin{equation*}
    \cT_{K^{-1}}^{\mathrm{trans}} := \big\{ (\tau_1, \dots,  \tau_{n+1}) \in (\cT_{K^{-1}})^{n+1} : |G(\tau_1) \wedge \cdots \wedge G(\tau_{n+1})| \geq K^{-n} \big\}.
\end{equation*}
Here $|G(\tau_1) \wedge \cdots \wedge G(\tau_{n+1})|$ denotes the infimum of $|G(\xi_1) \wedge \cdots \wedge G(\xi_{n+1})|$ over all $\xi_j \in \tau_j$, $1 \leq j \leq n+1$. The main result then reads as follows.

\begin{theo}[\cite{BCT2006}]\label{thm: BCT} Let $p_n := 2 \cdot \frac{n+1}{n}$ and $p_n \leq p \leq \infty$ and suppose $(\tau_1, \dots, \tau_{n+1}) \in \cT_{K^{-1}}^{\mathrm{trans}}$. For all $\varepsilon > 0$ and all $R \geq 1$, the inequality
\begin{equation}\label{eq: BCT}
   \big\| \prod_{j=1}^{n+1} |Uf_j|^{1/(n+1)}\big\|_{L^p(B_R^{n+1})} \lesssim_{\varepsilon} K^E R^{\varepsilon} \prod_{j=1}^{n+1} \|f_j\|_{L^2(\R^n)}^{1/(n+1)}
\end{equation}
holds whenever $f_j \in L^2(\R^n)$ satisfies $\supp \hat{f}_j \subseteq \tau_j$ for $1 \leq j \leq n+1$. Here $E$ is a dimensional constant.
\end{theo}

The technique used to prove the $n=1$ case in Proposition~\ref{prop: bilinear Strichartz} does not generalise to higher dimensions. The proof of Theorem~\ref{thm: BCT} lies beyond the scope of this exposition. A short argument, using ideas of a similar flavour to those explored in this article, is given in \textcite{G2015}.\footnote{More precisely, \textcite{G2015} establishes the (non-endpoint) \textit{multilinear Kakeya inequality}, which is equivalent to Theorem~\ref{thm: BCT} by an argument described in \textcite{BCT2006}. Alternatively, one can apply the argument of \textcite{G2015} to directly prove the multilinear Strichartz inequality by incorporating wave packet decomposition techniques.} We also refer the reader to \textcite{B2014} for an accessible introduction to this topic. 

As a direct consequence of Theorem~\ref{thm: BCT}, we have the following multilinear variant of the fractal energy estimate from Theorem~\ref{thm: DZ max}. 

\begin{coro}[Multilinear fractal energy estimate]\label{cor: multi L2 fractal}
For all $\varepsilon >  0$ and all $R \geq 1$, $1 \leq \alpha \leq n+1$ the inequality
\begin{equation}\label{eq: multi reduced max}
\big\| \prod_{j=1}^{n+1} |Uf_j|^{1/(n+1)}\big\|_{L^2(Z_{\cQ})} \lesssim_{\varepsilon} K^E\Delta_{\alpha}(\cQ)^{1/(2(n+1))}R^{\alpha/(2(n+1)) + \varepsilon} \prod_{j=1}^{n+1} \|f_j\|_{L^2(\R^n)}^{1/(n+1)}
\end{equation}
holds whenever $f_j \in L^2(\R^n)$ satisfy the hypotheses of Theorem~\ref{thm: BCT} and $\cQ$ is a family of unit lattice cubes in $B^{n+1}(0,R)$. 
\end{coro}

\begin{proof} This is a repeat of the simple argument discussed in the linear setting in \S\ref{sec: strichartz}. We apply H\"older's inequality and Theorem~\ref{thm: BCT} to deduce that
\begin{align*}
\big\| \prod_{j=1}^{n+1} |Uf_j|^{1/(n+1)}\big\|_{L^2(Z_{\cQ})}  &\leq |Z_{\cQ}|^{1/(2(n+1))} \big\| \prod_{j=1}^{n+1} |Uf_j|^{1/(n+1)}\big\|_{L^{p_n}(B^{n+1}_R)} \\ 
&\lesssim_{\varepsilon} K^E \Delta_{\alpha}(\cQ)^{1/(2(n+1))} R^{\alpha/(2(n+1)) + \varepsilon} \prod_{j=1}^{n+1} \|f_j\|_{L^2(\R^n)}^{1/(n+1)}.
\end{align*}
Here, in the second step, we use the bound $|Z_{\cQ}| = \#\cQ \leq \Delta_{\alpha}(\cQ) R^{\alpha}$, which is a direct consequence of the definitions.
\end{proof}

 If we assume, say, $K = O_{\varepsilon}(1)$, then \eqref{eq: multi reduced max} is a multilinear analogue of the desired fractal energy estimate \eqref{eq: fractal L2} from Theorem~\ref{thm: fractal L2} (and, in fact, \eqref{eq: multi reduced max} has a better dependence on $\Delta_{\alpha}(\cQ)$). Note that transversality plays a crucial r\^ole in the multilinear Strichartz estimate underpinning these observations. Indeed, if we were free to drop the transversality hypothesis in Theorem~\ref{thm: BCT} and take $f = f_1 = \cdots = f_{n+1}$, then we would arrive at the linear estimate 
\begin{equation}\label{eq: false L2}
  \|Uf\|_{L^{p_n}(B(0,R))} \lesssim_{\varepsilon} R^{\varepsilon} \|f\|_{L^2(\R^n)};
\end{equation}
however, taking $f$ to be a wave packet as defined in Example~\ref{ex: wave packets}, it is easy to see that \eqref{eq: false L2} fails (see Remark~\ref{rmk: sharp Stricharz}). Thus, the key difficulty in proving Theorem~\ref{thm: fractal L2} is to control the non-transversal interactions.




\subsection{Multilinear Bernstein inequality}

In this subsection we address some slightly technical multilinear extensions of the results discussed in \S\ref{sec: uncertainty} which will be of use in later arguments.  

\begin{lemm}[Multilinear Bernstein inequality]\label{lem: multi Bernstein} Let $1 \leq p \leq q \leq \infty$ and suppose $F_j \in \cS(\R^d)$ satisfy $\supp \hat{F}_j \subseteq Q_0 := [-1/2,1/2]^d$ for $1 \leq j \leq d$. Then 
\begin{equation*}\label{eq: multi Bernstein}
    \big\|\prod_{j=1}^d |F_j|^{1/d}\big\|_{L^q(\R^d)} \lesssim \big\|\prod_{j=1}^d|F_j|^{1/d} \big\|_{L^p(\R^d)}.
\end{equation*}
\end{lemm}

\begin{rema} By applying an affine scaling, Lemma~\ref{lem: multi Bernstein} immediately implies a generalisation of itself for functions Fourier supported in some fixed parallelepiped $\pi$. Since, for our purposes, we only require the result at unit scale, we omit the details.   
\end{rema}

There is, in fact, nothing particularly multilinear \textit{per se} about Lemma~\ref{lem: multi Bernstein}: it is a direct consequence of linear Bernstein inequalities. Indeed, under the hypothesis of Lemma~\ref{lem: multi Bernstein}, if we define $F := F_1 \cdots  F_d$, then we may equivalently express \eqref{eq: multi Bernstein} as
\begin{equation*}
    \|F\|_{L^{q/d}(\R^d)} \lesssim \|F\|_{L^{p/d}(\R^d)}.
\end{equation*}
Suppose for $p/d \geq 1$. Since $F$ has Fourier support in $Q_0 + \cdots + Q_0$, the $d$-fold Minkowski sum, in this case Lemma~\ref{lem: multi Bernstein} is a direct consequence of the linear Bernstein inequality applied to $F$. To deal with the remaining case $p/d < 1$, we establish a variant of Lemma~\ref{lem: loc const}.

\begin{lemm}\label{lem: s loc const} Let $0 < s \leq 1$ and $M \geq 1$. There exists a function $\eta_M \colon \R^d \to [0,\infty)$ satisfying the following:
\begin{enumerate}[i)]
    \item If $F \in \cS(\R^d)$ satisfies $\supp \hat{F} \subseteq Q_0 := [-1/2, 1/2]^d$, then
\begin{equation*}
    |F(z)|^s \leq \sum_{Q \in \cQ_{M, \mathrm{all}}} |a_Q|^s\chi_Q(z) \lesssim M^d |F|^s \ast \eta_M(z) \qquad \textrm{for all $z \in \R^d$}
\end{equation*}
where $\cQ_{M, \mathrm{all}}$ is the collection of all lattice $M$-cubes and 
\begin{equation*}
    a_Q := \sup_{z \in Q} |F(z)| \qquad \textrm{for all $Q \in \cQ_{M, \mathrm{all}}$.}
\end{equation*}

\item The function $\eta_M$ is $L^1$-normalised and rapidly decaying away from $[-M/2, M/2]^d$ in the sense that
\begin{equation*}
    \eta_M(z) \lesssim_{N,s} M^{-d} (1 + 2M^{-1} |z|_{\infty})^{-N} \qquad \textrm{for all $N \in \N$.}
\end{equation*}
\end{enumerate}
\end{lemm}

Once Lemma~\ref{lem: s loc const} is proved, Lemma~\ref{lem: multi Bernstein} follows easily by adapting the argument used to prove the linear Bernstein inequality in Lemma~\ref{lem: Bernstein}. Note that we only need the case $M = 1$ of Lemma~\ref{lem: s loc const} for this purpose; it is useful, however, to have the result for general $M$ in view of later applications.  

\begin{proof}[Proof (of Lemma~\ref{lem: s loc const})] Fix $\eta_0 \in \cS(\R^d)$ satisfying $\supp \hat{\eta}_0 \subseteq [-2,2]^d$ and $\hat{\eta}_0(\xi) = 1$ for all $\xi \in [-1,1]^d$ . Let $F \in \cS(\R^d)$ satisfy the hypothesis of part i). Given $Q \in \cQ_{M, \mathrm{all}}$ and $z_Q \in Q$ such that $|F(z_Q)| = a_Q$, we have
\begin{equation}\label{eq: s loc const 1}
 a_Q = |F(z_Q)| \leq \int_{\R^d} |F(w)| |\eta_0(z_Q - w)| \,\ud w. 
\end{equation}
We interpret the right-hand side as the norm $\|G_R\|_{L^1(\R^d)}$, where
\begin{equation*}
    G_R(w) := F(w)\eta_0(z_Q - w).
\end{equation*}
It is not difficult to see $\hat{G}_R$ is supported in $[-10,10]^d$. We claim that, by Bernstein's inequality,
\begin{equation}\label{eq: loc const 2}
    \|G_R\|_{L^1(\R^d)} \lesssim \|G_R\|_{L^s(\R^d)}.
\end{equation}
The issue here is that $0 < s \leq 1$ and so we cannot appeal to Lemma~\ref{lem: Bernstein} directly. However, applying Bernstein's inequality with exponents $p = 1$ and $q = \infty$ gives
\begin{equation*}
    \|G_R\|_{L^1(\R^d)} \leq \|F\|_{L^{\infty}(\R^d)}^{1 - s} \|F\|_{L^s(\R^d)}^s \lesssim \|F\|_{L^1(\R^d)}^{1-s}\|F\|_{L^s(\R^d)}^s, 
\end{equation*}
which then rearranges to produce the desired bound \eqref{eq: loc const 2}. 

In light of the above, we can upgrade \eqref{eq: s loc const 1} to
\begin{equation*}
    |a_Q|^s \lesssim \int_{\R^d} |F(w)|^s |\eta_0(z_Q - w)|^s\,\ud w.
\end{equation*}
If we now define $\eta_M \colon \R^d \to [0,\infty)$ by 
\begin{equation*}
    \eta_M(z) :=M^{-d} \sup_{|w - z|_{\infty} \leq M} |\eta_0(w)|^s,
\end{equation*}
then the desired result follows as in the proof of Lemma~\ref{lem: loc const}.
\end{proof}

We may also spatially localise the multlinear Bernstein inequality, as in Corollary~\ref{cor: local Bernstein}.

\begin{coro}\label{cor: loc multi Bernstein} Under the hypotheses of Lemma~\ref{lem: multi Bernstein}, if $Q \subseteq \R^d$ is any cube of side-length at least $1$, then 
\begin{equation*}
    \big\|\prod_{j=1}^d |F_j|^{1/d}\big\|_{L^q(Q)} \lesssim \big\|\prod_{j=1}^d|F_j|^{1/d} \big\|_{L^p(w_Q)},
\end{equation*}
where $w_Q \colon \R^d \to [0,\infty)$ is a weight adapted to $Q$ (see Definition~\ref{def: adapted weight}).  
\end{coro}

\begin{proof} This follows from Lemma~\ref{lem: multi Bernstein} via the same argument used to prove Corollary~\ref{cor: local Bernstein}.
\end{proof}




\section{Broad-narrow analysis}\label{sec: Broad-narrow}




\subsection{Motivation}

In the previous section we derived a multilinear variant of the fractal energy estimate, Corollary~\ref{cor: multi L2 fractal}, as a direct consequence of the multilinear Strichartz estimates of \textcite{BCT2006}. The problem now is to obtain \textit{bona fide} linear estimates from their multilinear counterparts. 

On the face of it, passing from multilinear to linear estimates appears challenging. Indeed, the proof of Corollary~\ref{cor: multi L2 fractal} crucially exploits the transversality hypothesis, required in order to invoke Theorem~\ref{thm: BCT}. Consequently, the methods of \S\ref{sec: mutlilinear harmonic analysis} are ill equipped to deal with interactions between resonant wave packets.

In this section we describe an ingenious method induced in \textcite{BG2011} which does in fact allow passage from multilinear to linear estimates. This is now commonly referred to as the \textit{Bourgain--Guth method} or \textit{broad-narrow analysis} and forms the backbone of many recent advances in harmonic analysis. 

Broad-narrow analysis was originally introduced to study the famous Fourier restriction conjecture, but was later adapted in \textcite{B2013} to make progress on the Carleson problem. It relies on decomposition, scaling and induction-on-scale techniques which have their roots in earlier works of \textcite{TVV1998,W2001}, amongst others.




\subsection{Broad-narrow analysis for $n=1$: an illustration}\label{sec: b/n n=1}

To begin, we illustrate the core ideas behind broad-narrow analysis in a very simple setting. In particular, we use the bilinear estimate from Proposition~\ref{prop: bilinear Strichartz} to prove the following (slightly weaker) variant of the $n=1$ case of Theorem~\ref{thm: Strichartz}. 

\begin{prop}\label{prop: 1d Strichartz} For all $\varepsilon > 0$ and all $R \geq 1$, the inequality
\begin{equation*}
   \|Uf\|_{L^6(B^{1+1}(0,R))} \lesssim_{\varepsilon}  R^{\varepsilon}\|f\|_{L^2(\R)}
\end{equation*}
holds whenever $f \in L^2(\R)$. 
\end{prop}

The following proof of Proposition~\ref{prop: 1d Strichartz} is included for illustrative purposes only: we do not require the result later in the discussion. The proof does, however, provide a simple and effective introduction to broad-narrow methods. 

The first step is to relate the linear operator $Uf$ to bilinear $Uf_1 Uf_2$ expressions involving functions $f_1$, $f_2$ satisfying the transversality hypothesis of Proposition~\ref{prop: bilinear Strichartz}. For $n=1$, this is easily achieved via an elementary pointwise inequality. 

Let $f \in L^2(\R)$ and assume, without loss of generality, that $\supp \hat{f} \subseteq B^1(0,1/2)$. We decompose 
\begin{equation*}
    f = \sum_{\tau \in \cT_{K^{-1}}} f_{\tau},
\end{equation*}
similarly to the discussion in \S\ref{sec: wave packet dec}, so that each $f_{\tau} \in L^2(\R)$ satisfies $\supp \hat{f}_{\tau} \subseteq \tau$. 

\begin{lemm}[Broad-narrow decomposition, $n=1$]\label{lem: b/n dec n=1} For all $z = (x,t) \in \R^{1+1}$, we have
\begin{equation}\label{eq: n=1 b/n dec}
|Uf(z)| \lesssim \max_{\tau \in \cT_{K^{-1}}}|Uf_{\tau}(z)| + K \max_{(\tau_1, \tau_2) \in \cT_{K^{-1}}^{\mathrm{trans}}}  \prod_{j=1}^2 |Uf_{\tau_j}(z)|^{1/2}.
\end{equation}
\end{lemm}

Here $\cT_{K^{-1}}^{\mathrm{trans}}$ denotes the collection of all $K^{-1}$-transverse pairs of caps in $\cT_{K^{-1}}$, as defined in \eqref{eq: trans pairs def}. The first term on the right-hand side of \eqref{eq: n=1 b/n dec} is referred to as the \textit{narrow} term, whilst the second is referred to as the \textit{broad} term. 

\begin{proof}[Proof (of Lemma~\ref{lem: b/n dec n=1})] Given $z = (x,t) \in \R^{1+1}$, we let $\tau_z \in \cT_{K^{-1}}$ be a choice of cap satisfying
\begin{equation*}
    |Uf_{\tau_z}(z)| = \max_{\tau \in \cT_{K^{-1}}} |Uf_{\tau}(z)|.
\end{equation*}
We define the set of \textit{narrow} and \textit{broad} caps (for $Uf$ at $z$) by
\begin{equation*}
    \mathcal{N}_z := \big\{\tau \in \cT_{K^{-1}} : |G(\tau) \wedge G(\tau_z)| < K^{-1} \big\} \quad \textrm{and} \quad \mathcal{B}_z := \cT_{K^{-1}} \setminus \mathcal{N}_z, 
\end{equation*}
respectively. Note that $\mathcal{N}_z$ simply consists of the cap $\tau_z$ and some of its neighbours.

The decomposition of $\cT_{K^{-1}}$ into broad and narrow caps induces a decomposition of the operator
\begin{equation}\label{eq: b/n 1}
|Uf(z)| = \big|\sum_{\tau \in \cT_{K^{-1}}} Uf_{\tau}(z)\big| \leq \big|\sum_{\tau \in \mathcal{N}_z} Uf_{\tau}(z)\big| +\sum_{\tau \in \mathcal{B}_z}  |Uf_{\tau}(z)|.
\end{equation}

 For the term involving narrow caps, we simply note that $\# \mathcal{N}_z\lesssim 1$ and so 
\begin{equation}\label{eq: b/n 2}
    \big|\sum_{\tau \in \mathcal{N}_z} Uf_{\tau}(z)\big| \lesssim \max_{\tau \in \cT_{K^{-1}}}|Uf_{\tau}(z)|.
\end{equation}
On the other hand, clearly $(\tau, \tau_z) \in \cT_{K^{-1}}^{\mathrm{trans}}$ for all $\tau \in \cB_z$ and so 
\begin{align}
\nonumber
   \sum_{\tau \in \mathcal{B}_z}  |Uf_{\tau}(z)| &\leq \sum_{\tau \in \mathcal{B}_z}  |Uf_{\tau}(z)|^{1/2} |Uf_{\tau_z}(z)|^{1/2} \\
\label{eq: b/n 3}
   &\lesssim K \max_{(\tau_1, \tau_2) \in \cT_{K^{-1}}^{\mathrm{trans}} } \prod_{j=1}^2 |Uf_{\tau_j}(z)|^{1/2}.
\end{align}
Plugging \eqref{eq: b/n 2} and \eqref{eq: b/n 3} into \eqref{eq: b/n 1}, we deduce the desired estimate. 
\end{proof}

As an immediate consequence of the pointwise bound in Lemma~\ref{lem: b/n dec n=1}, we deduce the following norm bound. 

\begin{coro}[$L^q$ broad-narrow decomposition, $n=1$]\label{cor: b/n dec} For all $1 \leq q \leq \infty$ and all $R \geq 1$, we have
\begin{equation}\label{eq: b/n n=1 norm}
    \|Uf\|_{L^q(B^{1+1}_R)} \lesssim \Big(\sum_{\tau \in \cT_{K^{-1}}} \|Uf_{\tau}\|_{L^q(B^{1+1}_R)}^q \Big)^{1/q} + K^3 \max_{(\tau_1, \tau_2) \in \cT_{K^{-1}}^{\mathrm{trans}} }  \Big\| \prod_{j=1}^2 |Uf_{\tau_j}|^{1/2} \Big\|_{L^q(B^{1+1}_R)}.
\end{equation}
For $q = \infty$ the $\ell^q$ expression is understood as a maximum. 
\end{coro}

\begin{proof} The $q = \infty$ case is immediate. Fixing $1 \leq q < \infty$, dominate the maxima in \eqref{eq: n=1 b/n dec} by the corresponding $\ell^q$ sum and take the $L^q$-norm in $z$ of both sides of the resulting expression. The desired result then follows by interchanging $\ell^q$ and $L^q$ norms. 
\end{proof}

To motivate what follows, we pause to consider the terms appearing on the right-hand side of \eqref{eq: b/n n=1 norm}. 
\begin{itemize}
\item The broad term involves a bilinear expression and, for appropriate $q$, can be estimated using the bilinear Strichartz estimate from Proposition~\ref{prop: bilinear Strichartz}. 
\item To estimate the narrow term, the key observation is that the expressions $\|Uf_{\tau}\|_{L^q(B^{1+1}_R)}$ appearing on the right-hand side are of the same form as the expression $\|Uf\|_{L^q(B^{1+1}_R)}$ appearing on the left-hand side. This is a symmetry (or self-similarity) of the inequality. Moreover, the functions $f_{\tau}$ appearing on the right are frequency localised versions of the function $f$ on the left; in this sense the $f_{\tau}$ are simpler objects than $f$. These considerations naturally lead to an inductive argument.  
\end{itemize}

\begin{proof}[Proof (of Proposition~\ref{prop: 1d Strichartz})]
Fix $\varepsilon > 0$. We let $\bC \geq 1$ and $K \geq 2$ denote fixed constants, depending only on the admissible parameter $\varepsilon$, which are chosen large enough to satisfy the forthcoming requirements of the proof. 

We argue by induction on the scale parameter $R$. As a simple consequence of the Cauchy--Schwarz inequality and Plancherel's theorem, we have 
\begin{equation*}
    \|Uf\|_{L^6(B^{1+1}(0,R))} \lesssim R^{1/3} \|Uf\|_{L^{\infty}(B^{1+1}(0,R))} \lesssim R^{1/3}\|\hat{f}\|_{L^2(\R)} = R^{1/3}\|f\|_{L^2(\R)}.
\end{equation*}
 Thus, Proposition~\ref{prop: 1d Strichartz} holds trivially for small scales, say $R \leq 100$; this serves as a base case for the induction. 

Fix $R \geq 100$ and $f \in L^2(\R)$ and assume, without loss of generality, that $\supp \hat{f} \subseteq B^1(0,1/2)$. We further assume the following holds. 

\begin{induction} For all $1 \leq \tilde{R} \leq R/2$, the inequality 
\begin{equation*}
\|Ug\|_{L^6(B^{1+1}(0, \tilde{R}))} \leq \bC \tilde{R}^{\varepsilon} \|g\|_{L^2(\R)} 
\end{equation*}
holds whenever $g \in L^2(\R)$. 
\end{induction}

 For $2 \leq q < \infty$ we apply the broad-narrow decomposition from Corollary~\ref{cor: b/n dec} to deduce that\footnote{Later we fix $q = 6$, but it is useful to keep the parameter $q$ free for the time being to see how the numerology works out in general.}
\begin{equation}\label{eq: Strichartz 0}
    \|Uf\|_{L^q(B^{1+1}_R)} \lesssim \Big(\sum_{\tau \in \cT_{K^{-1}}} \|Uf_{\tau}\|_{L^q(B^{1+1}_R)}^2 \Big)^{1/2} + K^3 \max_{(\tau_1, \tau_2) \in \cT_{K^{-1}}^{\mathrm{trans}} }  \Big\| \prod_{j=1}^2 |Uf_{\tau_j}|^{1/2} \Big\|_{L^q(B^{1+1}_R)}.
\end{equation}
Here we have weakened the $\ell^q$ sum to an $\ell^2$ sum in the narrow term, using the nesting of $\ell^q$ norms.\smallskip

\noindent \textit{The broad term.} We begin by estimating the broad contribution: that is, the second term on the right-hand side of \eqref{eq: Strichartz 0}. This is achieved by direct application of the bilinear Strichartz estimate from Proposition~\ref{prop: bilinear Strichartz}. In particular, by \eqref{eq: n=1 bilinear}, for $4 \leq q < \infty$ we have 
\begin{equation}\label{eq: Strichartz 1}
    \|Uf\|_{L^q(B^{1+1}_R)} \lesssim \Big(\sum_{\tau \in \cT_{K^{-1}}} \|Uf_{\tau}\|_{L^q(B^{1+1}_R)}^2 \Big)^{1/2} + K^4 \|f\|_{L^2(\R)}.
\end{equation}

\noindent \textit{The narrow term.} It remains to estimate the narrow contribution, corresponding to the first term of the right-hand side of \eqref{eq: Strichartz 0}. This is achieved via a combination of parabolic rescaling and appeal to the induction hypothesis. 

For each $\tau \in \cT_{K^{-1}}$, as in \eqref{eq: pseudo loc 1} we decompose\footnote{The decomposition according to the strips $S$ is somewhat overkill for the purposes of this argument, but is an important feature when adapting these methods to prove Theorem~\ref{thm: fractal L2} in \S\ref{sec: main proof}. An alternative approach to the current proof is to rescale $Uf_{\tau}$ directly using Corollary~\ref{cor: parabolic rescaling}.} 
\begin{equation}\label{eq: pseudo loc 1 again}
f_{\tau} = \sum_{S \in \bbS_{\tau}[R]} f_S \qquad \textrm{where} \qquad f_S := \sum_{T \in \T_S[R]} f_T.
\end{equation}
By  Lemma~\ref{lem: pseudo loc}, we have
\begin{equation}\label{eq: Strichartz 2}
    \|Uf_{\tau}\|_{L^q(B^{1+1}_R)}^q \lesssim \sum_{S \in \bbS_{\tau}[R]} \|Uf_S\|_{L^q(\bar{S})}^q + R^{-100nq} \|f\|_{L^2(\R)}^q.
\end{equation}

We now invoke parabolic rescaling in the form of Lemma~\ref{lem: Lp rescaling}. In particular, for each $S \in \bbS_{\tau}[R]$ there exists a function $\tilde{f}_{S} \in L^2(\R)$ which is Fourier supported in $B^1(0,1)$ and satisfies
\begin{equation}\label{eq: Strichartz 3}
    \|Uf_S\|_{L^q(\bar{S})} \leq K^{3/q-1/2}\|U\tilde{f}_S\|_{L^q(B^{1+1}(0,\tilde{R}))} \qquad \textrm{and} \qquad \|\tilde{f}_S\|_{L^2(\R)} = \|f_S\|_{L^2(\R)},
\end{equation}
where  $\tilde{R} := 20 R/K^2$. Provided $K$ is chosen sufficiently large, $\tilde{R} \leq R/2$, and so we may apply the induction hypothesis to conclude that
\begin{equation}\label{eq: Strichartz 4}
    \|U\tilde{f}_S\|_{L^q(B^{1+1}(0,\tilde{R}))} \leq \bC \tilde{R}^{\varepsilon} \|\tilde{f}_S\|_{L^2(\R)} \lesssim \bC K^{-2\varepsilon} R^{\varepsilon} \|f_S\|_{L^2(\R)}. 
\end{equation}
Combining \eqref{eq: Strichartz 2}, \eqref{eq: Strichartz 3} and \eqref{eq: Strichartz 4}, we therefore deduce that
\begin{equation}\label{eq: Strichartz 5}
    \Big(\sum_{\tau \in \cT_{K^{-1}}} \|Uf_{\tau}\|_{L^q(B^{1+1}_R)}^2 \Big)^{1/2} \lesssim \bC K^{3/q-1/2 -2\varepsilon} R^{\varepsilon} \Big(\sum_{S \in \bbS[R]} \|f_S\|_{L^2(\R)}^2 \Big)^{1/2} + R^{-10n} \|f\|_{L^2(\R)},
\end{equation}
where $\bbS[R]$ denotes the union of all the sets $\bbS_{\tau}[R]$ over all $\tau \in \cT_{K^{-1}}$. The families of wave packets $\T_S[R]$ appearing in the definition \eqref{eq: pseudo loc 1 again} are essentially disjoint as $S$ varies over $\bbS[R]$. Thus, by the orthogonality properties of the wave packets,
\begin{equation*}
    \Big(\sum_{S \in \bbS[R]} \|f_S\|_{L^2(\R)}^2 \Big)^{1/2} \lesssim \|f\|_{L^2(\R)}.
\end{equation*}
Applying this bound to \eqref{eq: Strichartz 5}, we deduce that
\begin{equation}\label{eq: Strichartz 6}
    \Big(\sum_{\tau \in \cT_{K^{-1}}} \|Uf_{\tau}\|_{L^q(B^{1+1}_R)}^2 \Big)^{1/2} \lesssim \bC K^{3/q-1/2 -2\varepsilon} R^{\varepsilon}\|f\|_{L^2(\R)} ,
\end{equation}
which is the final estimate for the narrow term.

To conclude the proof, we combine \eqref{eq: Strichartz 1} and \eqref{eq: Strichartz 6}; thus, for $4 \leq q \leq \infty$, we have
\begin{equation*}
    \|Uf\|_{L^q(B^{1+1}_R)} \leq C \big( \bC K^{3/q-1/2}K^{-2\varepsilon} R^{\varepsilon} + K^4 \big) \|f\|_{L^2(\R)}, 
\end{equation*}
where $C$ is suitable a choice of absolute constant, which accounts for all the factors arising from the implicit constants in the above argument. Note that the exponent $3/q - 1/2$ is non-positive precisely when $q \geq 6$. In particular, specialising to the case $q = 6$, we have
\begin{equation}\label{eq: Strichartz 7}
    \|Uf\|_{L^6(B^{1+1}_R)} \leq C \big( \bC K^{-2\varepsilon}  + K^4 \big)R^{\varepsilon} \|f\|_{L^2(\R)}.
\end{equation}

The estimate \eqref{eq: Strichartz 7} involves two free parameters (both of which must be chosen independently of $f$ and $R$):
\begin{itemize}
    \item The constant $\bC$ appearing in the induction hypothesis;
    \item The intermediate scale $K$.
\end{itemize}
We fine tune these parameters to ensure that the induction closes. We first choose $\bC$ in terms of $K$ so as to satisfy $\bC = 2 C K^4$. Thus, \eqref{eq: Strichartz 6} becomes 
\begin{equation*}
    \|Uf\|_{L^6(B^{1+1}_R)} \leq \Big( C \bC K^{-2\varepsilon} R^{\varepsilon} +  \frac{\bC}{2}\Big) \|f\|_{L^2(\R)}.
\end{equation*}
Finally, we choose $K$, depending only on $\varepsilon$, so that $C K^{-2\varepsilon} \leq 1/2$. With this choice, we conclude 
\begin{equation*}
    \|Uf\|_{L^6(B^{1+1}_R)} \leq \bC R^{\varepsilon} \|f\|_{L^2(\R)},
\end{equation*}
which closes the induction and completes the proof. 
\end{proof}

The simple argument presented above is particular to the $n=1$ case. This is due to the innocuous nature of the narrow term when $n = 1$; in higher dimensions we shall see that the narrow term is significantly more complex. 




\subsection{What is going on here?} The induction argument used in the proof of Proposition~\ref{prop: 1d Strichartz} is very neat, but perhaps obscures the mechanics of what is happening in the proof. Indeed, arguments like this can sometimes seem a little magical. To get a better sense of what is going on, it is helpful to `unpack' the induction argument and think of it as a recursive process. This tends to be messier, but can  give a better sense of how the argument works. Here we give an informal sketch of this recursive process, reinterpreting the proof from the previous subsection. 

As the process progresses, we pass through a decreasing chain of spatial scales
\begin{equation}\label{eq: scales}
    R \to R/K^2 \to R/K^4 \to \cdots \to R/K^{2N},
\end{equation}
where $N$ corresponds to the total number of steps in the recursion. At the terminal step (corresponding to the base case in the proof of Proposition~\ref{prop: 1d Strichartz}), we have reached the unit scale, and so we have, say, $R/K^{2N} \leq 100$. From this, we see that the total number of steps is roughly
\begin{equation*}
    N \sim \frac{\log R}{\log K}. 
\end{equation*}

At each step we use Corollary~\ref{cor: b/n dec} to split the norm into two parts: the broad term, which is analysed directly, and the narrow term which we continue to decompose. Thus, at each step we gain one additional piece in the decomposition so that throughout the whole process we split the norm into $N+1$ pieces. In particular, we have:
\begin{itemize}
    \item A broad term for each step of the process;
    \item A narrow term for the terminal step only. 
\end{itemize}
Since $N + 1 \lesssim \log R \lesssim R^{\varepsilon/2}$, it suffices to show that the contribution from each piece of this decomposition is $O_{\varepsilon}(R^{\varepsilon/2}\|f\|_{L^2(\R)})$.

The remaining narrow term from the terminal step is controlled using the trivial energy estimate, corresponding to the analysis of the base case in the inductive setup. On the other hand, each of the broad terms is estimated using the bilinear Strichartz inequality from Proposition~\ref{prop: bilinear Strichartz}. By applying parabolic rescaling, we can always separate the pairs of caps appearing in the broad term to ensure they are $K^{-1}$-transverse. This means that the constant arising from the bilinear estimate is uniform over all steps of the recursion.\footnote{We did not apply parabolic rescaling to the broad term directly in the inductive proof of Proposition~\ref{prop: 1d Strichartz}. However, by applying parabolic rescaling to the narrow term at step $k$ of the process, we effectively rescale the broad term at step $k+1$ (which is formed from decomposing the step $k$ narrow term). Thus, the above is indeed an accurate representation of the proof of Proposition~\ref{prop: 1d Strichartz}.}

There is one final subtlety. When carrying out the broad-narrow decomposition at each step, we incur some additional absolute constant $C_{\circ}$ in our inequality, say $C_{\circ} = 10$. On their own, these constants are harmless, but as we iterate the procedure they accumulate as powers $C_{\circ}^k$. At the terminal stage, we will have gained an additional factor of $C_{\circ}^N$: since $N$ is logarithmic in $R$, this could be catastrophic. The parameter $K$ is used to deal with this issue. In particular, by choosing $K$ sufficiently large, depending only on $\varepsilon$, we can ensure
\begin{equation*}
    C_{\circ}^N = R^{C \log C_{\circ}/\log K} \leq R^{\varepsilon/2},
\end{equation*}
which is admissible for our purposes. In short, the larger $K$, the larger the jumps between scales in \eqref{eq: scales} and, consequently, the fewer the number of steps $N$ in the recursive process. By choosing $K$ sufficiently large, we can favourably control the constant $C_{\circ}^N$ which arises through the recursive procedure.




\subsection{Pointwise broad-narrow decomposition in higher dimensions}

In the remainder of this section we explore extensions of the simple ideas introduced in \S\ref{sec: b/n n=1} to higher spatial dimensions. This provides the framework for the proof of the fractal energy estimate (Theorem~\ref{thm: fractal L2}) in the next section.\footnote{In particular, the goal is \textbf{not} to extend the proof of Proposition~\ref{prop: 1d Strichartz} (which is included for illustrative purposes only) to higher dimensions. Whilst broad-narrow analysis can be used to prove Strichartz estimates for $n \geq 2$, this approach is cumbersome in the extreme compared with, say, the original proof of Theorem~\ref{thm: Strichartz} from \textcite{T1975}. On the other hand, the tools developed here are effective when it comes to studying Theorem~\ref{thm: fractal L2} in general dimensions.}

For our purposes, the correct implementation of the broad-narrow decomposition for $n \geq 2$ turns out to be highly non-trivial and relies on the deep \textit{decoupling theory} of \textcite{BD2015}. Our first step is to generalise the simple pointwise broad-narrow decomposition from Lemma~\ref{lem: b/n dec n=1} to higher dimensions. This will in fact prove a misstep, and we shall go back and refine our estimates in the proof of Lemma~\ref{lem: b/n dec 2} below. Nevertheless, we consider the pointwise decomposition to gain an initial understanding of the problem. 

Recall in the proof of Lemma~\ref{lem: b/n dec n=1} we defined a collection of \textit{narrow caps}, which were caps with normals $G(\tau)$ aligned along some fixed $1$-dimensional subspace. More generally, given a $d$-dimensional linear subspace $V \subseteq \R^{n+1}$, we define
\begin{equation}\label{eq: V aligned caps}
    \cT_{K^{-1}}(V) := \big\{ \tau \in \cT_{K^{-1}} : |\sin \angle(G(\tau), V)| \leq C_n K^{-1} \big\}.
\end{equation}
Here $\angle(G(\tau), V)$ denotes the infimum of the angles $\angle(G(\xi), V)$ over all $\xi \in \tau$ and $C_n \geq 1$ is a dimensional constant, chosen large enough to satisfy the forthcoming requirements of the proof. With this definition, the general from of Lemma~\ref{lem: b/n dec n=1} reads thus. 

\begin{lemm}[Pointwise broad-narrow decomposition]\label{lem: b/n dec 1} For all $z = (x,t) \in \R^{n+1}$, we have 
\begin{equation}\label{eq: b/n dec}
|Uf(z)| \lesssim \max_{V \in \mathrm{Gr}(n,\R^{n+1})}\Big|\sum_{\tau \in \cT_{K^{-1}}(V)}Uf_{\tau}(z)\Big| + K^n \max_{(\tau_1, \dots, \tau_{n+1}) \in \cT_{K^{-1}}^{\, \mathrm{trans}} }  \prod_{j=1}^{n+1} |Uf_{\tau_j}(z)|^{1/(n+1)}.
\end{equation}
Here the left-hand maximum is taken over all $n$-dimensional linear subspaces in $\R^{n+1}$.
\end{lemm}

We immediate see that Lemma~\ref{lem: b/n dec 1} implies Lemma~\ref{lem: b/n dec n=1} in the $n = 1$ case. Comparing \eqref{eq: b/n dec} with \eqref{eq: n=1 b/n dec}, a significant additional complication in higher dimensions is the form of the narrow term. This involves the function
\begin{equation*}
    f_V := \sum_{\tau \in \cT_{K^{-1}}(V)}f_{\tau},
\end{equation*}
which is localised to a whole family of caps (aligned along a strip), rather than a single cap. For $n \geq 2$, the analysis of this term involves highly non-trivial tools from decoupling theory, described in \S\ref{sec: decoupling} below.

We remark that the precise form of Lemma~\ref{lem: b/n dec 1} is not used in our subsequent analysis; instead, we will rely on an $L^q$ variant introduced in Lemma~\ref{lem: b/n dec 2} below.\footnote{In contrast with the $n=1$ case, in higher dimensions our formulation of the $L^q$ broad-narrow decomposition does not directly follow from the pointwise decomposition.} Nevertheless, Lemma~\ref{lem: b/n dec 1} provides a useful conceptual stepping stone. 

\begin{proof}[Proof (of Lemma~\ref{lem: b/n dec 1})] The proof is an elaboration of the argument used to establish the $n=1$ case in Lemma~\ref{lem: b/n dec n=1}. The first step is to identify a codimension 1 subspace $V_z \in \mathrm{Gr}(n,\R^{n+1})$ which `captures' as many of the large $|Uf_{\tau}(z)|$ as possible. More precisely, define the \textit{broad part} of the operator
\begin{equation}\label{eq: gen b/n 1} 
    U_{\mathrm{Br}}f(z) := \min_{V \in \mathrm{Gr}(n,\R^{n+1})} \max_{\tau \notin \cT_{K^{-1}}(V)} |Uf_{\tau}(z)|
\end{equation}
and suppose $V_z \in \mathrm{Gr}(n,\R^{n+1})$ realises the minimum in \eqref{eq: gen b/n 1}. Then $V_z$ has the desired property, in the sense that the largest value of $|Uf_{\tau}(z)|$ for $\tau \notin \cT_{K^{-1}}(V)$ is minimised.

For slightly technical reasons, we also define 
\begin{equation*}
    \cT_{K^{-1}}(V, z) := \big\{ \tau \in \cT_{K^{-1}}(V) : |Uf_{\tau}(z)| \geq U_{\mathrm{Br}}f(z) \big\}
\end{equation*}
and further choose $V_z$ so that, of all spaces realising the minimum in \eqref{eq: gen b/n 1}, the space $V_z$ also maximises $\#\cT_{K^{-1}}(V, z)$. Once again, we can think of this condition as ensuring $V_z$ captures as many large $|Uf_{\tau}(z)|$ as possible.

From the definition, the caps $\tau \in \cT_{K^{-1}}(V_z, z)$ (or, more precisely, their normals) are aligned around the $n$-dimensional subspace $V_z$. However, they do not align around around any lower dimensional subspace. \medskip

\noindent \textbf{Claim.} There does not exist a subspace $W \subseteq \R^{n+1}$ of dimension $n-1$ such that $\cT_{K^{-1}}(V_z, z) \subseteq \cT_{K^{-1}}(W)$. \medskip

The idea is that if all the caps in $\cT_{K^{-1}}(V_z, z)$ are `captured' by a subspace~$W$ of dimension $n-1$, then we have the freedom to extend~$W$ to a subspace $V_z'$ of dimension~$n$ which captures even more large caps. But this contradicts the maximality of~$V_z$. We postpone the precise details of this argument until the end of the proof. \smallskip

Assuming the claim, it is a simple matter to conclude the proof of Lemma~\ref{lem: b/n dec 1}. In analogy with the proof of Lemma~\ref{lem: b/n dec n=1}, we define the collections of \textit{narrow} and \textit{broad} caps (for $Uf$ at $z$) by
\begin{equation*}
    \mathcal{N}_z := \cT_{K^{-1}}(V_z) \qquad \textrm{and} \qquad  \cB_z := \cT_{K^{-1}} \setminus \cT_{K^{-1}}(V_z),
\end{equation*}
respectively. By the triangle inequality and the defining properties of $V_z$, we have
\begin{align}
    \nonumber
    |Uf(z)| &\leq \big| \sum_{\tau \in \mathcal{N}_z} Uf_{\tau}(z) \big| + \sum_{\tau \in \cB_z} |Uf_{\tau}(z)|  \\
    \label{eq: gen b/n 2} 
    &\lesssim \max_{V \in \mathrm{Gr}(n,\R^{n+1})} \big| \sum_{\tau \in \cT_{K^{-1}}(V)} Uf_{\tau}(z) \big| + 
 K^n U_{\mathrm{Br}}f(z).   
\end{align}

In view of the claim, there exists an $n$-tuple of caps $(\tau_{z,1}, \dots, \tau_{z,n}) \in \cT_{K^{-1}}(V_z, z)^n$ which is $K^{-(n+1)}$-transverse in the sense that
\begin{equation}\label{eq: checking transverse}
        \big|G(\tau_{z,1}) \wedge \cdots \wedge G(\tau_{z, n}) \big| \geq K^{-(n-1)}. 
\end{equation}
To see this, simply take $(\tau_{z,1}, \dots, \tau_{z,n})$ to be a tuple which maximises the left-hand wedge product. If \eqref{eq: checking transverse} fails for this choice, then 
\begin{equation*}
    \big|G(\tau_{z,1}) \wedge \cdots \wedge G(\tau_{z, n-1}) \wedge G(\tau) \big| < K^{-(n-1)} \qquad \textrm{for all $\tau \in \cT_{K^{-1}}(V_z, z)$.}
\end{equation*}
Define $W := \mathrm{span}\{G(\xi_{\tau_{z,1}}), \dots, G(\xi_{\tau_{z, n-1}})\}$, where $\xi_{\tau_{z,j}}$ denotes the centre of $\tau_{z,j}$ for $1 \leq j \leq n-1$. If the constant $C_n \geq 1$ in \eqref{eq: V aligned caps} is chosen sufficiently large, depending only on $n$, then $W$ is an $(n-1)$-dimensional subspace satisfying $\cT_{K^{-1}}(V_z, z) \subseteq \cT_{K^{-1}}(W)$, contradicting the claim. Hence \eqref{eq: checking transverse} must hold. 

By the definition of the broad functional, there exists a cap $\tau_{z, n+1} \notin \cT_{K^{-1}}(V_z)$ such that 
\begin{equation*}
    |Uf_{\tau_{z, n+1}}(z)| = U_{\mathrm{Br}}f(z).
\end{equation*}
It follows that the $(n+1)$-tuple $(\tau_{z,1}, \dots, \tau_{z, n+1})$ is $K^{-n}$-transverse and satisfies
\begin{equation}\label{eq: gen b/n 3} 
    U_{\mathrm{Br}}f(z) \leq \prod_{j=1}^{n+1} |Uf_{\tau_{z,j}}(z)|^{1/(n+1)} \leq \max_{(\tau_1, \dots, \tau_{n+1}) \in \cT_{K^{-1}}^{\,\mathrm{trans}} }  \prod_{j=1}^{n+1} |Uf_{\tau_j}(z)|^{1/(n+1)}.
\end{equation}
The desired result now follows by combining \eqref{eq: gen b/n 3} and \eqref{eq: gen b/n 2}. \smallskip

It remains to prove the claim. Following the proof sketch, we argue by contradiction, assuming such a subspace $W$ exists. Let $\tau^* := \tau_{z, n+1}$ be as above, so that $\tau^* \notin \cT_{K^{-1}}(V_z)$ realises the maximum in the definition \eqref{eq: gen b/n 1} for $V = V_z$. Define $V_z'$ to be a subspace of dimension $n$ which contains $W$ and $G(\xi_{\tau^*})$, where $\xi_{\tau^*}$ is the centre of $\tau^*$.

First note that $V_z'$ realises the minimum in \eqref{eq: gen b/n 1}. Indeed, from the hypothesis on $W$ and the definition of $V_z'$ we have $\cT_{K^{-1}}(V_z) \subseteq \cT_{K^{-1}}(V_z')$ and so
\begin{equation*}
    U_{\mathrm{Br}}f(x) = \max_{\tau \notin \cT_{K^{-1}}(V_z)} |Uf_{\tau}(z)|\geq \max_{\tau \notin \cT_{K^{-1}}(V_z')} |Uf_{\tau}(z)| \geq U_{\mathrm{Br}}f(x). 
\end{equation*}

In light of the maximality of $V_z$, it follows that $\# \cT_{K^{-1}}(V_z', z) \leq \# \cT_{K^{-1}}(V_z,z)$. However, it is clear from the definitions that 
\begin{equation*}
    \cT_{K^{-1}}(V_z, z) \subseteq \cT_{K^{-1}}(V_z', z), \qquad \textrm{whilst} \quad \tau^* \in \cT_{K^{-1}}(V_z', z) \quad \textrm{and} \quad \tau^* \notin \cT_{K^{-1}}(V_z, z).
\end{equation*}
 Thus, $\# \cT_{K^{-1}}(V_z, z) < \# \cT_{K^{-1}}(V_z', z)$, which is a contradiction. 
\end{proof}

As in the $n=1$ case discussed in \S\ref{sec: b/n n=1}, the inequality \eqref{eq: b/n dec} is designed to access the multilinear theory from \S\ref{sec: mutlilinear harmonic analysis}. In particular, the multilinear Strichartz estimates can be used to control the `broad' contribution coming from the second term on the right-hand side of \eqref{eq: b/n dec}. What remains is to devise a method to analyse the `narrow' contribution, corresponding to the first term on the right-hand side of \eqref{eq: b/n dec}.

In the proof of 1-dimensional Strichartz estimate (Proposition~\ref{prop: 1d Strichartz}), we used a combination of parabolic rescaling and induction-on-scale to estimate the narrow contribution. This direct approach is particular to the 1-dimensional setting and significant complications arise when trying to extending these ideas to higher dimensions. 

To understand the difficulties, recall that the narrow term involves the operator $U$ applied to functions of the form
\begin{equation*}
    f_V := \sum_{\tau \in \cT_{K^{-1}}(V)} f_{\tau} \qquad \textrm{for $V \in \mathrm{Gr}(d,\R^{n+1})$.}
\end{equation*}
When $n = d = 1$, the function $f_V$ essentially corresponds to some $f_{\tau}$.\footnote{More precisely, $f_V$ is a sum of $f_{\tau}$ over a family of $O(1)$ adjacent caps, but this distinction is unimportant.} In this case, $Uf_V$ is Fourier supported on a cap on the parabola and we can exploit the scaling structure of the parabola in the guise of Lemma~\ref{lem: Lp rescaling}. However, for higher dimensions, the best we can say is that the function $Uf_V$ is supported on a parabolic strip. There is no viable scaling which maps a strip to the whole paraboloid.




\subsection{Analysis of the narrow term: decoupling}\label{sec: decoupling}

To analyse the narrow term in \eqref{eq: b/n dec}, we appeal to \textit{decoupling estimates} and the following celebrated theorem of Bourgain and Demeter. To introduce the result, let $\cQ_{K^2}$ denote the collection of space-time lattice $K^2$-cubes which intersect $B^{n+1}(0,R)$.

\begin{theo}[\cite{BD2015}]\label{thm: decoupling}  Let $1 \leq d \leq n+1$ and $2 \leq q \leq 2 \cdot \tfrac{d+1}{d-1}$. For all $\varepsilon > 0$, $K \geq 1$ we have 
\begin{equation}\label{eq: decoupling}
    \|Uf_V\|_{L^q(Q)} \lesssim_{\varepsilon} K^{\varepsilon}\Big(\sum_{\tau \in \cT_{K^{-1}}(V)} \|Uf_{\tau}\|_{L^q(w_Q)}^2 \Big)^{1/2}
\end{equation}
whenever $Q \in \cQ_{K^2}$ and $V \subseteq \R^{n+1}$ be a linear subspace of dimension $d$. 
\end{theo}

The inequality \eqref{eq: decoupling} is understood to hold for all $f \in L^2(\R^n)$ with $f_V$ and $f_{\tau}$ as defined above. The weight $w_Q$ is as in Definition~\ref{def: adapted weight}. A crucial feature of Theorem~\ref{thm: decoupling} is that the implied constant does not depend on the scale parameter $K$. 

\begin{rema} Theorem~\ref{thm: decoupling} is not explicitly stated in
  \textcite{BD2015}, but it can be easily be deduced as a consequence of Theorem~1.1 in
  \textcite{BD2015}: see the proof of Lemma~9.5 in \textcite{G2018}.    
\end{rema}

Theorem~\ref{thm: decoupling} provides an effective comparison between the function $f_V$ and its constituent parts $f_{\tau}$ for $\tau \in \cT_{K^{-1}}(V)$. Note that the right-hand side of \eqref{eq: decoupling} involves the norms $\|Uf_{\tau}\|_{L^q(w_Q)}$, which are amenable to parabolic rescaling. Thus, the Bourgain--Demeter theorem allows us to access in higher dimensions the parabolic rescaling and induction-on-scale arguments which are more readily available in the 1-dimensional case.

We will not present a proof of Theorem~\ref{thm: decoupling}; indeed, the argument of \textcite{BD2015} is lengthy and complex, incorporating all the techniques we have so far encountered (wave packet analysis, parabolic rescaling, multilinear Strichartz estimates, broad-narrow analysis, and so on). We will, however, make some elementary remarks to contextualise the result. 

We first note that there are elementary ways to compare $f_V$ with the $f_{\tau}$. One example is the triangle inequality, which can be combined with Cauchy--Schwarz to give 
\begin{equation}\label{eq: triangle ineq}
    \|Uf_V\|_{L^q(Q)} \leq \sum_{\tau \in \cT_{K^{-1}}(V)} \|Uf_{\tau}\|_{L^q(Q)} \lesssim K^{(d-1)/2} \Big(\sum_{\tau \in \cT_{K^{-1}}(V)} \|Uf_{\tau}\|_{L^q(Q)}^2 \Big)^{1/2}
\end{equation}
for all $1 \leq q \leq \infty$. However, (for large $K$) the $K^{(d-1)/2}$ factor on the right-hand side of \eqref{eq: triangle ineq} is much larger than corresponding factor on the right-hand side of \eqref{eq: decoupling}. The weakness in the triangle inequality is that it does not take into account cancellation between the terms $Uf_{\tau}$ in the sum defining $Uf_V = \sum_{\tau \in \cT_{K^{-1}}(V)} Uf_{\tau}$. Indeed, the $Uf_{\tau}$ oscillate with distinct frequencies (that is, they have disjoint Fourier support) and so it is natural to expect significant cancellation.

For $q = 2$, we can use Plancherel's theorem to exploit the disjoint frequency support and arrive at the substantially stronger estimate 
\begin{equation*}
    \|Uf_V\|_{L^2(Q)} \leq \Big(\sum_{\tau \in \cT_{K^{-1}}(V)} \|Uf_{\tau}\|_{L^2(w_Q)}^2\Big)^{1/2}.
\end{equation*}
This can then be interpolated\footnote{The decoupling inequalities are not norm inequalities for linear operators in the usual sense, so one has to be somewhat careful when it comes to interpolation. Nevertheless, it is possible to appeal to classical interpolation results by suitably interpreting the estimates.} with a trivial Cauchy--Schwarz estimate at $q = \infty$ (using the fact that $\# \cT_{K^{-1}}(V) \lesssim K^{d-1}$) to give 
\begin{equation}\label{eq: trivial dec}
    \|Uf_V\|_{L^q(Q)} \lesssim K^{(d-1)(1/2 - 1/q)} \Big(\sum_{\tau \in \cT_{K^{-1}}(V)} \|Uf_{\tau}\|_{L^q(w_Q)}^2 \Big)^{1/2}
\end{equation}
for $2 \leq q \leq \infty$. Although \eqref{eq: trivial dec} improves over \eqref{eq: triangle ineq} by taking into account orthogonality properties, it still falls far short of the decoupling estimate in \eqref{eq: decoupling}. To improve over, \eqref{eq: trivial dec} it is necessary to not only use the basic disjointness of the Fourier support, but also the specific parabolic geometry. Thus, all the techniques we have encountered in~\S\ref{sec: Basic tools} and~\S\ref{sec: mutlilinear harmonic analysis} are relevant to the proof of Theorem~\ref{thm: decoupling}.




\subsection{$L^p$ broad-narrow decomposition in higher dimensions}\label{sec: Lp b/n}

We wish to apply the decoupling estimate from Theorem~\ref{thm: decoupling} to bound the narrow term arising from our broad-narrow decomposition. However, the form of the pointwise broad-narrow decomposition from Lemma~\ref{lem: b/n dec 1} is not suited to this. The main problem is that the maximum over $V \in \mathrm{Gr}(n, \R^{n+1})$ is taken pointwise, and so the maximising subspace $V_z$ can vary from point to point. In order to apply Theorem~\ref{thm: decoupling}, we need to fix a \textit{single} subspace $V$ over a whole $K^2$-cube.

To get around this issue, we use the locally constant properties of the solution operator~$U$, dictated by the uncertainty principle. In particular, since $Uf$ (or $Uf_{\tau}$, or $Uf_V$) is frequency localised at unit scale, $|Uf|$ (or $|Uf_{\tau}|$, or $|Uf_V|$) should be locally constant at unit scale. Thus, we should be able to fix a single maximising subspace~$V$ over any given unit cube. 

Unit scale cubes are still too small to apply Theorem~\ref{thm: decoupling}, which requires cubes of side-length at least $K^2$. Nevertheless, we can adapt the above argument to work at the $K^2$ spatial scale. The idea is that $|Uf|$ will still be `locally constant up to $K^2$ factors' over $K^2$-cubes.\footnote{Indeed, on a heuristic level, this is just a consequence of the local constancy property at unit scale. To rigorously implement this principle, we shall apply Lemma~\ref{lem: s loc const} with $M=K^2$.} We then proceed as before, working with a fixed maximising subspace $V$ over any given $K^2$-cube, but including additional $K^2$ factors due to the lack of true local constancy at this scale. It is vital, however, that matters are arranged so that these additional powers of $K^2$ appear only in the broad term. Indeed, as we saw in the proof of Proposition~\ref{prop: 1d Strichartz}, we must carefully control the $K$ power in the narrow term in order for the decomposition to be effective.

Rigorous implementation of these ideas leads to the following bound. 

\begin{lemm}[$L^q$ broad-narrow decomposition]\label{lem: b/n dec 2} For all $1 \leq q \leq \infty$ and $Q \in \cQ_{K^2}$, we have 
\begin{equation}\label{eq: b/n norm}
    \|Uf\|_{L^q(Q)} \lesssim  \max_{V \in \mathrm{Gr}(n,\R^{n+1})} \|Uf_V\|_{L^q(Q)} 
     + K^E \max_{(\tau_1, \dots, \tau_{n+1}) \in \cT_{K^{-1}}^{\,\mathrm{trans}} }  \Big\| \prod_{j=1}^{n+1} |Uf_{\tau_j}|^{1/(n+1)} \Big\|_{L^{q,*}(Q)}.
\end{equation}
Here $E = E_n$ is a dimensional constant.
\end{lemm}

Here the $L^{q,*}$ expression is a `fuzzy' variant of the usual $L^q$ norm, which plays a largely unimportant technical r\^ole. In particular, let $\eta_{K^2}$ be a continuous, $L^1$-normalised function which is concentrated in $[-K^2/2,K^2/2]^{n+1}$, as in the statement of Lemma~\ref{lem: s loc const}. For  
\begin{equation*}
   \vec{\eta}_{K^2}(\vec{w}) := \eta_{K^2}(w_1) \cdots \eta_{K^2}(w_{n+1}),  \qquad \vec{w} = (w_1, \dots, w_{n+1}) \in (\R^{n+1})^{n+1},
\end{equation*}
and any choice of measurable set $S \subseteq \R^{n+1}$, we then define
\begin{equation*}
   \Big\| \prod_{j=1}^{n+1} |Uf_{\tau_j}|^{1/(n+1)} \Big\|_{L^{q,*}(S)} :=   \int_{(\R^{n+1})^{n+1}} \Big\|\prod_{j=1}^{n+1} |Uf_{\tau_j}(\,\cdot\, - w_j) |^{1/(n+1)} \Big\|_{L^q(S)} \vec{\eta}_{K^2}(\vec{w}) \,\ud \vec{w}.
\end{equation*}
 Such expressions arise due to the appearance of the mollifier in the rigorous formulation of the locally constant property from Lemma~\ref{lem: loc const} and Lemma~\ref{lem: s loc const}. In view of the translation invariance, multilinear estimates such as \eqref{eq: BCT} automatically imply `fuzzy' variants, with the left-hand $L^p$ norm replaced with the corresponding $L^{p,*}$ norm.

Before presenting the proof of Lemma~\ref{lem: b/n dec 2}, we discuss an immediate consequence. By combining the $L^q$ broad-narrow decomposition with the decoupling inequality from Theorem~\ref{thm: decoupling}, we arrive at the following broad-narrow decomposition. 

\begin{prop}\label{prop: BG dec} For all $2 \leq q \leq q_n := 2 \cdot \tfrac{n+1}{n-1}$, $Q \in \cQ_{K^2}$ and all $\varepsilon > 0$, we have 
\begin{align}
    \nonumber
    \|Uf\|_{L^q(Q)} &\lesssim_{\varepsilon} K^{\varepsilon} \Big(\sum_{\tau \in \cT_{K^{-1}}} \|Uf_{\tau}\|_{L^q(w_Q)}^2 \Big)^{1/2} \\
    \label{eq: BG final}
     & \qquad + K^E \max_{(\tau_1, \dots, \tau_{n+1}) \in \cT_{K^{-1}}^{\mathrm{trans}} }  \Big\| \prod_{j=1}^{n+1} |Uf_{\tau_j}|^{1/(n+1)} \Big\|_{L^{q,*}(Q)}.
\end{align}
Here $E = E_n$ is a dimensional constant. 
\end{prop}

\begin{rema}\label{rmk: K exp} By examining the proof below, it is possible to take $E :=  4n^2$ in both Lemma~\ref{lem: b/n dec 2} and Proposition~\ref{prop: BG dec}.
\end{rema}

 Proposition~\ref{prop: BG dec} plays a central r\^ole in the proof of Theorem~\ref{thm: fractal L2}. The narrow term is now of a similar form to that of the $n=1$ case in \eqref{eq: Strichartz 0}, and therefore amenable to parabolic rescaling and induction arguments. A crucial feature of  Proposition~\ref{prop: BG dec} is the large range  of exponents $2 \leq q \leq 2 \cdot \frac{n+1}{n-1}$ for which the estimate holds. One could attempt to dispense with the broad-narrow decomposition entirely and apply Theorem~\ref{thm: decoupling} directly with $d = n+1$. This leads to the bound 
\begin{equation*}
    \|Uf\|_{L^q(Q)} \lesssim_{\varepsilon} K^{\varepsilon} \Big(\sum_{\tau \in \cT_{K^{-1}}} \|Uf_{\tau}\|_{L^q(w_Q)}^2 \Big)^{1/2} \qquad \textrm{for $2 \leq q \leq 2 \cdot \tfrac{n+2}{n}$.}
\end{equation*}
 However, in a similar spirit to the observations of \S\ref{sec: strichartz}, the more restrictive range $2 \leq q \leq 2 \cdot \tfrac{n+2}{n}$ is insufficient for the purpose of proving Theorem~\ref{thm: fractal L2}. 
 
\begin{rema} Proposition~\ref{prop: BG dec} essentially appears in \textcite{BG2011} for the restricted range $2 \leq q \leq 2 \cdot \tfrac{n}{n+1}$. More precisely, the arguments of \textcite{BG2011} can be used to show that if the decoupling estimate \eqref{eq: decoupling} is valid for some $q\geq 2$ (and $d = n$), then \eqref{eq: BG final} holds for the same $q$. It also follows from the methods of \textcite{BG2011} that the decoupling estimate holds for $2 \leq q \leq 2 \cdot  \tfrac{n}{n+1}$ (see also \textcite{B2013a}, where the connection with decoupling is more explicit). The full range $2 \leq q \leq q_n$ for decoupling followed later in \textcite{BD2015}.
\end{rema}

\begin{proof}[Proof (of Lemma~\ref{lem: b/n dec 2})] Given a cube $Q \in \cQ_{K^2}$, we define the \textit{$L^q$ broad functional} 
\begin{equation}\label{eq: b/n p 1} 
    \|Uf\|_{{\mathrm{BL}^q(Q)}} := \min_{V \in \mathrm{Gr}(n,\R^{n+1})} \max_{\tau \notin \cT_{K^{-1}}(V)} \|Uf_{\tau}\|_{L^q(Q)}
\end{equation}
and let
\begin{equation*}
    \cT_{K^{-1}}(V,Q) := \big\{ \tau \in \cT_{K^{-1}}(V) : \|Uf_{\tau}\|_{L^q(Q)} \geq \|Uf\|_{{\mathrm{BL}^q(Q)}} \big\};
\end{equation*}
these definitions are the natural $L^q$ analogues of the pointwise definitions appearing in the proof of Lemma~\ref{lem: b/n dec 1}. 

From all spaces realising the minimum in \eqref{eq: b/n p 1}, choose $V_Q \in \mathrm{Gr}(n,\R^{n+1})$ so as to also maximise $\#\cT_{K^{-1}}(V,Q)$. Define the collections of \textit{narrow} and \textit{broad} caps (for~$Uf$ over~$Q$) by
\begin{equation*}
    \mathcal{N}_Q := \cT_{K^{-1}}(V_Q) \qquad \textrm{and} \qquad  \cB_Q := \cT_{K^{-1}} \setminus \cT_{K^{-1}}(V_Q).
\end{equation*}
Arguing as in the proof of Lemma~\ref{lem: b/n dec 1}, we then have
\begin{align}
    \nonumber
    \|Uf\|_{L^q(Q)} &\leq \big\| \sum_{\tau \in \mathcal{N}_Q} Uf_{\tau} \big\|_{L^q(Q)} + \sum_{\tau \in \cB_Q} \|Uf_{\tau}\|_{L^q(Q)}  \\
    \label{eq: b/n p 2} 
    &\lesssim \max_{V \in \mathrm{Gr}(n,\R^{n+1})} \|Uf_V\|_{L^q(Q)} + 
 K^n \|Uf\|_{{\mathrm{BL}^q(Q)}}.   
\end{align}
Furthermore,
\begin{equation}\label{eq: b/n p 3} 
    \|Uf\|_{{\mathrm{BL}^q(Q)}} \leq \max_{(\tau_1, \dots, \tau_{n+1}) \in \cT_{K^{-1}}^{\,\mathrm{trans}} }  \prod_{j=1}^{n+1} \|Uf_{\tau_j}\|_{L^q(Q)}^{1/(n+1)}.
\end{equation}

Comparing the `broad' term on the right-hand side of \eqref{eq: b/n p 3} with the corresponding term in \eqref{eq: b/n norm}, we can see that the order of the geometric mean and $L^q$ norms are interchanged. We can in fact bound the right-hand side of \eqref{eq: b/n norm} by the right-hand side of \eqref{eq: b/n p 3} simply by H\"older's inequality, but this estimate goes in the wrong direction. Thus, the problem is to prove a reverse H\"older inequality, which is achieved using the locally constant properties.

Let $\tilde{\chi} \in \cS(\R^{n+1})$ satisfy $|\tilde{\chi}(z)| \gtrsim 1$ for all $|z| \leq 2$ and $\supp \mathcal{F}\tilde{\chi} \subseteq B^{n+1}(0,1)$. Fix caps $\tau_j \in \cT_{K^{-1}}$ for $1 \leq j \leq n + 1$ and define the functions
\begin{equation*}
    F_j(z) := Uf_{\tau_j}(z) \cdot \tilde{\chi}(R^{-1}z), \quad \textrm{for $1 \leq j \leq n+1$.}
\end{equation*}
Since each cube $Q \in \cQ_{K^2}$ satisfies $Q \subseteq B^{n+1}(0,2R)$, it follows that 
\begin{equation*}
   \prod_{j=1}^{n+1} \|Uf_{\tau_j}\|_{L^q(Q)}^{1/(n+1)} \lesssim  \prod_{j=1}^{n+1} \|F_j\|_{L^q(Q)}^{1/(n+1)} \leq \Big[\prod_{j=1}^{n+1} |a_{j,Q}|^{1/(n+1)}\Big] |Q|^{1/q},
\end{equation*}
where $a_{j,Q}$ denotes the supremum of $|F_j(z)|$ over all $z \in Q$. Applying Lemma~\ref{lem: s loc const} to each function $F_j$ with exponent $s := 1/(n+1)$ and $M := K^2$, we deduce that 
\begin{align}
\nonumber
\prod_{j=1}^{n+1}\|Uf_{\tau_j}\|_{L^q(Q)}^{1/(n+1)} &\lesssim K^{2(n+1)} \Big\|\prod_{j=1}^{n+1} |F_j |^{1/(n+1)} \ast \eta_{K^2} \Big\|_{L^q(Q)} \\
\label{eq: b/n p 4} 
&\lesssim  K^{2(n+1)} \Big\| \prod_{j=1}^{n+1} |Uf_{\tau_j}|^{1/(n+1)} \Big\|_{L^{q,*}(Q)},
\end{align}
where the second step follows by Minkowski's inequality. Note that the additional factor of $K^{2(n+1)}$ arises since the $F_j$~are only Fourier localised to scale~$1$ (rather than scale~$K^{-2}$) and therefore only enjoy local constancy at unit scale. Combining~\eqref{eq: b/n p 4} with~\eqref{eq: b/n p 2} and~\eqref{eq: b/n p 3}, we deduce the desired bound. 
\end{proof}




\section{Proof of the fractal energy estimate}\label{sec: main proof}




\subsection{Recap and a final reduction}

We now combine all the tools introduced in the previous sections to prove the fractal energy estimate. For convenience, here we reproduce the statement.

\begin{theo}[\cite{DZ2019}, \textit{c.f.} Theorem~\ref{thm: fractal L2}]\label{thm: fractal L2 repeat} 
For all $\varepsilon >  0$ and all $R \geq 1$, $1 \leq \alpha \leq n+1$, the inequality
\begin{equation}\label{eq: fractal L2 repeat} 
\|Uf\|_{L^2(Z_{\cQ})} \lesssim_{\varepsilon} \Delta_{\alpha}(\cQ)^{1/(n+1)}R^{\alpha/(2(n+1)) + \varepsilon} \|f\|_{L^2(\R^n)} 
\end{equation}
holds whenever $f \in L^2(\R^n)$ and $\cQ$ is a family of lattice unit cubes in $B^{n+1}(0,R)$. 
\end{theo}

From \S\ref{sec: Standard reductions}, we know that Theorem~\ref{thm: fractal L2 repeat} implies the sharp $L^2$ maximal estimate in Theorem~\ref{thm: DZ max} and therefore also the pointwise convergence result for the Schr\"odinger equation in Theorem~\ref{thm: DZ pointwise}. 

The proof of Theorem~\ref{thm: fractal L2 repeat} hinges on the broad-narrow decomposition and induction-on-scale methods. Prior to \textcite{DZ2019}, these techniques were applied to bound the Schr\"odinger maximal function in \textcite{B2013}. The latter work establishes $H^s \to L^2$ bounds for the maximal function in the range $s > \tfrac{2n-1}{4n}$, which is more restrictive than the (essentially sharp) condition $s > \frac{n}{2(n+1)}$ from \textcite{DZ2019}. One of the main advantages of the approach of \textcite{DZ2019} is the novel form of the inductive statement, which allows a great deal of information to be translated between scales.

Rather than attempting to prove \eqref{eq: fractal L2 repeat} directly, we work with an auxiliary $L^2 \to L^{q_n}$ estimate, where $q_n := 2 \cdot \frac{n+1}{n-1}$. This is the exponent featured in the broad-narrow decomposition from Proposition~\ref{prop: BG dec} (which arises from the application of the lower-dimensional decoupling inequality from Theorem~\ref{thm: decoupling}).

\begin{prop}\label{prop: final reduction} Let $q_n := 2 \cdot \frac{n+1}{n-1}$. For all $0 < \varepsilon < 1$ and $R \geq 1$, defining $\delta := \varepsilon/100 n^2$ and $K := R^{\delta}$, the following holds. Suppose $f \in L^2(\R^n)$ and $\cQ$ is a family of lattice $K^2$-cubes contained in $B^{n+1}(0,R)$ such that\footnote{For the definition of dyadically constant, see \S\ref{sec: pigeonholing}.}
\begin{equation}\label{eq: dyadic const hyp}
    \|Uf\|_{L^{q_n}(Q)} \qquad \textrm{are dyadically constant over $Q \in \cQ$.}
\end{equation}
Then for all $1 \leq \alpha \leq n + 1$, we have
\begin{equation}\label{eq: key estimate}
\|Uf\|_{L^{q_n}(Z_{\cQ})} \lesssim_{\varepsilon}\Big[\frac{\Delta_{\alpha}(\cQ)}{\#\cQ}\Big]^{1/(n+1)} R^{\alpha/(2(n+1)) + \varepsilon} \|f\|_{L^2(\R^n)}.
\end{equation}
\end{prop}

It is not difficult to show that Proposition~\ref{prop: final reduction} implies Theorem~\ref{thm: fractal L2 repeat}. This relies on a pigeonholing argument, used to pass to the dyadically constant setup, which we isolate in Lemma~\ref{lem: pigeonholing Q} below. Before giving the details of this reduction, we indicate why the formulation of Proposition~\ref{prop: final reduction} is useful. Two major ingredients in our analysis are the multilinear Strichartz estimates from Theorem~\ref{thm: BCT} and the decoupling inequality used in the broad-narrow decomposition in Proposition~\ref{prop: BG dec}. The critical exponent for the former is $p_n := 2 \cdot \tfrac{n+1}{n}$ and for the latter is~$q_n$. Roughly speaking, the setup in Proposition~\ref{prop: final reduction} allows us to apply both these estimates at their respective critical exponents. The proposition itself is stated for the exponent~$q_n$, whilst the dyadic constancy hypothesis~\eqref{eq: dyadic const hyp} allows one to efficiently pass to the exponent~$p_n$ using reverse H\"older-type arguments (see~\S\ref{sec: broad dom} below). Thus, the pigeonholing can be thought of roughly as a weak form of real interpolation, which allows us to reconcile two very different estimates at distinct Lebesgue exponents.

\begin{lemm}\label{lem: pigeonholing Q} Let $2 \leq q \leq \infty$ and $1 \leq M \leq R$. Suppose $f \in L^2(\R^n)$ and $\cQ_M$ is a collection of lattice $M$-cubes contained in $B^{n+1}(0,R)$. Then there exists a subcollection $\cQ_M' \subseteq \cQ_M$ such that
\begin{equation}\label{eq: pigeon reduction 1}
    \|Uf\|_{L^q(Q)} \qquad \textrm{are dyadically constant over $Q \in \cQ_M'$.}
\end{equation}
 and
\begin{equation}\label{eq: pigeon reduction 2}
    \|Uf\|_{L^2(Z_{\cQ_M})} \lesssim (\log R)^{1/2}\|Uf\|_{L^2(Z_{\cQ_M'})} + R^{-100n}\|f\|_{L^2(\R^n)}.
\end{equation}
\end{lemm}

\begin{proof} Begin by bounding 
\begin{equation}\label{eq: pigeon reduction 3}
\|Uf\|_{L^2(Z_{\cQ_M})} \leq \Big(\sum_{Q \in \cQ_{M,0}} \|Uf\|_{L^2(Q)}^2 \Big)^{1/2} + \Big(\sum_{Q \in \cQ_{M,1}} \|Uf\|_{L^2(Q)}^2 \Big)^{1/2}
\end{equation}
where $\cQ_{M,0}$ and $\cQ_{M,1}$ are defined by
\begin{equation*}
    \cQ_{M,0} := \big\{ Q \in \cQ_M :  \|Uf\|_{L^q(Q)}  < R^{-200n}\|f\|_{L^2(\R^n)}\big\}, \quad \cQ_{M,1} := \cQ_M \setminus \cQ_{M,0}. 
\end{equation*}
The net contribution to \eqref{eq: pigeon reduction 3} arising from the cubes $Q \in \cQ_{M,0}$ is negligible, and can be bounded by the second term on the right-hand side of \eqref{eq: pigeon reduction 2}. 

By the definition of $\cQ_{M,0}$ and elementary estimates,
\begin{equation*}
    R^{-200n}\|f\|_{L^2(\R^n)} \leq \|Uf\|_{L^q(Q)} \lesssim R^{(n+1)/q}\|f\|_{L^2(\R^n)} \qquad \textrm{for all $Q \in \cQ_{M,1}$.}
\end{equation*}
Consequently, we may apply dyadic pigeonholing in the form of Lemma~\ref{lem: pigeonholing} ii) to find $\cQ_M' \subseteq \cQ_M$ satisfying \eqref{eq: pigeon reduction 1} and such that 
\begin{equation*}
 \Big(\sum_{Q \in \cQ_{M, 1}} \|Uf\|_{L^2(Q)}^2 \Big)^{1/2} \lesssim (\log R)^{1/2} \Big(\sum_{Q \in \cQ_M'} \|Uf\|_{L^2(Q)}^2 \Big)^{1/2},
\end{equation*}
which combines with \eqref{eq: pigeon reduction 3} and our earlier observation to give the desired bound. 
\end{proof}

\begin{proof}[Proof (Proposition~\ref{prop: final reduction} $\Rightarrow$ Theorem~\ref{thm: fractal L2 repeat})]

Let $\cQ$ be a family of lattice unit cubes in $B^{n+1}(0,R)$ and $\cQ_{K^2}$ the collection of lattice $K^2$-cubes which intersect $Z_{\cQ}$. Here $K := R^{\delta}$ is as in the statement of Proposition~\ref{prop: final reduction}.

By Lemma~\ref{lem: pigeonholing Q}, there exists a subcollection $\cQ_{K^2}' \subseteq \cQ_{K^2}$ such that 
\begin{equation*}
    \|Uf\|_{L^{q_n}(Q)} \qquad \textrm{are dyadically constant over $Q \in \cQ_{K^2}'$.}
\end{equation*}
     and
\begin{equation}\label{eq: final red 1}
   \|Uf\|_{L^2(Z_{\cQ})} \leq \|Uf\|_{L^2(Z_{\cQ_{K^2}})} \lesssim (\log R)^{1/2}\|Uf\|_{L^2(Z_{\cQ_{K^2}'})} + R^{-100n}\|f\|_{L^2(\R^n)}.
\end{equation}
We can therefore apply Proposition~\ref{prop: final reduction} to the family $\cQ_{K^2}'$ to deduce that
\begin{equation}\label{eq: final red 2}
\|Uf\|_{L^{q_n}(Z_{\cQ_{K^2}'})} \lesssim_{\varepsilon}\Big[\frac{\Delta_{\alpha}(\cQ_{K^2}')}{\#\cQ_{K^2}'}\Big]^{1/(n+1)}R^{\alpha/(2(n+1)) + \varepsilon/2} \|f\|_{L^2(\R^n)}.
\end{equation}

We now apply H\"older's inequality and combine \eqref{eq: final red 1} and \eqref{eq: final red 2} to deduce that
\begin{align*}
     \|Uf\|_{L^2(Z_{\cQ})} &\lesssim (\log R)^{1/2}|Z_{\cQ_{K^2}'}|^{1/(n+1)} \|Uf\|_{L^{q_n}(Z_{\cQ_{K^2}'})} + R^{-100n}\|f\|_{L^2(\R^n)} \\
      &\lesssim_{\varepsilon}R^{2\delta}(\log R)^{1/2} \Delta_{\alpha}(\cQ_{K^2}')^{1/(n+1)}R^{\alpha/(2(n+1)) + \varepsilon/2} \|f\|_{L^2(\R^n)},
\end{align*}
where we have used the elementary bound
\begin{equation*}
|Z_{\cQ_{K^2}'}| \leq K^{2(n+1)}[\# \cQ_{K^2}'] = R^{2(n+1)\delta}[\# \cQ_{K^2}']
\end{equation*}
Note that the factor $R^{2\delta}(\log R)^{1/2}$ can be bounded by $R^{\varepsilon/2}$.

 Finally, since each $Q_{K^2} \in \cQ_{K^2}$ has the property that $2 \cdot Q_{K^2}$ contains at least one cube from $\cQ$, it easily follows that 
$\Delta_{\alpha}(\cQ_{K^2}') \lesssim \Delta_{\alpha}(\cQ)$. Consequently, 
\begin{align*}
     \|Uf\|_{L^2(Z_{\cQ})} \lesssim_{\varepsilon}\Delta_{\alpha}(\cQ)^{1/(n+1)}R^{\alpha/(2(n+1)) + \varepsilon} \|f\|_{L^2(\R^n)},
\end{align*}
which is precisely the desired bound. 
\end{proof}




\subsection{Induction-on-scale}

The proof of Proposition~\ref{prop: final reduction} follows by induction on the scale parameter $R$. Here we introduce the underlying induction scheme.

Fix $\varepsilon > 0$, $1 \leq \alpha \leq n+1$ and set $\delta := \varepsilon/100 n^2$. We shall say a parameter is \textit{admissible} if it depends only on the dimension $n$ and $\varepsilon$. Recall from \eqref{eq: space-time energy est} that the basic energy estimate 
\begin{equation*}
    \|Uf\|_{L^2(Z_{\cQ})} \lesssim \Delta_{\alpha}(\cQ)^{1/(n+1)} R^{1/2} \|f\|_{L^2(\R^n)} 
\end{equation*}
always holds. Thus, Proposition~\ref{prop: final reduction} is trivial for small scales. In particular, let $R_0 \geq 1$ denote a fixed scale, depending only on admissible parameters $n$ and $\varepsilon$. For $\bC \geq 1$ a suitable choice of admissible constant,
\begin{equation}\label{eq: fractal Lqn} 
   \|Uf\|_{L^{q_n}(Z_{\cQ})} \leq \bC \Big[\frac{\Delta_{\alpha}(\cQ)}{\#\cQ}\Big]^{1/(n+1)}  R^{\alpha/(2(n+1)) + \varepsilon}\|f\|_{L^2(\R^n)},
\end{equation}
holds under the hypotheses of Proposition~\ref{prop: final reduction} whenever $1 \leq R \leq R_0$. This serves as the base case for our induction. 

Henceforth we assume $R_0$ and $\bC$ are fixed admissible constants, chosen so that the above base case holds and large enough to satisfy the forthcoming requirements of the proof. Let $R \geq R_0$ and define $K := R^{\delta}$. We shall work under the following induction hypothesis. 

\begin{induction} Let $1 \leq \tilde{R} \leq R/2$ and define $\tilde{K} := \tilde{R}^{\delta}$. The inequality 
\begin{equation*}
\|Ug\|_{L^{q_n}(Z_{\tilde{\cQ}})} \leq \bC \Big[\frac{\Delta_{\alpha}(\tilde{\cQ})}{\#\tilde{\cQ}}\Big]^{1/(n+1)}\tilde{R}^{\alpha/(2(n+1)) + \varepsilon} \|g\|_{L^2(\R^n)}.
\end{equation*}
holds whenever $g \in L^2(\R^n)$ and $\tilde{\cQ}$ is a non-empty family of lattice $\tilde{K}^2$-cubes contained in $B^{n+1}(0,R)$ such that
\begin{equation*}
    \|Ug\|_{L^{q_n}(\tilde{Q})}  \qquad \textrm{are dyadically constant over $\tilde{Q} \in \tilde{\cQ}$.}
\end{equation*}
\end{induction}

We now fix $f \in L^2(\R^n)$ and for $R \geq R_0$ as above let $\cQ$ be a family of $K^2$-cubes satisfying the hypotheses of Proposition~\ref{prop: final reduction}. The goal is to prove \eqref{eq: fractal Lqn}. Without loss of generality, we may assume $\supp \hat{f} \subseteq B^n(0,1/2)$.

\begin{rema}  The high-level structure of the proof is similar to the induction-on-scale argument used earlier to prove the 1-dimensional Strichartz estimate (Proposition~\ref{prop: 1d Strichartz}). There are, however, some notable differences:
\begin{itemize}
    \item The intermediate scale $K$ introduced in the statement of Proposition~\ref{prop: final reduction} plays the same r\^ole as the $K$ parameter in the proof of Proposition~\ref{prop: 1d Strichartz}. Here, however, $K = R^{\delta}$ depends on the inadmissible parameter $R$ whereas in the proof of Proposition~\ref{prop: 1d Strichartz} the parameter $K$ was chosen independently of $R$. This choice of $K$ allows us to compensate for small losses in the narrow term arising from the $K^{\varepsilon}$ factor in the decoupling inequality (Theorem~\ref{thm: decoupling}). 
    \item To close the induction, later in the argument we must ensure that $K$ is sufficiently large. Rather than fix a specific value of $K$, as in the proof of Proposition~\ref{prop: 1d Strichartz}, here we fix a lower bound for $K = R^{\delta}$ by introducing the parameter $R_0$. 
\end{itemize}
\end{rema}




\subsection{Broad / narrow dichotomy} The next step is to apply the broad-narrow decomposition described in \S\ref{sec: Lp b/n}. This will allow us to bring the powerful multilinear Strichartz estimates into play.

By Proposition~\ref{prop: BG dec}, there exists admissible constants $C_{\varepsilon}^{\mathrm{bn}}$, $E \geq 1$ such that
\begin{align}
\nonumber
    \|Uf\|_{L^{q_n}(Q)} &\leq C_{\varepsilon}^{\mathrm{bn}} K^{\varepsilon} \Big(\sum_{\tau \in \cT_{K^{-1}}} \|Uf_{\tau}\|_{L^{q_n}(w_Q)}^2 \Big)^{1/2} \\
    \label{eq: b/n dichot 1}
     & \qquad +C_{\varepsilon}^{\mathrm{bn}} K^E \max_{\tau \in \cT_{K^{-1}}^{\mathrm{trans}} }  \Big\| \prod_{j=1}^{n+1} |Uf_{\tau_j}|^{1/(n+1)} \Big\|_{L^{q_n,*}(Q)}.
\end{align}
holds for all $Q \in \cQ$. Here the maximum is taken over all transverse $(n+1)$-tuples $\tau = (\tau_1, \dots, \tau_{n+1}) \in \cT_{K^{-1}}^{\mathrm{trans}}$. By Remark~\ref{rmk: K exp} we may (and shall) take $E := 4n^2$.

 We say a cube $Q \in \cQ$ is \textit{narrow} if 
\begin{equation*}
      \|Uf\|_{L^{q_n}(Q)} \leq 2C_{\varepsilon}^{\mathrm{bn}} K^{\varepsilon} \Big(\sum_{\tau \in \cT_{K^{-1}}} \|Uf_{\tau}\|_{L^{q_n}(w_Q)}^2 \Big)^{1/2};
\end{equation*}
otherwise, we say $Q$ is \textit{broad}. As a consequence of \eqref{eq: b/n dichot 1}, the inequality 
\begin{equation}\label{eq: b/n dichot 2}
      \|Uf\|_{L^{q_n}(Q)} \leq 2C_{\varepsilon}^{\mathrm{bn}} K^E \max_{\tau \in \cT_{K^{-1}}^{\mathrm{trans}} }  \Big\| \prod_{j=1}^{n+1} |Uf_{\tau_j}|^{1/(n+1)} \Big\|_{L^{q_n,*}(Q)}
\end{equation}
holds whenever $Q \in \cQ$ is broad. 

We denote by $\cQ_{\mathrm{broad}}$ and $\cQ_{\mathrm{narrow}}$ the collections of broad and narrow cubes, respectively. The proof will split into two subcases, depending on whether the majority of the cubes are broad or narrow.




\subsection{Broad-dominant case}\label{sec: broad dom}

Suppose the majority of the cubes $Q \in \cQ$ are broad: that is, 
\begin{equation}\label{eq: broad case}
   \# \cQ_{\mathrm{broad}} \geq \# \cQ / 2 .
\end{equation}
We refer to this as the \textit{broad-dominant case}. Here we bound our operator by direct appeal to the multilinear Strichartz estimate from Theorem~\ref{thm: BCT}. This should come as no surprise, since we have already seen in Corollary~\ref{cor: multi L2 fractal} that Theorem~\ref{thm: BCT} implies multilinear fractal energy estimates. Due to the form of the desired estimate in \eqref{eq: key estimate}, we do not appeal directly to Corollary~\ref{cor: multi L2 fractal}, but the underlying idea is nevertheless the same.

Let $\lambda \geq 1$ and $\cQ' \subseteq \cQ_{\mathrm{broad}}$ be any collection of broad cubes which satisfies $\#\cQ' \geq \lambda^{-1} \#\cQ$. Then, by the dyadic constancy hypothesis \eqref{eq: dyadic const hyp}, we may bound
\begin{equation*}
    \|Uf\|_{L^{q_n}(Z_{\cQ})} \lesssim \lambda^{1/q_n} \|Uf\|_{L^{q_n}(Z_{\cQ'})} \leq \lambda \Big(\sum_{Q \in \cQ'}\|Uf\|_{L^{q_n}(Q)}^{q_n}\Big)^{1/q_n}.
\end{equation*}
Combining this with \eqref{eq: b/n dichot 2}, there exists an assignment of a transverse $(n+1)$-tuple of caps $\tau_Q = (\tau_{Q,1}, \dots, \tau_{Q,n+1}) \in \cT_{K^{-1}}^{\mathrm{trans}}$ to each $Q \in \cQ'$ such that
\begin{equation*}
  \|Uf\|_{L^{q_n}(Z_{\cQ})} \lesssim \lambda K^E  \Big(\sum_{Q \in \cQ'}  \Big\| \prod_{j=1}^{n+1} |Uf_{\tau_{Q,j}}|^{1/(n+1)} \Big\|_{L^{q_n,*}(Q)}^{q_n} \Big)^{1/q_n}.
\end{equation*}

Let $p_n := 2 \cdot \tfrac{n+1}{n}$ denote the critical exponent in the multilinear Strichartz estimate (Theorem~\ref{thm: BCT}). By the local multilinear Bernstein inequality from Corollary~\ref{cor: loc multi Bernstein}, we deduce that\footnote{The weighted $L^{p_n,*}$-norms are defined in the obvious manner; we omit the details.}
\begin{equation}\label{eq: broad dom 1}
    \|Uf\|_{L^{q_n}(Z_{\cQ})} \lesssim \lambda K^E  \Big(\sum_{Q \in \cQ'} \Big\| \prod_{j=1}^{n+1} |Uf_{\tau_{Q,j}}|^{1/(n+1)} \Big\|_{L^{p_n,*}(w_{Q})}^{q_n}\Big)^{1/q_n}.
\end{equation}

One way to estimate the $\ell^{q_n}$ sum on the right-hand side of \eqref{eq: broad dom 1} is to simply use the nesting of $\ell^p$-norms to deduce that
\begin{equation}\label{eq: broad dom 2}
\Big(\sum_{Q \in \cQ'} \Big\| \prod_{j=1}^{n+1} |Uf_{\tau_{Q,j}}|^{1/(n+1)} \Big\|_{L^{p_n,*}(w_{Q})}^{q_n}\Big)^{1/q_n} \lesssim K^E \max_{\tau \in \cT_{K^{-1}}^{\mathrm{trans}} }  \Big\| \prod_{j=1}^{n+1} |Uf_{\tau_j}|^{1/(n+1)} \Big\|_{L^{p_n,*}(w_{B_R})}.
\end{equation}
Here we have interchanged the $\ell^{p_n}$-norm and maximum in~$\tau$ at the expense of an additional factor of $K^E$. The multilinear expression can now be bounded using Theorem~\ref{thm: BCT}. However, the resulting estimates are insufficient for our purpose and we shall instead use the flexibility to choose~$\cQ'$ to improve~\eqref{eq: broad dom 2}. In particular, we choose~$\cQ'$ via a pigeonholing argument, along the lines of the reverse H\"older inequality from Lemma~\ref{lem: rev Holder}. 

Pigeonholing will repeatedly feature in the forthcoming arguments; it is convenient to introduce asymptotic notation to suppress the resulting logarithmic factors. 

\begin{notation} Let $A$, $B$ be non-negative real numbers. We write $A \lessapprox B$ or $B \gtrapprox A$ if for all $\eta > 0$ there exists a constant $C_{\varepsilon, \eta} \geq 1$, depending only on $\eta$ and the admissible parameters $n$ and $\varepsilon$, such that $A \leq C_{\varepsilon, \eta} R^{\eta} B$.
\end{notation}

Let $\cQ_{\mathrm{tiny}}$ denote the collection of all cubes $Q \in \cQ_{\mathrm{broad}}$ such that 
\begin{equation*}
    \Big\| \prod_{j=1}^{n+1} |Uf_{\tau_{Q,j}}|^{1/(n+1)} \Big\|_{L^{p_n,*}(w_{Q})} \leq R^{-100n}\|f\|_{L^2(\R^n)}.
\end{equation*}
If $\#\cQ_{\mathrm{tiny}} \geq \# \cQ_{\mathrm{broad}}/2$, then we may apply \eqref{eq: broad dom 1} with $\cQ' := \cQ_{\mathrm{tiny}}$ to obtain a very favourable estimate. Thus, we assume $\#\cQ_{\mathrm{tiny}} < \# \cQ_{\mathrm{broad}}/2$. 

Note that
\begin{equation*}
  R^{-100n}\|f\|_{L^2(\R^n)} <  \Big\| \prod_{j=1}^{n+1} |Uf_{\tau_{Q,j}}|^{1/(n+1)} \Big\|_{L^{p_n,*}(w_Q)} \lesssim R\|f\|_{L^2(\R^n)}
\end{equation*}
for all $Q \in \cQ_{\mathrm{broad}}\setminus \cQ_{\mathrm{tiny}}$, where the upper bound is a trivial consequence of the Cauchy--Schwarz inequality and Plancherel's theorem. Thus, by dyadic pigeonholing, there exists some $\cQ' \subseteq \cQ_{\mathrm{broad}}$ satisfying  $\# \cQ' \gtrapprox \# \cQ$ such that
\begin{equation*}
    \Big\| \prod_{j=1}^{n+1} |Uf_{\tau_{Q,j}}|^{1/(n+1)} \Big\|_{L^{p_n,*}(w_Q)} \qquad \textrm{are dyadically constant over $Q \in \cQ'$.} 
\end{equation*}
For this choice of set $\cQ'$, we can upgrade \eqref{eq: broad dom 2} to the reverse H\"older inequality
\begin{align}
\nonumber
\Big(\sum_{Q \in \cQ'} \Big\| \prod_{j=1}^{n+1} |Uf_{\tau_{Q,j}}|^{1/(n+1)} & \Big\|_{L^{p_n,*}(w_Q)}^{q_n}\Big)^{1/q_n} \\
\label{eq: broad dom 3}
&\lessapprox K^E [\#\cQ]^{-1/(2(n+1))}   \max_{\tau \in \cT_{K^{-1}}^{\mathrm{trans}} }\Big\| \prod_{j=1}^{n+1} |Uf_{\tau_j}|^{1/(n+1)} \Big\|_{L^{p_n,*}(w_{B_R})}.
\end{align}

 Combining \eqref{eq: broad dom 1} (for the choice of $\cQ'$ above) and \eqref{eq: broad dom 3}, we deduce that
 \begin{equation*}
     \|Uf\|_{L^{q_n}(Z_{\cQ})} \lessapprox K^{2E}    [\#\cQ]^{-1/(2(n+1))} \max_{\tau \in \cT_{K^{-1}}^{\mathrm{trans}}}  \Big\| \prod_{j=1}^{n+1} |Uf_{\tau_j}|^{1/(n+1)} \Big\|_{L^{p_n,*}(w_{B_R})}.
 \end{equation*}
Now applying the multilinear Strichartz estimate,\footnote{Here we are estimating a weighted norm $L^{p_n}(w_{B_R})$ rather than a $L^{p_n}(B_R)$-norm with sharp cut-off as in Theorem~\ref{thm: BCT}. However, Theorem~\ref{thm: BCT} automatically extends to the weighted case, using the translation invariance of the estimate and rapid decay of the weight.} 
\begin{equation*}
   \|Uf\|_{L^{q_n}(Z_{\cQ})} \lessapprox K^{3E}    [\#\cQ]^{-1/(n+1)} [\#\cQ]^{1/(2(n+1))} R^{\varepsilon/2}\|f\|_{L^2(\R^n)}. 
\end{equation*}
Recall that $\#\cQ \leq  \Delta_{\alpha}(\cQ)R^{\alpha}$, and therefore
\begin{equation*}
   \|Uf\|_{L^{q_n}(Z_{\cQ})} \lessapprox K^{3E}     [\#\cQ]^{-1/(n+1)}\Delta_{\alpha}(\cQ)^{1/(2(n+1))}  R^{\alpha/(2(n+1)) + \varepsilon/2}\|f\|_{L^2(\R^n)}. 
\end{equation*}
 Finally, since $\Delta_{\alpha}(\cQ) \gtrsim K^{-2\alpha}$ and $K = R^{\delta}$ where $\delta \leq \varepsilon/8E$, we conclude that  
\begin{equation}\label{eq: broad dom final est}
   \|Uf\|_{L^{q_n}(Z_{\cQ})} \lesssim_{\varepsilon} \Big[\frac{\Delta_{\alpha}(\cQ)}{\#\cQ}\Big]^{1/(n+1)}  R^{\alpha/(2(n+1)) + \varepsilon}\|f\|_{L^2(\R^n)}.
\end{equation}
 This provides a favourable estimate in the broad-dominant case. 




\subsection{Narrow-dominant case: introduction}

Now suppose \eqref{eq: broad case} fails, so that
\begin{equation*}
   \# \cQ_{\mathrm{narrow}} \geq \# \cQ / 2 
\end{equation*}
We refer to this as the \textit{narrow-dominant case}. Here we estimate our operator using a combination of parabolic rescaling and appeal to the induction hypothesis. 

The analysis of the narrow-dominant case is a major innovation of \textcite{DZ2019}. Indeed, as remarked earlier, \textcite{B2013} introduced the broad/narrow dichotomy to the study of the Schr\"odinger maximal operator. Many aspects of the arguments of \textcite{B2013,DZ2019} are similar (the use of induction-on-scale,\footnote{The argument of \textcite{B2013} is presented as a recursive process rather than an induction, but this is tantamount to the same thing.}  broad/narrow dichotomy, decoupling-type estimates, multilinear Strichartz); however, the novel form of the key estimate \eqref{eq: key estimate} of \textcite{DZ2019} allows information to be efficiently passed between scales and, consequently, leads to a much tighter bound in the narrow-dominant case.

We shall discuss the narrow-dominant case at length, and attempt to develop both a heuristic and technically detailed understanding of the argument. In particular, the following subsections are structured as follows:

\begin{itemize}
    \item In \S\ref{sec: narrow initial} we describe the initial steps of the analysis of the narrow-dominant case, setting the scene for the main argument. 
    \item In \S\ref{sec: narrow non-technical} we provide a non-technical overview of the main argument. For this overview, we make a number of simplifying assumptions.
    \item In \S\ref{sec: narrow technical} we provide a rigorous description of main argument. The goal here is primarily to remove the simplifying assumptions used in the previous subsection.     
\end{itemize}




\subsection{Narrow-dominant case: initial steps}\label{sec: narrow initial} 
The first few steps of the argument mirror those of the inductive proof of the 1-dimensional Strichartz estimate (Proposition~\ref{prop: 1d Strichartz}). For $\tau \in \cT_{K^{-1}}$ recall that, since $f_{\tau}$ has frequency localisation to the cube $\tau$, the wave $Uf_{\tau}$ behaves pseudo-locally. In particular, we may write 
\begin{equation*}
    f_{\tau} = \sum_{S \in \bbS_{\tau}[R]} f_S
\end{equation*}
as in \eqref{eq: pseudo loc 1} where, by Lemma~\ref{lem: pseudo loc}, the pointwise inequality
\begin{equation*}
    |Uf_{\tau}(z)| \lesssim \sum_{S \in \bbS_{\tau}[R]} |Uf_S(z)|\chi_{\bar{S}}(z) + R^{-10n} \|f\|_{L^2(\R^n)}
\end{equation*}
holds for all $z \in B^{n+1}(0,R)$.

Fix a narrow cube $Q \in \cQ_{\mathrm{narrow}}$. Since the sets $\bar{S}$ have bounded overlap,
\begin{equation}\label{eq: narrow 1}
    \|Uf_{\tau}\|_{L^{q_n}(w_Q)} \lesssim \Big(\sum_{S \in \bbS_{\tau}[R]} \|Uf_S\|_{L^{q_n}(w_Q)}^{q_n} \Big)^{1/q_n} + R^{-5n} \|f\|_{L^2(\R^n)}.
\end{equation}
Let $\bbS$ denote the (disjoint) union of the sets $\bbS_{\tau}[R]$ over all $\tau \in \cT_{K^{-1}}$. By \eqref{eq: narrow 1}, the definition of $\cQ_{\mathrm{narrow}}$ and the nesting of the $\ell^p$ spaces,
\begin{align}
\nonumber
 \|Uf\|_{L^{q_n}(Q)} &\lesssim_{\varepsilon} K^{\varepsilon}  \Big(\sum_{\tau \in \cT_{K^{-1}}} \|Uf_{\tau}\|_{L^{q_n}(w_Q)}^2\Big)^{1/2}
 \label{eq: narrow 2}
 \\ &\lesssim_{\varepsilon} K^{\varepsilon} \Big(\sum_{S \in \bbS} \|Uf_S\|_{L^{q_n}(w_Q)}^2 \Big)^{1/2} +  R^{-5n} \|f\|_{L^2(\R^n)}.
\end{align}

This situation looks very similar to the analysis of the narrow case in the proof of Proposition~\ref{prop: 1d Strichartz}. The key difference, however, is that we must now keep track of the localisation to the various cubes $Q \in \cQ$. 




\subsection{Analysis of the narrow case: non-technical overview}\label{sec: narrow non-technical} 

We now give a non-technical overview of the remainder of the proof. Any outstanding technical details are discussed in the following subsection. 

Since the weight $w_Q$ is rapidly decaying away from $Q$, the only strips $S$ which significantly contribute to the sum in \eqref{eq: narrow 2} are those belonging to 
\begin{equation*}
    \bbS(Q) := \{S \in \bbS : S \cap Q \neq \emptyset\}.
\end{equation*}
Thus, we essentially have\footnote{For the sake of this discussion, we will ignore rapidly decaying error terms in the estimates and assume we have sharp cutoffs rather than weighted $L^q$ norms. We address such technical points in the following subsection.}
\begin{equation*}
 \|Uf\|_{L^{q_n}(Q)} \lesssim_{\varepsilon} K^{\varepsilon} \Big(\sum_{S \in \bbS(Q)} \|Uf_S\|_{L^{q_n}(Q)}^2 \Big)^{1/2}.
\end{equation*}
We remark that, for any fixed $\tau \in \cT_{K^{-1}}$, the collection $\bbS(Q)$ contains only $O(1)$ strips lying in $\bbS_{\tau}[R]$. We may therefore think of $\bbS(Q)$ as a collection of strips passing through~$Q$ which are oriented in $K^{-1}$-separated directions. 

 We wish to sum the contributions over all $Q \in \cQ$. To do this effectively, we apply H\"older's inequality to convert the $\ell^2$ expression into an $\ell^q$ expression: 
\begin{equation*}
   \|Uf\|_{L^{q_n}(Q)} \lesssim_{\varepsilon} K^{\varepsilon} [\#\bbS(Q)]^{1/(n+1)} \Big(\sum_{S \in \bbS} \|Uf_S\|_{L^{q_n}(Q)}^{q_n} \Big)^{1/q_n}. 
\end{equation*}
 We may now take the $\ell^q$ sum over all $Q \in \cQ_{\mathrm{narrow}}$ to deduce that 
\begin{equation}\label{eq: narrow heuristic 1}
  \|Uf\|_{L^{q_n}(Z_{\cQ})} \lesssim_{\varepsilon} K^{\varepsilon} [\max_{Q \in \cQ} \#\bbS(Q)]^{1/(n+1)} \Big(\sum_{S \in \bbS} \|Uf_S\|_{L^{q_n}(Z_{\cQ})}^{q_n} \Big)^{1/q_n}.
\end{equation}




\subsubsection*{Parabolic rescaling and the induction hypothesis}

At this stage, we wish to apply parabolic rescaling and the induction hypothesis to bound each of the terms $\|Uf_S\|_{L^q(Z_{\cQ})}$. This is exactly as in the inductive proof of Proposition~\ref{prop: 1d Strichartz}. However, one major complication in the present setup is that our induction hypothesis involves some underlying family of $\tilde{K}^2$-cubes $\tilde{\cQ}$. We must therefore prepare the ground so that, after rescaling, a suitable family of $\tilde{K}^2$-cubes arises.

Fix $S \in \bbS$, so that
\begin{equation*}
    S = \big\{ (x,t) \in \R^{n+1} : |x - x(S) - t v(S)| \leq  R/K \textrm{ and } |t| \leq R \big\}
\end{equation*}
for some choice of initial position $x(S) \in B^n(0,R)$ and velocity $v(S) \in B^n(0,1)$. As in \S\ref{sec: pseudo-local}, let $\cA_S \colon \R^{n+1} \to \R^{n+1}$ denote the affine transformation
\begin{equation*}
    \cA_S \colon (x,t) \mapsto \big(K^{-1}(x - x(S) - tv(S)), K^{-2} t\big)
\end{equation*}
which maps bijectively between $S$ and $B^{n+1}(0,R / K^2)$. Define
\begin{equation}\label{eq: new scales}
    \tilde{R} := 20 R / K^2 \quad \textrm{and} \quad \tilde{K} := \tilde{R}^{\delta}.
\end{equation}

Let $\cP(S)$ denote a cover of the strip $S$ by parallelepipeds  aligned parallel to $S$ and of dimensions $K \tilde{K}^2 \times \cdots \times K \tilde{K}^2 \times K^2 \tilde{K}^2$. In particular, $\cP(S)$ consists of the sets
\begin{equation}\label{eq: parallelepipeds}
    P := \big\{ (x,t) \in \R^{n+1} : |x - x(P) - t v(S)|_{\infty} \leq K \tilde{K}^2/2 \textrm{ and } |t - t(P)| \leq K^2 \tilde{K}^2/2 \big\},
\end{equation}
where the centres $z(P) = (x(P), t(P))$ vary over the lattice points $K \tilde{K}^2 \Z^n \times K^2 \tilde{K}^2 \Z$. The dimensions of these sets are chosen so that the transformation $\cA_S$ maps each parallelepiped $P \in \cP(S)$ to a lattice $\tilde{K}^2$-cube
\begin{equation*}
    \cA_S(P) =  \{ (\tilde{x},\tilde{t}) \in \R^{n+1} : |\tilde{x} - K^{-1}(x(S) - x(P))|_{\infty} \leq \tilde{K}^2/2 \textrm{ and } |\tilde{t} - K^{-2}t(P)| \leq \tilde{K}^2/2 \big\};
\end{equation*}
see Figure~\ref{fig: Rescaling_parallelepipeds}.

\begin{figure}
     \centering
\includegraphics{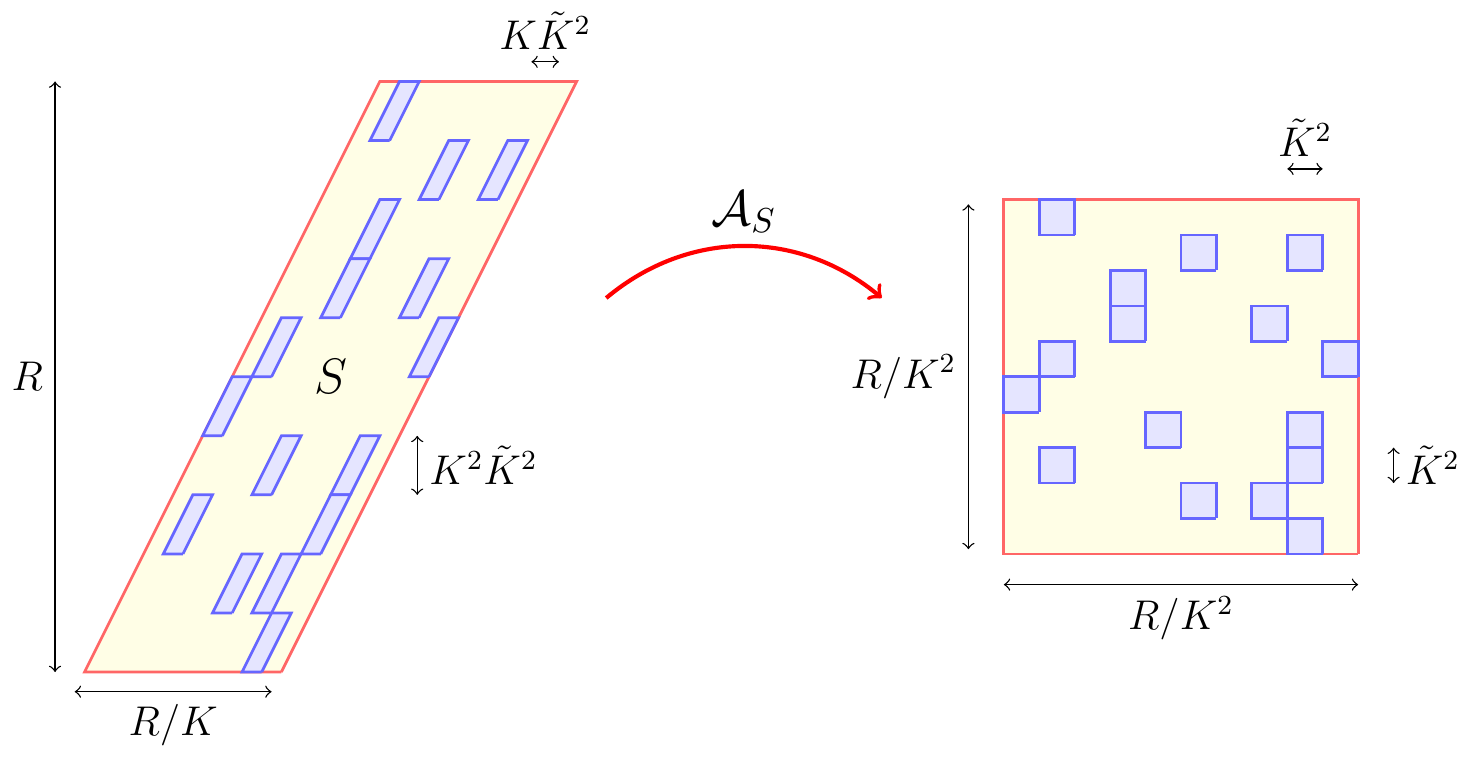}
        \caption{The map $\cA_S$ sends $S$ to $B^{n+1}(0, R/K^2)$ and each parallelepiped $P$ to a lattice $\tilde{K}^2$-cube. Note that the left and right-hand sides are drawn at different scales (the right-hand ball is in fact much smaller than the strip). }
    \label{fig: Rescaling_parallelepipeds}
\end{figure}

Since we are forming our $L^q$-norms over the set $Z_{\cQ}$, it suffices to only consider the subcollection $\cP(S;\cQ)$ of all parallelepipeds $P \in \cP(S)$ which intersect some cube $Q \in \cQ$. By the preceding discussion,
\begin{equation}\label{eq: resc cubes}
    \tilde{\cQ}(S) := \big\{ \cA_S(P) : P \in \cP(S;\cQ) \big\}  
\end{equation}
is a collection of lattice $\tilde{K}^2$-cubes in $B^{n+1}(0, \tilde{R})$.

Since the set $S \cap Z_{\cQ}$ is contained in \begin{equation*}
    Z_{\cP(S;\cQ)} := \bigcup_{P \in \cP(S;\cQ)} P,
\end{equation*}
and by Lemma~\ref{lem: pseudo loc}, the function $Uf_S$ is essentially supported on $S$, we may estimate
\begin{equation}\label{eq: narrow heuristic 2}
    \|Uf_S\|_{L^{q_n}(Z_{\cQ})} \lesssim \|Uf_S\|_{L^{q_n}(Z_{\cP(S;\cQ)})}.
\end{equation}
which combines with \eqref{eq: narrow heuristic 1} to give
\begin{equation}\label{eq: narrow heuristic 2.5}
  \|Uf\|_{L^{q_n}(Z_{\cQ})} \lesssim_{\varepsilon} K^{\varepsilon} [\max_{Q \in \cQ} \#\bbS(Q)]^{1/(n+1)} \Big(\sum_{S \in \bbS} \|Uf_S\|_{L^{q_n}(Z_{\cP(S;\cQ)})}^{q_n} \Big)^{1/q_n}.
\end{equation}

In order to apply the induction hypothesis later in the argument, it is useful to make the following assumption.\smallskip

\noindent{\textbf{Simplifying Assumption P1.}} For each $S \in \bbS$, we assume that 
\begin{equation}\label{eq: simplify 1}
    \|Uf_S\|_{L^{q_n}(P)} \qquad \textrm{are dyadically constant over $P \in \cP(S; \cQ)$.}
\end{equation}

For the purpose of this non-technical discussion, we shall make a number of such simplifying assumptions. Later, in \S\ref{sec: narrow non-technical}, we give a rigorous justification of these assumptions using pigeonholing. 

 We now rescale the norm on the right-hand side of \eqref{eq: narrow heuristic 2} using Lemma~\ref{lem: Lp rescaling}. In particular,
\begin{equation*}
    \|Uf_S\|_{L^{q_n}(Z_{\cP(S;\cQ)})} \leq K^{(n+2)/q_n-n/2} \|U \tilde{f}_S\|_{L^{q_n}(Z_{\tilde{\cQ}(S)})} = K^{-1/(n+1)} \|U \tilde{f}_S\|_{L^{q_n}(Z_{\tilde{\cQ}(S)})}
\end{equation*}
for some function $\tilde{f}_S \in L^2(\R^n)$ satisfying
\begin{equation*}
\|\tilde{f}_S \|_{L^2(\R^n)} = \|f_S \|_{L^2(\R^n)} \quad \textrm{and} \quad \supp \mathcal{F}(\tilde{f}_S) \subseteq B^n(0,1).
\end{equation*}

Each localised norm $\|U\tilde{f}_S\|_{L^{q_n}(P)}$ rescales to some $K^{-1/(n+1)}\|U\tilde{f}_S\|_{L^{q_n}(\tilde{Q})}$. Thus,  
\begin{equation*}
     \|U\tilde{f}_S\|_{L^{q_n}(\tilde{Q})} \qquad \textrm{are dyadically constant over $\tilde{Q} \in \tilde{\cQ}(S)$,}
\end{equation*}
as a consequence of \eqref{eq: simplify 1}. Therefore, we may apply the induction hypothesis to conclude that
\begin{equation*}
    \|Uf_S\|_{L^{q_n}(Z_{\cP(S;\cQ)})} \lesssim \bC K^{-1/(n+1)} \Big[  \frac{\Delta_{\alpha}(\tilde{\cQ}(S))}{\#\tilde{\cQ}(S)}\Big]^{1/(n+1)}\tilde{R}^{\alpha/(2(n+1)) + \varepsilon} \|f\|_{L^2(\R^n)}.
\end{equation*}
Since $\tilde{R} = 20 R/K^2$, this becomes
\begin{equation}\label{eq: narrow heuristic 3}
    \|Uf_S\|_{L^{q_n}(Z_{\cP(S;\cQ)})} \lesssim \bC K^{ - 2\varepsilon} \Big[  K^{-1 - \alpha} \frac{\Delta_{\alpha}(\tilde{\cQ}(S))}{\#\tilde{\cQ}(S)}\Big]^{1/(n+1)}R^{\alpha/(2(n+1)) + \varepsilon} \|f\|_{L^2(\R^n)}.
\end{equation} 
The decay factor of $K^{-2\varepsilon}$ arising from the application of the induction hypothesis is crucial for closing the argument. 




\subsubsection*{Summing the estimates} The next step is to sum the localised estimate \eqref{eq: narrow heuristic 3} over all choices of strip $S \in \bbS$. Substituting the estimate \eqref{eq: narrow heuristic 3} into \eqref{eq: narrow heuristic 2.5}, we deduce that
\begin{equation}\label{eq: narrow heuristic 4}
    \|Uf\|_{L^{q_n}(Z_{\cQ})} \lesssim \bC K^{-\varepsilon} M_K(\cQ)^{1/(n+1)} R^{\alpha/(2(n+1)) + \varepsilon} [\#\bbS]^{1/(n+1)}\Big(\sum_{S \in \bbS} \|f_S\|_{L^2(\R^n)}^{q_n}\Big)^{1/q_n}
\end{equation}
where
\begin{equation*}
  M_K(\cQ) :=  K^{-1 - \alpha} \Big[\max_{Q \in \cQ}\frac{\#\bbS(Q)}{\#\bbS} \Big] \Big[  \max_{S \in \bbS} \frac{\Delta_{\alpha}(\tilde{\cQ}(S))}{\#\tilde{\cQ}(S)}\Big].
\end{equation*}

One way to estimate the $\ell^q$ sum on the right-hand side of \eqref{eq: narrow heuristic 4} is to simply use the nesting of $\ell^q$-norms and orthogonality between the wave packets:
\begin{equation}\label{eq: narrow heuristic 5}
\Big(\sum_{S \in \bbS} \|f_S\|_{L^2(\R^n)}^{q_n}\Big)^{1/q_n} \leq \Big(\sum_{S \in \bbS} \|f_S\|_{L^2(\R^n)}^2\Big)^{1/2} \lesssim \|f_S\|_{L^2(\R^{n+1})}.
\end{equation}
However, we shall use a more nuanced bound, improving over the above. The idea is to use a reverse H\"older inequality, similar to that used in \S\ref{sec: broad dom}. Ideally,
\begin{equation}\label{eq: narrow heuristic 6}
    \Big(\sum_{S \in \bbS} \|f_S\|_{L^2(\R^n)}^{q_n}\Big)^{1/q_n} \lesssim [\# \bbS]^{-1/(n+1)} \Big(\sum_{S \in \bbS} \|f_S\|_{L^2(\R^n)}^2\Big)^{1/2} \lesssim [\# \bbS]^{-1/(n+1)} \|f\|_{L^2(\R^n)}.
\end{equation}
which gains an additional factor of $[\# \bbS]^{-1/(n+1)}$ over \eqref{eq: narrow heuristic 5}. This leads us to our next simplifying assumption.\smallskip

\noindent{\textbf{Simplifying Assumption S2.}} 
\begin{equation}\label{eq: simplify 3}
    \|f_S\|_{L^2(\R)} \qquad \textrm{are dyadically constant over $S \in \bbS$.}
\end{equation}
Under this assumption, the desired reverse H\"older inequality in \eqref{eq: narrow heuristic 5} holds.\smallskip

Combining \eqref{eq: narrow heuristic 4} and \eqref{eq: narrow heuristic 6}, we obtain
\begin{equation}\label{eq: narrow heuristic 6.5}
    \|Uf\|_{L^{q_n}(Z_{\cQ})} \lesssim \bC K^{-\varepsilon} M_K(\cQ)^{1/(n+1)} R^{\alpha/(2(n+1)) + \varepsilon} \|f\|_{L^2(\R^n)}.
\end{equation}
It remains to estimate the $M_K(\cQ)$ factor.




\subsubsection*{Closing the induction} In order to conclude the argument, we shall show
\begin{equation}\label{eq: narrow heuristic 7}
   M_K(\cQ)  \lesssim \frac{\Delta_{\alpha}(\cQ)}{\# \cQ}. 
\end{equation}
Indeed, once we have this bound, we can plug it into \eqref{eq: narrow heuristic 6.5} and then choose $\bC$ and $R_0$ appropriately to close the induction and complete the proof. 

The definition of $M_K(\cQ)$ involves two factors:
\begin{equation}\label{eq: narrow heuristic 8}
     \Big[ \max_{Q \in \cQ} \frac{\#\bbS(Q)}{\#\bbS} \Big] \qquad \textrm{and} \qquad K^{-1 - \alpha}  \Big[\max_{S \in \bbS} \frac{\Delta_{\alpha}(\tilde{\cQ}(S))}{\#\tilde{\cQ}(S)} \Big].
\end{equation}
We shall split the proof of \eqref{eq: narrow heuristic 7} into 3 steps: the first two steps shall treat the first factor in \eqref{eq: narrow heuristic 8} and the remaining step treats the remaining.\medskip

\noindent \textbf{1. Multiplicity bounds.} Recall from the definitions that $\#\bbS(Q)/\#\bbS$ is the proportion of all strips $S \in \bbS$ which pass through $Q$. Hence, we refer to this quantity as the \textit{multiplicity} of $Q$. On the other hand, if we define
\begin{equation*}
    \cQ(S) := \{Q \in \cQ : S \cap Q \neq \emptyset\}, \qquad S \in \bbS,
\end{equation*}
then $\#\cQ(S)/\#\cQ$ is the proportion of all cubes $Q \in \cQ$ which lie in a fixed strip $S$. We refer to this quantity as the \textit{multiplicity} of $S$. Ideally, we would like to show
\begin{equation}\label{eq: multiplicity comparison}
       \max_{Q \in \cQ} \frac{\#\bbS(Q)}{\#\bbS} \lesssim \max_{S \in \bbS} \frac{\#\cQ(S)}{\#\cQ};
\end{equation}
in other words, if there exists a high multiplicity cube, then there must exist a high multiplicity strip. 

\begin{exem}\label{ex: cube counting} Without further hypotheses, it is easy to see \eqref{eq: multiplicity comparison} may fail. For instance, suppose there are $M$ strips in $\bbS$, there are $M+1$ cubes in $\cQ$. We may arrange things so that:
\begin{itemize}
    \item There exists precisely one cube $Q_0 \in \cQ$ which lies in every strip: $\#\bbS(Q_0) = M$;
    \item Every other cube in $\cQ$ lies in precisely one strip: $\#\bbS(Q) = 1$ for all $Q \in \cQ \setminus\{Q_0\}$;
    \item Every strip contains $2$ cubes: $\#\cQ(S) = 2$ for all $S \in \bbS$. 
\end{itemize} 
Then 
\begin{equation*}
   \max_{Q \in \cQ} \frac{\#\bbS(Q)}{\#\bbS} = \frac{\#\bbS(Q_0)}{\#\bbS} = 1 \qquad \textrm{and} \qquad \max_{S \in \bbS} \frac{\#\cQ(S)}{\#\cQ} = \frac{2}{M+1}. 
\end{equation*}
Thus, if $M \gg 1$, then this violates \eqref{eq: multiplicity comparison}. The idea here, however, is that the violation arises from the single `outlier' cube $Q_0$; if we throw away such outliers, and focus only on `typical' cubes, then we can hope for \eqref{eq: multiplicity comparison} to hold. 
\end{exem}

By double-counting we always have a trivial inequality
\begin{equation}\label{eq: narrow heuristic 9}
    \min_{Q \in \cQ} \frac{\#\bbS(Q)}{\#\bbS} \leq \max_{S \in \bbS} \frac{\#\cQ(S)}{\#\cQ}.
\end{equation}
Indeed, \eqref{eq: narrow heuristic 9} immediately follows once we note
\begin{equation*}
 \sum_{Q \in \cQ} \#\bbS(Q) = \#\{ (S, Q) : S \in \bbS, Q \in \cQ(S)\} =  \sum_{S \in \bbS} \#\cQ(S).
\end{equation*}
To exploit the bound \eqref{eq: narrow heuristic 9}, we make a further assumption on the multiplicities of the cubes.\smallskip

\noindent{\textbf{Simplifying Assumption Q.}} 
\begin{equation}\label{eq: simplify 5}
    \#\bbS(Q) \qquad \textrm{are dyadically constant over $Q \in \cQ$.}
\end{equation}

This assumption ensures there are no `outlier' cubes as in Example~\ref{ex: cube counting}. In particular, using \eqref{eq: simplify 5}, we may upgrade \eqref{eq: narrow heuristic 9} to the desired bound \eqref{eq: multiplicity comparison}. 
 This helps us to control the first factor in \eqref{eq: narrow heuristic 8}. \medskip
 
 \noindent \textbf{2. From strips to parallelepipeds.} The next step is to relate $\#\cQ(S)$ to $\# \tilde{\cQ}(S)$. We begin by recalling some of the definitions. For each strip $S \in \bbS$ there is an associated family of parallelepipeds $\cP(S;\cQ)$. Furthermore, the family of cubes $\tilde{\cQ}(S)$ is obtained by transforming the $P \in \cP(S;\cQ)$ under $\cA_S$ (see \eqref{eq: resc cubes}). In particular,
 \begin{equation}\label{eq: same size}
  \#\cP(S;\cQ) =  \#\tilde{\cQ}(S).
 \end{equation}
 Thus, our task here is to relate $\#\cQ(S)$ to $\# \cP(S;\cQ)$.  

 Let $\cP(\cQ)$ denote the union of the $\cP(S;\cQ)$ over all $S \in \bbS$ and let
\begin{equation*}
\cQ(P) := \{Q \in \cQ : P \cap Q \neq \emptyset\} \qquad \textrm{for all $P \in \cP(\cQ)$.}
\end{equation*}
It follows from the definitions that
\begin{equation*}
    \cQ(S) = \bigcup_{P \in \cP(S;\cQ)} \cQ(P).
\end{equation*}
We can therefore compare the multiplicity of a strip to the multiplicities of the parallelepipeds via the simple inequality
\begin{equation*}
  \#\cQ(S) \leq \sum_{P \in \cP(S;\cQ)} \#\cQ(P) \leq  [\max_{P \in \cP(S;\cQ)} \#\cQ(P)] [\#\cP(S;\cQ)] .
\end{equation*}
Thus, in view of \eqref{eq: multiplicity comparison} and \eqref{eq: same size}, we have
\begin{equation*}
    \max_{Q \in \cQ} \frac{\#\bbS(Q)}{\#\bbS} \leq  [\max_{P \in \cP(\cQ)} \#\cQ(P)][ \max_{S \in \bbS} \#\tilde{\cQ}(S)].
\end{equation*}

The maxima on the right-hand side of the above inequality are awkward to bound. However, the situation is improved if we introduce the following additional assumptions.\smallskip

\noindent{\textbf{Simplifying Assumption P2.}} 
\begin{equation}\label{eq: simplify 2}
    \#\cQ(P) \qquad \textrm{are dyadically constant over all $P \in \cP(\cQ)$.}
\end{equation}

\noindent{\textbf{Simplifying Assumption S2.}}
\begin{equation*}
    \#\cP(S; \cQ) \qquad \textrm{are dyadically constant over $S \in \bbS$.}
\end{equation*}

These are (thankfully!) our final simplifying assumptions. By \eqref{eq: simplify 2} and \eqref{eq: simplify 3}, it follows that
\begin{equation}\label{eq: narrow heuristic 11}
    \Big[ \max_{Q \in \cQ} \frac{\#\bbS(Q)}{\#\bbS} \Big]   \lesssim \frac{1}{\#\cQ}[\min_{P \in\cP(\cQ)} \#\cQ(P)] [ \min_{S \in \bbS} \#\tilde{\cQ}(S)].
\end{equation}

\medskip
 
 \noindent \textbf{3. Comparing densities.} The final step is to relate the densities $\Delta_{\alpha}(\tilde{\cQ}(S))$ and $\Delta_{\alpha}(\cQ)$. This is achieved via the following simple lemma.

\begin{lemm}\label{lem: relating scales} With the above setup, for any $S \in \bbS$ we have
    \begin{equation}\label{eq: narrow heuristic 12}
    [\min_{P \in \cP(\cQ)} \#\cQ(P)] \Delta_{\alpha}(\tilde{\cQ}(S)) \lesssim K^{1+\alpha} \Delta_{\alpha}(\cQ).
\end{equation}
\end{lemm}

\begin{proof} Let $\tilde{B} \subseteq B^{n-1}(0, \tilde{R})$ be a ball of radius $r := \rad(\tilde{B}) \geq \tilde{K}^2$. Note that
\begin{equation*}
     \#\{\tilde{Q} \in \tilde{\cQ}(S) : \tilde{Q} \subseteq \tilde{B} \} = \#\{P \in \cP(S; \cQ) : P \subseteq \cA^{-1}_S(\tilde{B}) \},
\end{equation*}
where we have used the definition of the family of cubes $\tilde{\cQ}(S)$ from \eqref{eq: resc cubes}. Letting $T := \cA^{-1}_S(\tilde{B})$, it follows that
\begin{equation*}\label{eq: comp scales 1}
  [\min_{P \in \cP(\cQ)} \#\cQ(P)]  [\#\{\tilde{Q} \in \tilde{\cQ}(S) : \tilde{Q} \subseteq \tilde{B} \}] \leq \#\{Q \in \cQ : Q \subseteq T \}.
\end{equation*}

Since $T$ is a strip of dimension $Kr \times \cdots \times K r \times K^2 r$, it can be covered by $O(K)$ balls $B \subseteq B^{n+1}(0,R)$ of radius $Kr$. Furthermore, this cover can be chosen such that if $Q \in \cQ$ satisfies $Q \subseteq T$, then $Q \subseteq B$ for some choice of ball in the cover. Thus, 
\begin{equation}\label{eq: comp scales 2}
  \#\{Q \in \cQ : Q \subseteq T \} \lesssim K [\max_{\substack{ B \subseteq B^{n+1}(0,R) \\ \rad(B) = Kr }} \#\{Q \in \cQ : Q \subseteq B \}] \leq K^{1+\alpha} r^{\alpha} \Delta_{\alpha}(\cQ). 
\end{equation}

Combining \eqref{eq: comp scales 1} and \eqref{eq: comp scales 2}, we deduce that 
\begin{equation*}
      [\min_{P \in \cP(\cQ)} \#\cQ(P)]  \Delta_{\alpha}(\tilde{\cQ}; \tilde{B}) \lesssim K^{1+\alpha}  \Delta_{\alpha}(\cQ)
\end{equation*}
and the desired result follows by taking a supremum over all choices of $\tilde{B}$. 
\end{proof}

Combining \eqref{eq: narrow heuristic 11} and \eqref{eq: narrow heuristic 12}, we therefore have
\begin{equation*}
\Big[ \max_{Q \in \cQ} \frac{\#\bbS(Q)}{\#\bbS} \Big]  \Big[ \max_{S \in \bbS} \frac{\Delta_{\alpha}(\tilde{\cQ}(S))}{\#\tilde{\cQ}(S)} \Big]  \lesssim K^{1+\alpha}\frac{\Delta_{\alpha}(\cQ)}{\#\cQ}.
\end{equation*}
Recalling the definition of $M_k(\cQ)$, this immediately implies the desired bound \eqref{eq: narrow heuristic 7}. Consequently, we have the narrow estimate
\begin{equation*}
    \|Uf\|_{L^{q_n}(Z_{\cQ})} \lesssim \bC K^{-\varepsilon} \Big[\frac{\Delta_{\alpha}(\cQ)}{\#\cQ}\Big]^{1/(n+1)} R^{\alpha/(2(n+1)) + \varepsilon} \|f\|_{L^2(\R^n)},
\end{equation*}
at least under the Simplifying Assumptions P1, P2, S1, S2 and Q introduced above.




\subsection{Analysis of the narrow case: technical details}\label{sec: narrow technical}

We now reexamine the argument from \S\ref{sec: narrow non-technical}, including technical details. The main additional ingredient is a sequence of pigeonholing steps used to rigorously justify the various simplifying assumptions used above. 

Define the scales $\tilde{R}$ and $\tilde{K}$ as in \eqref{eq: new scales}. For each $S \in \bbS$, we decompose each enlarged strip $\bar{S}$ into parallel parallelepipeds $\cP(S)$ as in \eqref{eq: parallelepipeds}. In particular, each $P \in \cP(S)$ is aligned parallel to $S$ and has dimensions $K\tilde{K}^2 \times \cdots \times K \tilde{K}^2 \times K^2 \tilde{K}^2$. By Lemma~\ref{lem: pseudo loc}, we have the pointwise bound
 \begin{equation}\label{eq: S pointwise dom}
     | Uf_S(z)| \lesssim \sum_{P \in \cP(S)} |Uf_S(z)| \chi_P(z) + R^{-100n} \|f\|_{L^2(\R^n)}.
 \end{equation}

We let $\cP_{\mathrm{all}}$ denote the disjoint union of the sets $\cP(S)$ over all $S \in \bbS$. Thus, given a parallelepiped $P \in \cP_{\mathrm{all}}$, there exists a unique element $S \in \bbS$ such that $P \in \cP(S)$, which we denote by $S(P)$.




\subsubsection*{Pigeonholing} We first pigeonhole in the parallelepipeds $P \in \cP_{\mathrm{all}}$. To this end, define 
\begin{equation*}
    \cP_{\mathrm{tiny}} := \big\{ P \in \cP : \|Uf_{S(P)}\|_{L^{q_n}(P)} \leq R^{-100n} \|f\|_{L^2(\R^n)} \big\}.
\end{equation*}
Parallelepipeds $P \in \cP_{\mathrm{tiny}}$ are negligible for our purpose: for a formal interpretation of this see \eqref{eq: P tiny neg} below. On the other hand, if $P \in \cP_{\mathrm{all}} \setminus \cP_{\mathrm{tiny}}$, then a trivial estimate using the Cauchy--Schwarz inequality and Plancherel's theorem shows that
\begin{equation*}
   R^{-100n}  \|f\|_{L^2(\R^n)} \leq \|Uf_{S(P)}\|_{L^{q_n}(P)} \lesssim |P|^{1/q} \|f\|_{L^2(\R^n)} \lesssim R \|f\|_{L^2(\R^n)}.
\end{equation*}
This will allow us to dyadically pigeonhole in $\|Uf_{S(P)}\|_{L^{q_n}(P)}$. We shall also pigeonhole in the cardinality of the sets 
\begin{equation*}
    \cQ(P) := \big\{Q \in \cQ : P \cap Q^{(\delta)} \neq \emptyset \big\}, \quad \textrm{where} \quad Q^{(\delta)} := K^{\delta} \cdot Q \quad \textrm{for $Q \in \cQ$;}
\end{equation*}
that is, the number of $K^2$-cubes $Q \in \cQ$ which lie in the vicinity of a given $P$.\medskip

\noindent{\textbf{Pigeonholing the parallelepipeds.}} The family of parallelepipeds $\cP_{\mathrm{all}} \setminus \cP_{\mathrm{tiny}}$ can be written as a disjoint union $\cP_{\mathrm{all}} \setminus \cP_{\mathrm{tiny}} = \cP_1 \cup \cdots \cup \cP_I$ where 
\begin{equation}\label{eq: dyadic const P1}
    \|Uf_{S(P)}\|_{L^{q_n}(P)} \qquad \textrm{are dyadically constant over $P \in \cP_i$}
\end{equation}
and 
\begin{equation*}
    \#\cQ(P) \qquad \textrm{are dyadically constant over $P \in \cP_i$}
\end{equation*}
for each $1 \leq i \leq I$ and $I \lessapprox 1$. This corresponds to Simplifying Assumptions P1 and P2 in \S\ref{sec: narrow non-technical}. \medskip

We now turn to pigeonholing the strips $S \in \bbS$. To this end, first define
\begin{equation*}
    \bbS_{\mathrm{tiny}} := \big\{ S \in \bbS : \|f_S\|_{L^2(\R^n)} \leq R^{-100n} \|f\|_{L^2(\R^n)} \big\}.
\end{equation*}
The strips $S \in \bbS_{\mathrm{tiny}}$ are negligible for our purpose. More precisely, let $\cQ_{\mathrm{tiny}}$ denote the set of cubes $Q \in \cQ$ such that 
\begin{equation*}
  \Big(\sum_{S \in \bbS} \|Uf_S\|_{L^{q_n}(w_Q)}^2\Big)^{1/2}  \leq 2 \Big(\sum_{S \in \bbS_{\mathrm{tiny}}} \|Uf_S\|_{L^{q_n}(w_Q)}^2\Big)^{1/2}. 
\end{equation*}
If $\#\cQ_{\mathrm{tiny}} \geq \# \cQ/2$, then \eqref{eq: narrow 2} and the dyadically constant hypothesis \eqref{eq: dyadic const hyp} immediately yield a very favourable estimate. Henceforth, we assume $\cQ_{\mathrm{main}} := \cQ \setminus \cQ_{\mathrm{tiny}}$ satisfies $\#\cQ_{\mathrm{main}} \geq \# \cQ/2$.

If $S \in \bbS \setminus \bbS_{\mathrm{tiny}}$, then the orthogonality properties of the wave packets (see Lemma~\ref{lem: wp dec}) imply that
\begin{equation*}
 R^{-100n} \|f\|_{L^2(\R^n)} \leq   \|f_S\|_{L^2(\R^n)} \lesssim \|f\|_{L^2(\R^n)}.
\end{equation*}
This will allow us to dyadically pigeonhole in $\|f_S\|_{L^2(\R^n)}$. We shall also pigeonhole in the cardinalities of the sets 
\begin{equation*}
  \cP_i(S) := \cP(S) \cap \cP_i \qquad \textrm{for $S \in \bbS$ and $1 \leq i \leq I$.}  
\end{equation*}

\noindent{\textbf{Pigeonholing the strips.}} For each $1 \leq i \leq I$, the family of strips $\bbS \setminus \bbS_{\mathrm{tiny}}$ can be written as a disjoint union $\bbS \setminus \bbS_{\mathrm{tiny}} = \bbS_{i,1} \cup \cdots \cup \bbS_{i, J_i}$ where
\begin{equation*}
    \|f_S\|_{L^2(\R^n)} \qquad \textrm{are dyadically constant over $S \in \bbS_{i,j}$}
\end{equation*}
and
\begin{equation*}
    \#\cP_i(S) \qquad \textrm{are dyadically constant over $S \in \bbS_{i,j}$}
\end{equation*}
and $J_i \lessapprox 1$ for all $1 \leq i \leq I$, $1 \leq j \leq J_i$. This corresponds to Simplifying Assumptions~S1 and~S2 in~\S\ref{sec: narrow non-technical}. \medskip

Finally, we pigeonhole in the cubes $Q \in \cQ$. This step is a little more involved and we require a number of preliminary observations. 

Given $Q \in \cQ$ and any set $\cP \subseteq \cP_{\mathrm{all}}$, define 
\begin{equation*}
    w_Q|_{\cP} := w_Q \cdot \sum_{P \in {\cP}}\chi_P.
\end{equation*}
Fix $Q \in \cQ_{\mathrm{main}}$. By the pointwise bound \eqref{eq: S pointwise dom} and nesting of $\ell^q$ norms, we have
 \begin{equation}\label{eq: P tiny neg}
     \|Uf_S\|_{L^{q_n}(w_Q)} \lesssim \Big(\sum_{i = 1}^I  \|Uf_S\|_{L^{q_n}(w_Q|_{\cP_i(S)})}^2 \Big)^{1/2} + R^{-10n} \|f\|_{L^2(\R^n)}. 
 \end{equation}
Here the contribution from the parallelepipeds  $P \in \cP_{\mathrm{tiny}}$ is absorbed into the rapidly decaying term. Taking the $\ell^2$-norm over $S \in \bbS$, we have
\begin{equation*}
    \Big(\sum_{S \in \bbS} \|Uf_S\|_{L^{q_n}(w_Q)}^2\Big)^{1/2} \lesssim \Big(\sum_{S \in \bbS\setminus \bbS_{\mathrm{tiny}}} \sum_{i = 1}^I  \|Uf_S\|_{L^{q_n}(w_Q|_{\cP_i(S)})}^2 \Big)^{1/2} + R^{-5n} \|f\|_{L^2(\R^n)};
\end{equation*}
note that, since $Q \in \cQ_{\mathrm{main}}$, the contribution from the strips $S \in \bbS_{\mathrm{tiny}}$ is negligible. By reordering the right-hand sum in the above expression, we have
\begin{equation*}
    \sum_{S \in \bbS\setminus \bbS_{\mathrm{tiny}}} \sum_{i = 1}^I \|Uf_S\|_{L^{q_n}(w_Q|_{\cP_i(S)})}^2 = \sum_{i = 1}^I \sum_{j=1}^{J_i}\sum_{S \in \bbS_{i,j}}  \|Uf_S\|_{L^{q_n}(w_Q|_{\cP_i(S)})}^2.
\end{equation*}
Finally, we combine the above observations with \eqref{eq: narrow 2} to deduce that
\begin{equation}\label{eq: prelim Q pigeonhole}
    \|Uf\|_{L^{q_n}(Q)} \lesssim_{\varepsilon} K^{\varepsilon} \Big(\sum_{i = 1}^I \sum_{j=1}^{J_i}\sum_{S \in \bbS_{i,j}}  \|Uf_S\|_{L^{q_n}(w_Q|_{\cP_i(S)})}^2\Big)^{1/2} +  R^{-5n} \|f\|_{L^2(\R^n)}.
\end{equation}

We now turn to pigeonholing the cubes $Q \in \cQ$, which involves a two-step process.\medskip

\noindent{\textbf{Pigeonholing the cubes.}} Let $Q \in \cQ_{\mathrm{main}}$. Applying pigeonholing to \eqref{eq: prelim Q pigeonhole}, there exists some index pair $(i_Q,j_Q)$ such that $\bbS_Q := \bbS_{i_Q, j_Q}$ and $\cP_Q(S) := \cP_{i_Q}(S)$ satisfy
\begin{equation*}
    \|Uf\|_{L^{q_n}(Q)}  \lessapprox  K^{\varepsilon} \Big(\sum_{S \in \bbS_Q}  \|Uf_S\|_{L^{q_n}(w_Q|_{\cP_Q(S)})}^2 \Big)^{1/2} +  R^{-5n} \|f\|_{L^2(\R^n)}.
\end{equation*}
Furthermore, by pigeonholing, there exists some refinement $\cQ_0' \subseteq \cQ_{\mathrm{main}}$ and index $(i_0, j_0)$ such that $\bbS' := \bbS_{i_0, j_0}$ and $\cP' := \cP_{i_0}$ satisfy 
\begin{equation*}
    \bbS_Q = \bbS' \quad \textrm{for all $Q \in \cQ_0'$} \quad \textrm{and} \quad \#\cQ_0' \gtrapprox \#\cQ 
\end{equation*}

Given $Q \in \cQ$, define 
\begin{equation*}
    \bbS'(Q) := \big\{ S \in \bbS' : \bar{S} \cap Q^{(\delta)} \neq \emptyset \big\}
\end{equation*}
where, as above, $Q^{(\delta)} := K^{\delta} \cdot Q$ for $Q \in \cQ$. By a second application of the pigeonhole principle, we can find a further refinement $\cQ' \subseteq \cQ_{0}'$ such that 
\begin{equation}\label{eq: dyadic const Q}
    \#\bbS'(Q) \quad \textrm{are dyadically constant over $Q \in \cQ'$.} 
\end{equation}
This corresponds to Simplifying Assumption Q in \S\ref{sec: narrow non-technical}. \medskip

If $Q \in \cQ'$, then it follows from the above construction and definitions that
\begin{equation*}
    \|Uf\|_{L^{q_n}(Q)}  \lessapprox K^{\varepsilon} \Big(\sum_{S \in \bbS'(Q)} \|Uf_S\|_{L^{q_n}(w_Q|_{\cP'(S)})}^2 \Big)^{1/2} +  R^{-5n} \|f\|_{L^2(\R^n)}.
\end{equation*}
In order to sum in $Q$, we apply H\"older's inequality to pass from an $\ell^2$ to an $\ell^q$ norm. Furthermore, since the weight $w_Q$ is rapidly decaying away from $Q$, we may pass from $w_Q|_{\cP'(S)}$ to the weight $w_Q|_{\cP'(S;\cQ')}$ supported on the parallelepipeds
\begin{equation*}
    \cP'(S;\cQ') := \big\{ P \in \cP'(S) : P \cap Q^{(\delta)} \neq \emptyset \textrm{ for some } Q \in \cQ' \big\}.
\end{equation*}
In particular,
\begin{equation*}
    \|Uf\|_{L^{q_n}(Q)}  \lessapprox K^{\varepsilon}
    [\#\bbS'(Q)]^{1/(n+1)} \Big(\sum_{S \in \bbS'(Q)} \|Uf_S\|_{L^{q_n}(w_Q|_{\cP'(S;\cQ')})}^{q_n} \Big)^{1/q_n} +  R^{-5n} \|f\|_{L^2(\R^n)}.
\end{equation*}
Taking $q_n$-powers and summing over all the cubes in the refined collection,
\begin{equation}\label{eq: narrow rigour 1}
    \|Uf\|_{L^{q_n}(Z_{\cQ})} \lessapprox K^{\varepsilon} [\min_{Q \in \cQ'} \#\bbS'(Q)]^{1/(n+1)} \Big(\sum_{S \in \bbS'} \|Uf_S\|_{L^{q_n}(Z_{\cP'(S;\cQ')})}^{q_n} \Big)^{1/q_n} +   \|f\|_{L^2(\R^n)}
\end{equation}
where $ Z_{\cP'(S;\cQ')}$ denotes the union of the $P \in \cP'(S;\cQ')$. Here we have used the dyadically constant hypothesis \eqref{eq: dyadic const hyp} and \eqref{eq: dyadic const Q}.




\subsubsection*{Parabolic rescaling and the induction hypothesis} Fix $S \in \bbS'$. We applying parabolic rescaling from Lemma~\ref{lem: Lp rescaling} to deduce that
\begin{equation*}
    \|Uf_S\|_{L^{q_n}(Z_{\cP'(S;\cQ')})} \leq K^{-1/(n+1)} \|U \tilde{f}_S\|_{L^{q_n}(Z_{\tilde{\cQ}(S)})}
\end{equation*}
where $\tilde{f}_S \in L^2(\R^n)$ satisfies 
\begin{equation*}
    \|\tilde{f}_S\|_{L^2(\R^n)} = \|f_S\|_{L^2(\R^n)}  \quad \textrm{and} \quad \supp \mathcal{F}(\tilde{f}_S) \subseteq B^n(0,1).
\end{equation*}
Here $\tilde{\cQ}(S)$ is a collection of lattice $\tilde{K}^2$-cubes contained in $B^{n+1}(0, \tilde{R})$. Furthermore, for each $\tilde{Q} \in \tilde{\cQ}(S)$ there exists some $P \in \cP'(S;\cQ')$ such that
\begin{equation*}
    \|Uf_S\|_{L^{q_n}(P)} = K^{-1/(n+1)} \|U \tilde{f}_S\|_{L^{q_n}(\tilde{Q})};
\end{equation*}
it therefore follows from \eqref{eq: dyadic const P1} that
\begin{equation*}
    \|U\tilde{f}_S\|_{L^{q_n}(\tilde{Q})} \qquad \textrm{are dyadically constant over $\tilde{Q} \in \tilde{\cQ}(S)$.}
\end{equation*}

In light of the above, we may apply the induction hypothesis to conclude that
\begin{equation}\label{eq: narrow rigour 2}
  \|Uf_S\|_{L^{q_n}(Z_{\cP'(S;\cQ')})} \lesssim \bC K^{- 2\varepsilon}\Big[K^{-1-\alpha}\frac{\Delta_{\alpha}(\tilde{\cQ}(S))}{\#\tilde{\cQ}(S)}\Big]^{1/(n+1)} R^{\alpha/(2(n+1)) + \varepsilon}\|f_S\|_{L^2(\R^n)}.
\end{equation}
This is, of course, the analogue of the estimate \eqref{eq: narrow heuristic 3} from the non-technical argument in \S\ref{sec: narrow non-technical}. The key difference here is that the family of cubes $\tilde{\cQ}(S)$ is formed from the refined set of parallelepipeds $\cP'(S;\cQ')$; this will allow us to (rigorously) exploit the various dyadically constancy properties.




\subsubsection*{Summing the estimates} As in \S\ref{sec: narrow non-technical}, the next step is to sum the localised estimate \eqref{eq: narrow rigour 2} over all choices of strip $S \in \bbS'$. Substituting \eqref{eq: narrow rigour 2} into \eqref{eq: narrow rigour 1}, we deduce that
\begin{equation}\label{eq: narrow rigour 3}
    \|Uf\|_{L^{q_n}(Z_{\cQ})} \lessapprox \bC K^{-\varepsilon} M_K'(\cQ')^{1/(n+1)} R^{\alpha/(2(n+1)) + \varepsilon} [\#\bbS'\,]^{1/(n+1)}\Big(\sum_{S \in \bbS'} \|f_S\|_{L^2(\R^n)}^{q_n}\Big)^{1/q_n},
\end{equation}
where 
\begin{equation}\label{eq: narrow rigour 4}
  M_K'(\cQ') :=  K^{-1 - \alpha} \Big[\max_{Q \in \cQ'}\frac{\#\bbS'(Q)}{\#\bbS'} \Big] \Big[  \max_{S \in \bbS'} \frac{\Delta_{\alpha}(\tilde{\cQ}(S))}{\#\tilde{\cQ}(S)}\Big].
\end{equation}
is defined in a similar manner to the corresponding quantity in \S\ref{sec: narrow non-technical}.

To estimate the $\ell^q$ sum on the right-hand side of \eqref{eq: narrow rigour 3}, we use a reverse H\"older inequality as in \eqref{eq: narrow heuristic 6}. Indeed, reverse H\"older can now be applied rigorously without further assumptions, since the norms $\|f_S\|_{L^2(\R^n)}$ for $S \in \bbS'$ are dyadically constant by construction. Consequently,
\begin{equation*}
    \|Uf\|_{L^{q_n}(Z_{\cQ})} \lesssim \bC K^{-\varepsilon} M_K'(\cQ')^{1/(n+1)} R^{\alpha/(2(n+1)) + \varepsilon} \|f\|_{L^2(\R^n)}.
\end{equation*}




\subsubsection*{Multiplicity bounds} The final step of the narrow analysis is to control the constant $M_K'(\cQ')$ and, in particular, show that 
\begin{equation}\label{eq: narrow rigour 5}
    M_K'(\cQ') \lessapprox  K^{O(\delta)} \frac{\Delta_{\alpha}(\cQ)}{\# \cQ}. 
\end{equation}
Recall from our initial pigeonholing:
\begin{center}
    \begin{tabular}{l c l}
 $\#\bbS'(Q)$ &  $\qquad$  & are dyadically constant over $Q \in \cQ'$; \\
 $\#\cP'(S; \cQ')$ &   & are dyadically constant over $S \in \bbS'$; \\
 $\#\cQ'(P)$ &    & are dyadically constant over $P \in \cP'$.
\end{tabular}
\end{center}

Arguing exactly as in \S\ref{sec: narrow non-technical}, we therefore deduce that
\begin{equation}\label{eq: narrow rigour 6}
    \Big[ \max_{Q \in \cQ'} \frac{\#\bbS'(Q)}{\#\bbS'} \Big]  \Big[ \max_{S \in \bbS'} \frac{\Delta_{\alpha}(\tilde{\cQ}(S))}{\#\tilde{\cQ}(S)} \Big]  \lesssim \frac{1}{\#\cQ'}[\min_{P \in\cP'(\cQ')} \#\cQ'(P)] [ \max_{S \in \bbS'} \Delta_{\alpha}(\tilde{\cQ}(S))].
\end{equation}
As before, we can use Lemma~\ref{lem: relating scales} to compare the densities, giving
\begin{equation}\label{eq: narrow rigour 7}
    [\min_{P \in \cP'(\cQ')} \#\cQ'(P)] [ \max_{S \in \bbS'} \Delta_{\alpha}(\tilde{\cQ}(S))] \lesssim K^{O(\delta)}K^{1+\alpha} \Delta_{\alpha}(\cQ').
\end{equation}
Note that here we lose a factor of $K^{O(\delta)}$ due to the fact that the set of parallelepipeds $\cP(S;\cQ)$ is defined with the enlarged cubes $Q^{(\delta)}$. 

Combining \eqref{eq: narrow rigour 6} and \eqref{eq: narrow rigour 7}, we have 
\begin{equation*}
    \Big[ \max_{Q \in \cQ'} \frac{\#\bbS'(Q)}{\#\bbS'} \Big]  \Big[ \max_{S \in \bbS'} \frac{\Delta_{\alpha}(\tilde{\cQ}(S))}{\#\tilde{\cQ}(S)} \Big]  \lessapprox K^{O(\delta)}K^{1+\alpha}\frac{ \Delta_{\alpha}(\cQ)}{\#\cQ};
\end{equation*}
here we have also used the fact that $\#\cQ' \gtrapprox \# \cQ$ and $\Delta_{\alpha}(\cQ') \leq \Delta_{\alpha}(\cQ)$. Recalling the definition of $M_K'(\cQ')$ from \eqref{eq: narrow rigour 4}, this immediately implies the desired bound \eqref{eq: narrow rigour 5}. Consequently, we have the narrow estimate
\begin{equation}\label{eq: narrow dom final est}
    \|Uf\|_{L^{q_n}(Z_{\cQ})} \lesssim_{\varepsilon} \bC K^{-\varepsilon/2} \Big[\frac{\Delta_{\alpha}(\cQ)}{\#\cQ}\Big]^{1/(n+1)} R^{\alpha/(2(n+1)) + \varepsilon} \|f\|_{L^2(\R^n)},
\end{equation}
which essentially agrees with the bound derived in \S\ref{sec: narrow non-technical}.




\subsection{Concluding the argument}

To conclude the proof, it remains to collect together the estimates proved above and show that they can be used to close the induction. Recall:
\begin{itemize}
    \item In \S\ref{sec: broad dom}, we showed that \eqref{eq: broad dom final est} holds in the broad-dominant case;
    \item In \S\S\ref{sec: narrow initial}-\ref{sec: narrow technical}, we showed that \eqref{eq: narrow dom final est} holds in the narrow-dominant case.
\end{itemize}

Combining \eqref{eq: broad dom final est} and \eqref{eq: narrow dom final est}, we see that there exists a constant $C_{\varepsilon} \geq 1$ such that the bound
\begin{equation}\label{eq: conclusion 1}
   \|Uf\|_{L^{q_n}(Z_{\cQ})} \leq C_{\varepsilon} \big( \bC K^{-\varepsilon/2}  +  1 \big) K^{-\varepsilon/2} \Big[\frac{\Delta_{\alpha}(\cQ)}{\#\cQ}\Big]^{1/(n+1)}  R^{\alpha/(2(n+1)) + \varepsilon}\|f\|_{L^2(\R^n)}.
\end{equation}
holds in either case. 

The estimate \eqref{eq: conclusion 1} involves two free parameters:
\begin{itemize}
    \item We are free to choose the cutoff $R_0$ for the base case;
    \item We are free to choose the constant $\bC$ appearing in the induction hypothesis, provided our choice is independent of $R$.\footnote{On a minor technical note, $\bC$ should be chosen large depending on $R_0$ to ensure the base case holds. In this sense, the two parameters are not entirely independent of one another.} 
\end{itemize}
At this point, we fine tune these parameters to ensure that the induction closes. Recalling that $K = R^{\delta}$, we choose $R_0$ large enough so that $C_{\varepsilon}K^{-\varepsilon/2} \leq 1/2$ whenever $R \geq R_0$ and take $\bC = 2 C_{\varepsilon}$. With these parameters, \eqref{eq: conclusion 1}  implies 
\begin{equation*}
   \|Uf\|_{L^{q_n}(Z_{\cQ})} \leq \bC \Big[\frac{\Delta_{\alpha}(\cQ)}{\#\cQ}\Big]^{1/(n+1)}  R^{\alpha/(2(n+1)) + \varepsilon}\|f\|_{L^2(\R^n)},
\end{equation*}
which is precisely what is required to close the induction.

\printbibliography

\end{document}